\documentclass[12pt]{article}

\usepackage[utf8]{inputenc}

\usepackage{amsfonts,graphicx,amsmath,amssymb,amsthm,geometry,
hyperref,color,nicefrac,enumerate,bbm,verbatim,lmodern, url}
\usepackage{xcolor} 

\usepackage[draft]{todonotes}   

\newcommand{\smallsum}{\textstyle\sum}

\allowdisplaybreaks[1]

\newtheorem{lemma}{Lemma}[section]
\newtheorem{corollary}[lemma]{Corollary}

\newtheorem{theorem}[lemma]{Theorem}

\newtheorem{setting}{Setting}[section]

\renewcommand{\O}{ { \bf O } }

\renewcommand{\L}{\mathcal{L}}

\providecommand{\N}{{\ensuremath{\mathbbm{N}}}}
\providecommand{\R}{{\ensuremath{\mathbbm{R}}}}
\providecommand{\E}{{\ensuremath{\mathbb{E}}}}
\providecommand{\B}{{\ensuremath{\mathcal{B}}}}
\providecommand{\F}{{\ensuremath{\mathcal{F}}}}
\providecommand{\f}{{\ensuremath{\mathbb{F}}}}

\renewcommand{\H}{{\ensuremath{\mathbb{H}}}}
\renewcommand{\P}{{\ensuremath{\mathbb{P}}}}
\providecommand{\1}{{\ensuremath{\mathbbm{1}}}}

\providecommand{\values}{{\ensuremath{\mathfrak{v}}}}
\providecommand{\N}{{\ensuremath{\mathbbm{N}}}}
\providecommand{\R}{{\ensuremath{\mathbbm{R}}}}

\providecommand{\E}{{\ensuremath{\mathbb{E}}}}
\renewcommand{\P}{{\ensuremath{\mathbb{P}}}}
\providecommand{\1}{{\ensuremath{\mathbbm{1}}}}
\providecommand{\HS}{{\ensuremath{\textup{HS}}}}

\allowdisplaybreaks 
\begin{document}

\title{Exponential moment bounds and strong\\ convergence
	 rates for 
	 tamed-truncated 
numerical\\
 approximations of stochastic convolutions}
\author{Arnulf Jentzen$^{1} $, Felix Lindner$^{2} $, and Primo\v{z} Pu\v{s}nik$^{3} $
	\bigskip 
	\\
	\small{$^1$ Seminar for Applied Mathematics, Department of Mathematics,}
	\\
	\small{ETH Zurich, Switzerland, 
		e-mail: arnulf.jentzen@sam.math.ethz.ch} 
	\smallskip
	\\
	\small{$^2$  Institute of Mathematics, Faculty of Mathematics and Natural Sciences,}
	\\
	\small{University of Kassel, Germany,
		e-mail: lindner@mathematik.uni-kassel.de}
	\smallskip
	\\
	\small{$^3$ Seminar for Applied Mathematics, Department of Mathematics,}
	\\
	\small{ETH Zurich, Switzerland,
		e-mail: primoz.pusnik@sam.math.ethz.ch}
}


\maketitle

\begin{abstract}
	%
	In this article we establish
	exponential moment bounds,
	moment bounds in 
	fractional order smoothness spaces, 
	a uniform H\"older continuity in time,
	and
	strong convergence rates    
	for a class of fully discrete exponential Euler-type numerical approximations of infinite dimensional stochastic convolution processes.
	The considered approximations involve specific taming and truncation terms and are therefore well suited to be used in the context of SPDEs with 
	non-globally 
	Lipschitz continuous
	nonlinearities. 
\end{abstract}
\tableofcontents 
\newpage 
\section{Introduction}
Stochastic partial differential equations
(SPDEs) of evolutionary type
are important modeling tools
in economics and the natural sciences
(see, e.g., Harms et al.\ \cite[Theorem~3.5]{HarmsStefanovitsTeichmannWutrich2017},
Filipovi\'c et al.\ \cite[Equation~(1.2)]{FilipovicTappeTeichmann2010},
Bl\"omker \& Romito~\cite[Equation~(1)]{BlomkerRomito2015},
Hairer~\cite[Equation~(3)]{HairerKPZ},
Mourrat \& Weber~\cite[Equation~(1.1)]{MourratWeber2017}, 
Birnir~\cite[Equation~(7)]{Birnir2013a},
and
Birnir~\cite[Equation~(1.5)]{Birnir2013b}).
However, exact solutions to SPDEs are 
usually not known explicitly.
Therefore, it has been and still is a very active research area to develop and analyze numerical approximation methods which approximate the exact solutions of SPDEs with a reasonable approximation accuracy in a reasonable computational time.
It is known that in order to
approximate the exact solutions of SPDEs
appropriately, the
numerical methods employed
should enjoy similar statistical properties,
such as finite uniform moment bounds
(see, e.g., 
Hutzenthaler \& Jentzen~\cite{HutzenthalerJentzenMemoires2015} and the references therein).
Unfortunately,
moments of the easily realizable Euler-Maruyama
and exponential Euler approximation methods
are known to diverge
for some 
stochastic differential equations (SDEs) and SPDEs with superlinearly growing nonlinearties
(see, e.g., Hutzenthaler et al.\ \cite{HutzenthalerJentzenKloeden2011, HutzenthalerJentzenKloeden2013}). 
This poses the challenge to develop new efficient approximation
methods which preserve finite moments
(see, e.g., 
Hutzenthaler et al.\ \cite[Theorem~1.1 and Lemma~3.9]{HutzenthalerJentzenKloeden2012},  
Gan \& Wang~\cite[Theorem~3.2 and Lemma~3.4]{WangGan2013},
Hutzenthaler \& Jentzen~\cite[Corollary~2.21 and Theorem~3.15]{HutzenthalerJentzenMemoires2015},
%
Tretyakov \& Zhang~\cite[Theorem~2.1]{TretyakovZhang2013},  
%
Sabanis~\cite[Theorem~2.2, Corollary~2.3, and Lemmas~3.2, 3.3]{Sabanis2013ECP},
and
%
Sabanis~\cite[Theorems~1--3, Lemmas~1,2]{Sabanis2016}
for finite dimensional 
stochastic evolution equations 
and, e.g., 
\cite[Proposition~7.3, Theorem~7.6]{JentzenPusnik2015},
Becker \& Jentzen~\cite[Corollaries~5.1, 6.15, and 6.17]{BeckerJentzen2016},
Becker \& Jentzen~\cite[Lemma~5.4, Theorem~1.1]{BeckerGessJentzenKloeden2017}
for infinite dimensional SPDEs). 
In this context, it has been revealed recently in, e.g., 
Hutzenthaler \& Jentzen~\cite[Theorem~1.3]{HutzenthalerJentzen2014PerturbationArxiv}
(cf., e.g.,
D\"orsek~\cite[Proposition~3.1]{Dorsek2012},
Hutzenthaler et al.\ \cite[Corollary~2.10]{HutzenthalerJentzenWang2017Published},
and~\cite[Corollary~3.4]{JentzenPusnik2016})
that finite exponential moments of numerical schemes are crucial for deriving strong convergence with rates  in the case of SDEs and SPDEs
with
non-globally Lipschitz continuous 
and
non-globally monotone
nonlinearities.

In this article we derive finite uniform exponential moment bounds for a class of fully discrete exponential Euler-type numerical approximations of infinite dimensional stochastic convolution processes  
(see Corollary~\ref{corollary:ExpMomentsNoise}
in Section~\ref{section:Conslusion} below). The considered numerical approximations involve specific taming and truncation terms and are therefore well suited to be applied in the context of semi-linear SPDEs with non-globally Lipschitz continuous and non-globally monotone nonlinearities. In addition to deriving exponential moment bounds we also establish finite uniform moment bounds in fractional order smoothness spaces (see Corollary~\ref{Corollary:FiniteMoments}
in Section~\ref{section:Conslusion} below), a uniform H\"older continuity in time 
(see Corollary~\ref{corollary:App_noise_regularity} in Section~\ref{section:Conslusion} below) as well as strong convergences rates for the considered numerical approximations 
(see Corollary~\ref{corollary:noise_approximation_converges} in Section~\ref{section:Conslusion} below). 
The application of our results to semi-linear SPDEs such as stochastic Burgers equation will be the subject of a future research article.
In 
Theorem~\ref{theorem:main}
below
we 
illustrate the 
results established  
in 
 Corollary~\ref{corollary:noise_approximation_converges}
and
Corollary~\ref{corollary:ExpMomentsNoise}.
The stochastic convolution process and its numerical approximations are denoted by 
$ O \colon [0,T] \times \Omega \to D( (-A)^\gamma ) $ 
and 
$ \O^{ M, N} \colon [0,T] \times \Omega \to P_N(H) $, 
$ M, N \in \N $, respectively.
\begin{theorem}
	\label{theorem:main}
	\sloppy
	Let
	$ ( H, \langle \cdot, \cdot \rangle_H, \left \| \cdot \right\|_H ) $
	and 
	$ \left( U, \langle \cdot, \cdot \rangle_U, \left\| \cdot \right \|_U \right) $
	be 
	separable 
	$ \R $-Hilbert spaces,
	let 
	$ ( h_n )_{ n \in \N} \subseteq H $
	be an 
	orthonormal basis of $ H $,  
	let
	$ (\values_n)_{n \in \N} \subseteq \R $
	satisfy
	$ \sup_{n \in \N}\values_n  < 0 $,
	let
	$ A \colon D( A ) \subseteq H \to H $ 
	be the linear operator 
	which satisfies   
	$ D(A) = \{
	v \in H \colon \sum_{n = 1}^\infty | \values_n \langle h_n, v \rangle_H  | ^2 <  \infty
	\} $
	and  
	$ \forall \, v \in D(A) \colon 
	A v = \sum_{ n = 1 }^\infty \values_n \langle h_n, v \rangle _H h_n $,
	let 
	$ p, T \in (0, \infty) $, 
	$ \beta \in [0, \infty) $,  
	$ \gamma \in [0, \nicefrac{1}{2} + \beta) $,
	$ \eta \in [0, \nicefrac{1}{2} + \beta - \gamma ) $,  
	$ \rho \in [0, \nicefrac{1}{2} + \beta - \gamma) \cap [0, \nicefrac{1}{2} ) $, 
	$ B \in \HS( U, D ( (-A)^\beta ) ) $, 
	let
	$ ( \Omega, \F, \P, ( \f_t )_{t \in [0,T]} ) $ 
be a filtered probability space which fulfills the usual conditions,
	let $ (W_t)_{t\in [0,T]} $
	be an $ \operatorname{Id}_U $-cylindrical  
	$ ( \f_t )_{t\in [0,T]} $-Wiener process,
	let
	$ O \colon [0,T] \times  \Omega \to D ( (-A)^\gamma ) $
	be a stochastic process which satisfies for every
	$ t \in [0,T] $ that
	$ 
	\P( O_t = \int_0^t 
	e^{(t-s)A} 
	B 
	\, 
	dW_s ) = 1 $,
	let 
	$ ( P_N )_{ N \in \N } \subseteq L(H) $ 
	satisfy
	for every $ N \in \N $,
	$ x \in H $ 
	that
	$ P_N(x) = \sum_{n=1}^N \langle h_n, x \rangle_H h_n $,
	let   
	$ W^N \colon [0,T] \times \Omega \to P_N(H) $,
	$ N \in \N $, 
	be 
	stochastic processes  
	which satisfy for every 
	$ N \in \N $,
	$ t \in [0,T] $
	that 
	$ \P ( W^N_t = \int_0^t P_N B \, dW_s ) = 1 $,
	%
	%
	let 
	$ \chi^{ M, N } \colon [0,T] \times \Omega \to [0,1] $,
	$ M, N\in \N $,
	be 
	$ ( \f_t )_{ t \in [0,T] } $-adapted stochastic processes
	which satisfy  
	$  \sup_{M, N\in \N} \sup_{ s \in [0,T] } 
	( 	
	\E [ 
	| \chi_s^{ M, N }  
	-
	1 
	|^{ \max\{ p, 2 \} } ] M^{ \max\{ p, 2 \} \rho } ) < \infty $,
	and
	let
	$ \O^{M, N} 
	\colon 
	[0, T] \times \Omega \to P_N(H) $,
	$ M, N \in \N $,
	be 
	stochastic processes  
	which satisfy for every 
	$ M, N \in \N $,
	$ m \in \{ 0, 1, \ldots, M-1 \} $,   
	$ t \in [ \nicefrac{mT}{M}, \nicefrac{ (m+1) T }{M} ] $ 
	that
	$ \O^{M, N}_0 = 0 $ 
	and 
		\begin{equation}
	\begin{split} 
	\label{eq:Method}
	\O_t^{M, N}
	=
	e^{(t - \nicefrac{m T}{ M } )A}
	\Big(  
	\O_{ \nicefrac{ m T }{ M } }^{ M, N }
	+
	\chi_{ \nicefrac{ m T }{ M } }^{ M, N } 
	\Big[
	\tfrac{ 	
		%
		%
		W^N_t - W^N_{ \nicefrac{ m T }{M} } 
	}{
		1 + 
		\|  
W^N_t - W^N_{ \nicefrac{ m T }{ M } } 
		\|_H^2
	}
\Big]
\Big) 
	.
	\end{split} 
	\end{equation}
	Then 
	\begin{enumerate}[(i)]
		\item  \label{item:FiniteExpIntro}
		it holds that
	$	\inf_{ \varepsilon \in (0,\infty) }
		\sup_{ M, N \in \N }
		\sup_{ t \in [0,T] }
		\E\big[ \exp \! \big( \varepsilon \| \O_t^{ M, N } \|_H^2 \big) \big] 
		< 
		\infty $ and 
	\item\label{item:RateIntro}  there exists a real number
	$ C \in \R $
	such that for every
	$ M, N \in \N $
	it holds that
	\begin{equation}
	\begin{split}
	\label{eq:RateIntro}
	&
	\sup\nolimits_{t\in [0,T]} 
	\big( 
	\E \big[ 
	\| (-A)^{ \gamma } ( \O_t^{ M, N } -  O_t ) \|_{ H }^p
	\big] 
	\big)^{ \nicefrac{1}{p} } 
	\leq
	C
	\big(
	(
	\inf\nolimits_{ n \in \{ N+1, N+2, \ldots \} }
	| \values_n |
	)^{ - \eta }
	+
	M^{-\rho}
	\big)
	.
	\end{split}
	\end{equation}
\end{enumerate}
\end{theorem}
Observe that 
item~\eqref{item:FiniteExpIntro}
in Theorem~\ref{theorem:main} 
is a direct consequence of
Corollary~\ref{corollary:ExpMomentsNoise}
(with
$ H = H $,
$ U = U $,
$ \H = (h_n)_{n \in \N} $,
$ \values = \values $,
$ A = A $,
$ \beta = \beta $,
$ T = T $, 
$ ( \Omega, \F, \P ) = ( \Omega, \F, \P ) $,
$ ( \f_t )_{ t \in [0,T] } 
=
( \f_t )_{ t \in [0,T] } $,
$ ( W_t )_{ t \in [0,T] } = ( W_t )_{ t \in [0,T] } $,
$ B = B $,
$ P_{ \{ h_1, \ldots, h_N \} } = P_N $,
$ \hat P_{ \mathbb{U} } = \operatorname{Id}_U $,
$ \chi^{ \{ 0, T/M, \ldots, T \}, \{ h_1, \ldots, h_N \}, \mathbb{U} } 
=
\chi^{ M, N } $,
$ \O^{ \{ 0, T/M, \ldots, T \}, \{ h_1, \ldots, h_N \}, \mathbb{U} } 
=
\O^{M, N} $, 
$ \varepsilon = \varepsilon $
for
$ M, N \in \N $,
$ \varepsilon \in [0, \nicefrac{1}{ 
	( 
	8 
	[
	\max \{
	\| B \|_{ \HS(U,H) }, 1 \}
	]^2
	\max \{ T, 1 \}
	)^2
}) $
in the notation of 
Corollary~\ref{corollary:ExpMomentsNoise})
and H\"older's inequality.
Moreover, note that
item~\eqref{item:RateIntro}
in Theorem~\ref{theorem:main} 
follows from
Corollary~\ref{corollary:noise_approximation_converges}
(with
$ H = H $,
$ U = U $,
$ \H = (h_n)_{n \in \N} $,
$ \values = \values $,
$ A = A $,
$ \beta = \beta $,
$ T = T $, 
$ ( \Omega, \F, \P ) = ( \Omega, \F, \P ) $,
$ ( \f_t )_{ t \in [0,T] } 
=
( \f_t )_{ t \in [0,T] } $,
$ ( W_t )_{ t \in [0,T] } = ( W_t )_{ t \in [0,T] } $,
$ B = B $,
$ P_{ \{ h_1, \ldots, h_N \} } = P_N $,
$ \hat P_{ \mathbb{U} } = \operatorname{Id}_U $,
$ \chi^{ \{ 0, T/M, \ldots, T \}, \{ h_1, \ldots, h_N\}, \mathbb{U} } 
=
\chi^{ M, N } $,
$ \O^{ \{ 0, T/M, \ldots, T \}, \{ h_1, \ldots, h_N\}, \mathbb{U} } 
=
\O^{M, N} $,
$ p = \max \{ p, 1 \} $,
$ C = C $, 
$ \gamma = \gamma $,
$ \eta = \eta $,
$ \rho = \rho $,
$ O = O $
for  
$ M, N \in \N $
in the notation of 
Corollary~\ref{corollary:noise_approximation_converges}).
We would like to point out  
that the exponential moment bound in 
Corollary~\ref{corollary:ExpMomentsNoise} below
(see also item~\eqref{item:FiniteExpIntro} above)
is not a direct consequence of the 
one
established
in~\cite{JentzenPusnik2016}.
The difference between the numerical 
method in~\cite[(1) in Section~1]{JentzenPusnik2016}
and~\eqref{eq:Method} above lies, roughly speaking, in the 
less restrictive choice for the truncation functions
$ \chi^{M, N} \colon [0,T] \times \Omega \to [0,1] $, 
$ M, N \in \N $,
in~\eqref{eq:Method}
compared to~\cite[(1) in Section~1]{JentzenPusnik2016}.
This extended class of discrete approximations allows
to truncate the numerical method 
independently of the current value of the approximation process itself. 
The flexibility in the choice of the truncation function
is, in turn, important  
for applying
Corollaries~\ref{Corollary:FiniteMoments}--\ref{corollary:ExpMomentsNoise} 
to establish 
strong convergence rates
for fully discrete numerical approximations
in the case of SPDEs
with non-globally Lipschitz continuous
and non-globally monotone nonlinearities.
The remainder of this article is structured as follows.
In Section~\ref{section:StrengthenedAprioriBound} we analyze a temporally semi-discrete version of our approximation scheme 
(see~\eqref{eq:Numerics} in
Setting~\ref{setting:Strengthened_strong_a_priori_moment_bounds} 
below).
In Subsection~\ref{subsection:Ito} 
 the considered numerical approximations are rewritten as It\^o processes and finite moment 
 bounds in fractional order smoothness spaces are derived. 
 The It\^o representation enables us to
 establish
 H\"older regularity in time in Subsection~\ref{subsection:Holder}.
Furthermore, under additional assumptions on the truncation
function we establish
temporal strong convergence rates of the proposed numerical methods 
in Subsection~\ref{subsection:Error}.
In Subsection~\ref{subsection:ExpoMoments}
we further improve our moment   estimates from
Subsection~\ref{subsection:Ito}
in order to derive finite exponential moment bounds 
in Lemma~\ref{lemma:Exp_regularity}.
The results from Section~\ref{section:StrengthenedAprioriBound}
are combined in Section~\ref{section:Conslusion} to establish 
uniform moment 
bounds in fractional order smoothness spaces,
a uniform H\"older regularity in time,
strong convergence rates,
and uniform exponential moment bounds
for fully discrete 
tamed-truncated numerical approximations
in Corollaries~\ref{Corollary:FiniteMoments}--\ref{corollary:ExpMomentsNoise},
respectively.

\subsection{General setting}
Throughout this article the following setting is frequently used. 
\begin{setting}
	\label{setting:main} 
For every set $ X $ let  $ \mathcal{P}(X) $ be the power set of $ X $,
for every set $ X $ let
$ \mathcal{P}_0(X) $ be the set given by
$ \mathcal{P}_0(X) = \{ \theta \in \mathcal{P}(X) \colon \theta \, \text{has finitely many elements} \} $,
for every 
$ T \in (0,\infty) $ 
let $ \varpi_T $ 
be the set given by
 $ \varpi_T
 =
 \{
 \theta \in \mathcal{P}_0( [0,T] ) \colon
 \{0, T\} \subseteq \theta   
 \} $,
 for every $ T \in (0,\infty) $
 let 
 $ \left | \cdot \right|_T \colon \varpi_T \to [0,T] $
 be the function which satisfies for every
 $ \theta \in \varpi_T $ that
 \begin{equation}
  | \theta  |_T
 =
 \max\!\Big\{
 x \in (0,\infty)
 \colon
 \big(
 \exists \, a, b \in \theta
 \colon
 \big[
 x = b - a
 \text{ and }
 \theta \cap ( a, b ) = \emptyset
 \big]
 \big)
 \Big\}
 \in (0,T],
 \end{equation}
 for every 
 $ \theta \in ( \cup_{T\in (0,\infty)} \varpi_T) $
 let
 $ \llcorner \cdot \lrcorner_\theta \colon [0,\infty) \to [0,\infty) $ 
 be the function which satisfies for every 
 $ t\in (0,\infty) $
 that 
 $ \llcorner t \lrcorner_\theta
 =  \max ( [0,t) \cap \theta ) $
 and  
 $ \llcorner 0 \lrcorner_\theta  
 = 0 $,
 for every measure space
 $ ( \Omega, \F, \mu) $,
 every measurable space $ ( S, \mathcal{S} ) $,
 every set $ R $,
 and every function
 $ f \colon \Omega \to R $
 let
 $ [f]_{\mu, \mathcal{S}} $
 be the set given by
$ [f]_{\mu,\mathcal{S}}=\{g\colon\Omega\to S \colon (\exists\, A\in\F \colon[\mu(A)=0\text{ and }\{\omega\in\Omega \colon f(\omega)\neq g(\omega)\}\subseteq A]) \text{ and } (\forall\, A\in\mathcal{S} \colon g^{-1}(A)\in \F)\} $,
let
$ ( H, \langle \cdot, \cdot \rangle _H, \left \| \cdot \right\|_H ) $
and
$ ( U, \langle \cdot, \cdot \rangle_U, \left\| \cdot \right \|_U ) $
be separable $ \R $-Hilbert spaces,
let 
$ \H \subseteq H $
be a non-empty orthonormal basis of $ H $,  
let
$ \values \colon \H \to \R $
be a function which satisfies
$ \sup_{h\in \H} \values_h < 0 $,
let
$ A \colon D( A ) \subseteq H \to H $ 
be the linear operator which satisfies
$ D(A) = \{
v \in H \colon \sum_{ h \in \H}  | \values_h \langle h, v \rangle_H  |^2 <  \infty
 \} $ 
and  
$ \forall \, v \in D(A) \colon 
A v = \sum_{ h \in \H} \values_h \langle h, v \rangle _H h $,
and
let
$ (H_r, \langle \cdot, \cdot \rangle_{H_r}, \left \| \cdot \right \|_{H_r} ) $, $ r\in \R $, be a family of interpolation spaces associated to $ -A $
(cf., e.g., \cite[Section~3.7]{SellYou2002}).
\end{setting}
\section[Regularity properties of 
temporal   
approximations
of stochastic convolutions]{Regularity properties of 
	temporally 
	semi-discrete 
	 tamed-truncated  
	 approximations
of stochastic convolutions}
\label{section:StrengthenedAprioriBound}
\begin{setting} 
\label{setting:Strengthened_strong_a_priori_moment_bounds}
Assume Setting~\ref{setting:main},
let 
$ \beta \in [0,\infty) $, 
$ \gamma \in [0, \nicefrac{1}{2} + \beta ) $,
$ T \in (0,\infty) $,
$ \theta \in \varpi_T $,
$ B \in  \HS(U,H_\beta) $,
let
$ \mathbb{B} \in L(H, U ) $ 
be the bounded linear operator
which satisfies for every
$ u \in U $, $ h \in H $ that
$ \langle B u, 
h \rangle_H 
=
\langle u, 
\mathbb{B} h \rangle_U $, 
let
$ ( \Omega, \F, \P, ( \f_t )_{t \in [0,T]} ) $ 
be a filtered probability space which fulfills the usual conditions,
let $ (W_t)_{t\in [0,T]} $
be an $ \operatorname{Id}_U $-cylindrical  
$ ( \f_t )_{t\in [0,T]} $-Wiener process,  
let
$ \chi \colon [0,T] \times \Omega \to [0, 1] $ 
be an $ ( \f_t )_{t\in [0,T]} $-adapted
stochastic process,
and
let 
$ \O \colon [0,T] \times \Omega \to H_\gamma $ 
be a   
stochastic process
which satisfies for every $ t \in [0,T] $ that 
$ \O_0  = 0 $ 
and
\begin{equation}
\begin{split}
\label{eq:Numerics}
[
\O_t
]_{\P, \mathcal{B}( H_\gamma) } 
& 
= 
 [
e^{(t-\llcorner t \lrcorner_\theta )A} 
\O_{
	\llcorner t \lrcorner_\theta
}  
 ]_{\P, \mathcal{B}( H_\gamma )}  
+ 
\frac{
	\int_{ \llcorner t \lrcorner_\theta }^t
	\chi_{  
	\llcorner t \lrcorner_\theta}
	\,
	e^{(t-\llcorner t \lrcorner_\theta )A}   
	B 
	\, 
	dW_s
}{
	1 + 
	\|
	\int_{\llcorner t \lrcorner_\theta}^t  
	B 
	\, 
	dW_s
	\|_H^2
}   
.
\end{split}
\end{equation}
\end{setting}
\subsection[Moment bounds
for temporal approximations of stochastic convolutions]{Moment bounds
for temporally semi-discrete approximations of stochastic convolutions}
\label{subsection:Ito}
In this subsection we
provide 
in
Lemma~\ref{lemma:Ito_representation} 
a representation
of 
the approximation process
$ \O \colon [0,T] \times \Omega \to H_\gamma $
from Setting~\ref{setting:Strengthened_strong_a_priori_moment_bounds} 
as a mild It\^o process
(cf.\ Da Prato et al.\ \cite[Definition~1]{DaPratoJentzenRoeckner2012}).
This enables us
to obtain certain
moment bounds
for
$ \O \colon [0,T] \times \Omega \to H_\gamma $ 
in
Lemmas~\ref{lemma:some_helping_estimates}
and~\ref{lemma:more_regularity}.
\begin{lemma}
	\label{lemma:Ito_representation}
	Assume Setting~\ref{setting:Strengthened_strong_a_priori_moment_bounds} 
	and for every $ s \in [0,T] $ let
	$ X_{s, (\cdot) }(\cdot) 
	=
	( X_{s,t}(\omega) )_{ (t,\omega)\in[s,T]\times\Omega }
	\colon 
	[s,T] \times \Omega
	\to H_\beta $
	be an $ ( \f_t )_{ t \in [s,T] } $-adapted stochastic process 
	with continuous sample paths 
	which satisfies for every
	$ t \in [s, T] $
	that
	$
	[
	X_{s,t}
	]_{\P, \mathcal{B}( H_\beta )}
	= \int_s^t B 
	\, 
	dW_u
	$.
	Then it holds for every $ t \in [0,T] $ that
	\begin{equation}
	\begin{split}
	& 
	[ \O_t ]_{\P, \mathcal{B}(H_\gamma)} 
	=
	\int_0^t 
	\chi_{ \llcorner u \lrcorner_\theta } 
	e^{(t- \llcorner u \lrcorner_\theta )A}
	\Big[
	\tfrac{B  }{1+\| X_{{\llcorner u \lrcorner_\theta}, u} \|_H^2}
	-
	\tfrac{2 X_{{\llcorner u \lrcorner_\theta}, u} \langle X_{{\llcorner u \lrcorner_\theta}, u}, B  (\cdot) \rangle_H}
	{ ( 1 + \| X_{{\llcorner u \lrcorner_\theta}, u} \|_H^2 )^2 } 
	\Big]
	\, 
	dW_u 
	\\
	&  
	+
	\Big[
	\int_0^t
	\chi_{\llcorner u \lrcorner_\theta}
	e^{(t-\llcorner u \lrcorner_\theta)A}
	\Big( 
	\tfrac{ 4  
		\|  \mathbb{B}
		X_{\llcorner u \lrcorner_\theta, u}  \|_U^2
		X_{\llcorner u \lrcorner_\theta, u}	
	}
	{(1 + \| X_{\llcorner u \lrcorner_\theta, u} \|_H^2)^3}
	-
	\tfrac{
		2 B  
		\mathbb{B}
		X_{\llcorner u \lrcorner_\theta, u}
		+ 
		\| B   \|_{\HS(U,H)}^2
		X_{\llcorner u \lrcorner_\theta, u} 
	}{
		( 1 + \| X_{\llcorner u \lrcorner_\theta, u} \|_H^2 )^2
	}
	\Big) 
	\, 
	du
	\Big]_{\P, \mathcal{B}(H_\gamma)} 
	.
	\end{split}
	\end{equation}
\end{lemma}
\begin{proof}[Proof of Lemma~\ref{lemma:Ito_representation}]
	Throughout this proof 
	let
	$ \psi \colon H \to H $ be the function 
	which satisfies for every $ v \in H $ that
	$ \psi( v ) = \tfrac{v}{1+\|v\|_H^2} $
	and let
	$ \mathbb{U} \subseteq U $ be an orthonormal basis of $ U $.
	Note that for every
	$ u, v, z \in H $ it holds that
	\begin{equation}
	\begin{split}
	\psi'(z)(u)
	=
	\tfrac{u}{1+\|z\|_H^2} - \tfrac{ 2z\langle z, u \rangle_H }{(1+ \|z\|_H^2)^2}
	\end{split}
	\end{equation}
	and
	\begin{equation}
	\begin{split}
	\psi''(z)(u,v)
	=
	- \tfrac{ 2 [ u \langle z, v \rangle_H 
		+ v \langle z, u \rangle_H 
		+ z \langle u, v \rangle_H ] }
	{ ( 1 + \| z \|_H^2 )^2 }
	+
	\tfrac{ 8 z \langle z, u \rangle_H \langle z, v \rangle_H }
	{ ( 1 + \| z \|_H^2 )^3 }
	.
	\end{split}
	\end{equation}
	It\^o's formula 
	(see, e.g.,
	Brze\'zniak et al.\ \cite[Theorem~2.4]{BrzezniakNeervenVeraarWeis2008}
	with
	$ H = U $,
	$ E = H $, 
	$ F = H $,
	$ f = ( [0,T] \times H \ni (t, x) \mapsto \psi(x) \in H ) $,
	$ \Phi = ( [0,T] \times \Omega \ni (t,\omega) \mapsto B \in \HS(U, H) ) $
	in the notation of 
	Brze\'zniak et al.\ \cite[Theorem~2.4]{BrzezniakNeervenVeraarWeis2008})
	hence proves 
	that
	for every 
	$ s \in [0,T] $, $ t \in [s,T] $ 
	it holds that
	\begin{equation}
	\begin{split} 
	&
	[ \psi( X_{s,t} ) ]_{\P, \mathcal{B}(H_\beta)}
	=  
	\int_s^t 
	\Big(
	\tfrac{B  }{1+\| X_{s, u} \|_H^2}
	-
	\tfrac{2 X_{s, u} \langle X_{s, u}, B (\cdot) \rangle_H}
	{ ( 1 + \| X_{s, u} \|_H^2 )^2 } 
	\Big)
	\, 
	dW_u 
	\\
	&
	\quad
	+
	\Big[
	\int_s^t
	\Big(
	\sum\nolimits_{ { \bf u} \in \mathbb{U} }
	\tfrac{ 4 X_{s,u} | \langle X_{s,u}, B { \bf u}  \rangle_H|^2}
	{(1 + \| X_{s,u} \|_H^2)^3}
	-
	\tfrac{
		2 B  {\bf u}  \langle X_{s,u}, 
		B  { \bf u} \rangle_H
		+ X_{s,u} \| B { \bf u}  \|_H^2 
	}{
		( 1 + \| X_{s,u} \|_H^2 )^2
	}
	\Big)
	\, 
	du
	\Big]_{\P, \mathcal{B}(H_\beta)}
	\\
	&
	= 
	\int_s^t 
	\Big(
	\tfrac{B  }{1+\| X_{s, u} \|_H^2}
	-
	\tfrac{2 X_{s,u} \langle X_{s,u}, B (\cdot) \rangle_H}
	{ ( 1 + \| X_{s,u} \|_H^2 )^2 }
	\Big) 
	\, 
	dW_u
	\\
	&
	\quad
	+
	\Big[
	\int_s^t 
	\Big(
	\tfrac{ 4 
		X_{s, u}
		\| \mathbb{B}
		X_{s,u}  \|_U^2
	}
	{(1 + \| X_{s,u} \|_H^2)^3}
	-
	\tfrac{
		2 B  \mathbb{B} X_{s,u} 
		+ \| B \|_{ \HS(U,H) }^2 X_{s,u}  
	}{
		( 1 + \| X_{s,u} \|_H^2 )^2
	}
	\Big)
	\, 
	du
	\Big]_{\P, \mathcal{B}(H_\beta)}
	.
	\end{split}
	\end{equation}
	Therefore, we obtain for every $ t \in [0,T] $
	that
	\begin{equation}
	\begin{split}
	\label{eq:afterIto}
	&
	[ \O_t ]_{\P, \mathcal{B}(H_\gamma)} 
	= 
	 [
	e^{(t-\llcorner t \lrcorner_\theta )A}
	 (
	\O_{\llcorner t \lrcorner_\theta}
	+ 
	\chi_{\llcorner t \lrcorner_\theta}
	\psi (
	X_{\llcorner t \lrcorner_\theta, t}
	)  
	 )
	]_{\P, \mathcal{B}(H_\gamma)} 
	\\
	&
	=  
	[
	e^{(t-\llcorner t \lrcorner_\theta )A} 
	\O_{\llcorner t \lrcorner_\theta} 
 	]_{\P, \mathcal{B}(H_\gamma)} 
	+
	\int_{\llcorner t \lrcorner_\theta}^t
	\chi_{ \llcorner t \lrcorner_{ \theta }}  
	e^{(t- \llcorner t \lrcorner_\theta )A}
	\Big( 
	\tfrac{ B }
	{1+\| X_{{ \llcorner t \lrcorner_\theta }, u} \|_H^2}
	-
	\tfrac{2 X_{{ \llcorner t \lrcorner_\theta }, u} 
		\langle X_{{\llcorner t \lrcorner_\theta}, u}, 
		B (\cdot) \rangle_H}
	{ ( 1 + \| X_{{\llcorner t \lrcorner_\theta}, u} \|_H^2 )^2 } 
	\Big)
	\, 
	dW_u 
	\\
	& 
	\quad 
	+ 
	\Big[
	\int_{\llcorner t \lrcorner_\theta}^t 
	\chi_{ \llcorner t \lrcorner_{ \theta }}  
	e^{(t-\llcorner t \lrcorner_\theta)A}
	\Big(
	\tfrac{ 4 
		X_{\llcorner t \lrcorner_\theta, u}
		\| \mathbb{B}
		X_{\llcorner t \lrcorner_\theta, u}  \|_U^2
	}
	{(1 + \| X_{\llcorner t \lrcorner_\theta, u} \|_H^2)^3}
	-
	\tfrac{
		2 B  
		\mathbb{B}
		X_{ \llcorner t \lrcorner_\theta, u} 
		+ 
		\| B  \|_{ \HS(U,H) }^2
		X_{ \llcorner t \lrcorner_\theta, u}  
	}{
		( 1 + \| X_{\llcorner t \lrcorner_\theta, u} \|_H^2 )^2
	}
	\Big)
	\, 
	du  
	\Big ]_{\P, \mathcal{B}( H_\gamma ) } 
	.
	\end{split}
	\end{equation}
	This  
	completes the proof
	of Lemma~\ref{lemma:Ito_representation}.
\end{proof}
\begin{lemma}
	\label{lemma:some_helping_estimates}
	Assume Setting~\ref{setting:Strengthened_strong_a_priori_moment_bounds},
	let 
	$ p \in [2,\infty) $,
	$ \rho \in [\beta, \nicefrac{1}{2} + \beta ) $,
	$ \eta \in [\beta, 1+ \beta) $, 
	and
	for every $ s \in [0,T] $
	let
		$ X_{s, (\cdot) }(\cdot) 
	=
	( X_{s,t}(\omega) )_{ (t,\omega)\in[s,T]\times\Omega }
	\colon 
	[s,T] \times \Omega
	\to H_\beta $ 
	be an $ ( \f_t )_{ t \in [s,T]} $-adapted 
	stochastic process
	with continuous sample paths 
	which satisfies for every
	$ t \in [s, T] $
	that
	$
	[
	X_{s, t}
	]_{\P, \mathcal{B}(H_\beta)}
	= \int_s^t B 
	\, 
	dW_u
	$.
	Then 
	it holds 
	for every
	$ s \in [0,T] $, $ t \in [s,T] $ 
	that
	\begin{equation}
		\begin{split} 
		\label{eq:Estimate1}
			& 
			\Big\|
			\int_s^t 
			\chi_{ \llcorner u \lrcorner_\theta } 
			e^{(t- \llcorner u \lrcorner_\theta )A} 
			\,
			\tfrac{B }{1+\| X_{{ \llcorner u \lrcorner_\theta }, u} \|_H^2}  
			\, 
			dW_u
			\Big\|_{L^p(\P; H_\rho )}
			\leq  
			\| B \|_{\HS(U,H_\beta)}
			\tfrac{ p ( t - s )^{ \nicefrac{1}{2} + \beta - \rho } 	
		 }
			{
				\sqrt{ 2 ( 1 + 2 \beta - 2 \rho ) }
			}
			,
		\end{split}
	\end{equation}
	\begin{equation}
		\begin{split}
		\label{eq:Estimate2}
			&
			\Big\|
			\int_s^t 
			\chi_{ \llcorner u \lrcorner_\theta } 
			e^{(t- \llcorner u \lrcorner_\theta )A} 
			\,
			\tfrac{2 X_{{ \llcorner u \lrcorner_\theta }, u} \langle X_{{ \llcorner u \lrcorner_\theta }, u}, B 
				\, 
				dW_u  \rangle_H}
			{ ( 1 + \| X_{{ \llcorner u \lrcorner_\theta }, u} \|_H^2 )^2 } 
			\Big \|_{L^p(\P; H_\rho)}
			\\
			&
			\leq 
			2 \sqrt{2} p^3   
			\| B \|_{ \HS(U, H_\beta) }^3
			| \sup\nolimits_{ h \in \H } \values_h |^{-2\beta}
			\tfrac{  
				( t - s )^{\nicefrac{1}{2} + \beta -  \rho} 
			}
		{ \sqrt{  1 + 2 \beta - 2 \rho  } }  
			|\theta|_T,
		\end{split}
	\end{equation}
	and
	\begin{equation}
		\begin{split}
		\label{eq:Estimate3}
			&
			\Big\|
			\int_s^t
			\chi_{ \llcorner u \lrcorner_\theta } 
			e^{(t- \llcorner u \lrcorner_\theta )A}
			\Big( 
			\tfrac{ 4  
				\|  \mathbb{B}
				X_{\llcorner u \lrcorner_\theta, u}  \|_U^2
		X_{\llcorner u \lrcorner_\theta, u}	
		}
			{(1 + \| X_{\llcorner u \lrcorner_\theta, u} \|_H^2)^3}
			-
			\tfrac{
				2 B  
				\mathbb{B} 
				X_{\llcorner u \lrcorner_\theta, u}
				+ 
				\| B  \|_{\HS(U,H)}^2 
				X_{\llcorner u \lrcorner_\theta, u}
			}{
			( 1 + \| X_{\llcorner u \lrcorner_\theta, u} \|_H^2 )^2
		}
		\Big) 
		\, 
		du
		\Big\|_{\L^p(\P; H_\eta)}
		\\
		&
		\leq
		2
	\sqrt{2}
	p  
	\| B \|_{ \HS(U,H_\beta) }^3 
	| \sup\nolimits_{h \in \H} \values_h |^{-2\beta} 
	\tfrac{  
		(t- s )^{ 1 + \beta - \eta } 
 }{ 1 + \beta - \eta } 
	[ | \theta |_T ]^{ \nicefrac{1}{2} }
		.
	\end{split}
\end{equation}
\end{lemma}
\begin{proof}[Proof of Lemma~\ref{lemma:some_helping_estimates}]
	Note that 
	the
	Burkholder-Davis-Gundy type inequality in Lemma~7.7 in Da Prato \& Zabczyk~\cite{dz92}
	shows for every 	$ s \in [0,T] $, $ t \in [s,T] $ 
	that
	\begin{equation}
		\begin{split} 
			\label{eq:some111}
			&
			\Big\|
			\int_s^t 
			\chi_{ \llcorner u \lrcorner_\theta } 
			e^{(t- \llcorner u \lrcorner_\theta)A} 
			\,
			\tfrac{B }{1+\| X_{{ \llcorner u \lrcorner_\theta }, u} \|_H^2}  
			\, 
			dW_u
			\Big\|_{L^p(\P; H_\rho)}^2
			\\
			&
			\leq
			\tfrac{p(p-1)}{2}
			\int_s^t  
			\Big\|  
			e^{(t- \llcorner u \lrcorner_\theta )A}
			\tfrac{
				B 
			}{1+\| X_{ \llcorner u \lrcorner_\theta, u} \|_H^2}  
			\Big\|_{\L^p(\P; \HS(U, H_\rho))}^2
			\, 
			du
			\\
			&
			\leq
			\tfrac{p^2}{2}
			\int_s^t  
			\|  
			(-A)^{\rho-\beta}
			e^{(t- \llcorner u \lrcorner_\theta )A}
			\|_{L(H)}^2
			\Big\|
			\tfrac{
				B 
			}{1+\| X_{\llcorner u \lrcorner_\theta, u} \|_H^2}  
			\Big\|_{\L^p(\P; \HS(U, H_\beta))}^2
			\, 
			du
			\\
			&
			\leq
			\tfrac{p^2}{2}  
			\| B \|_{ \HS(U, H_\beta) }^2 
						\int_s^t 
						(t - \llcorner u \lrcorner_\theta)^{2 \beta - 2\rho} 
						\, 
						du 
			.
			\end{split}
			\end{equation}
			In addition, observe for every
				$ s \in [0,T] $, $ t \in [s,T] $ 
			 that
			\begin{equation} 
			\begin{split} 
				\label{eq:some112}
			&
			\int_s^t 
			(t - \llcorner u \lrcorner_\theta )^{2 \beta - 2\rho} 
			\, 
			du
			\leq 
			\int_s^t 
			(t - u)^{2 \beta - 2\rho} 
			\, 
			du 
			\leq 			
			\tfrac{ (t-s)^{1+2\beta-2\rho} }{ 1 + 2 \beta - 2 \rho } 		 
			.
		\end{split}
	\end{equation}
	Combining this and~\eqref{eq:some111} 
	establishes~\eqref{eq:Estimate1}.
	Furthermore, note that
	the
	Burkholder-Davis-Gundy type inequality in Lemma~7.7 in Da Prato \& Zabczyk~\cite{dz92} 
	proves that for every
		$ s \in [0,T] $, $ t \in [s,T] $ 
	it holds that
	\begin{equation}
		\begin{split}
			&
			\Big\|
			\int_s^t 
			\chi_{ \llcorner u \lrcorner_\theta } 
			e^{(t- \llcorner u \lrcorner_\theta )A} 
			\,
			\tfrac{2 X_{{ \llcorner u \lrcorner_\theta }, u} \langle X_{{ \llcorner u \lrcorner_\theta }, u}, B(\cdot) \rangle_H}
			{ ( 1 + \| X_{{ \llcorner u \lrcorner_\theta }, u} \|_H^2 )^2 }  
			\, 
			dW_u
			\Big \|_{L^p(\P; H_\rho)}^2
			\\
			&
			\leq
			\tfrac{p(p-1)}{2}
			\int_s^t
			\Big\|
			e^{(t- \llcorner u \lrcorner_\theta )A}
			\,
			\tfrac{2 X_{\llcorner u \lrcorner_\theta, u} \langle X_{\llcorner u \lrcorner_\theta, u}, B(\cdot) \rangle_H}
			{ ( 1 + \| X_{\llcorner u \lrcorner_\theta, u} \|_H^2 )^2 }  
			\Big\|_{\L^p(\P; \HS(U,H_\rho) )}^2
			\,
			du
						\\
						&
						\leq
						2p^2
						\int_s^t
						\big\|
						e^{(t- \llcorner u \lrcorner_\theta)A}
						\,
						 X_{\llcorner u \lrcorner_\theta, u} \langle X_{\llcorner u \lrcorner_\theta, u}, B(\cdot) \rangle_H   
						\big\|_{\L^p(\P; \HS(U,H_\rho) )}^2
						\,
						du
						\\
						&
						\leq
						2p^2
						\int_s^t
						 \big\|
					 \|
						(-A)^{\rho - \beta}
						e^{(t- \llcorner u \lrcorner_\theta )A}
						\|_{L(H)}
						\| 
						(-A)^\beta
						 X_{\llcorner u \lrcorner_\theta, u} 
						 \|_H
						 \| 
						 \langle X_{\llcorner u \lrcorner_\theta, u}, B(\cdot) \rangle_H
						 \|_{ \HS(U, \R) }   
						\big\|_{\L^p(\P; \R )}^2
						\,
						du 
			.
		\end{split}
	\end{equation}
	The H\"older inequality, the
	Burkholder-Davis-Gundy type inequality in Lemma~7.7 in Da Prato \& Zabczyk~\cite{dz92}, 
	and
	the fact that
	$ \| B \|_{ \HS(U,H)}
	\leq | \sup_{ h \in \H } \values_h |^{-\beta} \| B \|_{\HS(U,H_\beta)} $
	therefore
	ensure that for every
		$ s \in [0,T] $, $ t \in [s,T] $ 
	it holds that
	\begin{equation} 
		\begin{split}
			\label{eq:221}
			&
			\Big\|
			\int_s^t 
			\chi_{ \llcorner u \lrcorner_\theta } 
			e^{(t- \llcorner u \lrcorner_\theta )A} 
			\,
			\tfrac{2 X_{{ \llcorner u \lrcorner_\theta }, u} \langle X_{{ \llcorner u \lrcorner_\theta }, u}, B  
				(\cdot) \rangle_H}
			{ ( 1 + \| X_{{ \llcorner u \lrcorner_\theta }, u} \|_H^2 )^2 }  
			\, 
			dW_u 
			\Big\|_{L^p(\P; H_\rho)}^2
			\\
			&
			\leq
			2 p^2
			\|  B  \|_{ \HS( U, H ) }^2
			\int_s^t
 \|
(-A)^{\rho - \beta}
e^{(t- \llcorner u \lrcorner_\theta )A}
\|_{L(H)}^2
			\| 
			X_{\llcorner u \lrcorner_\theta, u} 
			\|_{\L^{2p}(\P;H_\beta)}^2
			\| 
			X_{\llcorner u \lrcorner_\theta, u} 
			\|_{\L^{2p}(\P;H)}^2
			\,
			du
			\\
			&
			\leq
			8 p^6 
			\|  B  \|_{ \HS(U,H) }^2
			\int_s^t
 \|
(-A)^{\rho - \beta}
e^{(t- \llcorner u \lrcorner_\theta )A}
\|_{L(H)}^2
			\bigg[ 
			\smallint_ { \llcorner u \lrcorner_\theta }^u
			\|
			B 
			\|_{\HS(U,H_\beta)}^2
			\,
			dw
			\bigg]
			\bigg[ 
			\smallint_ { \llcorner u \lrcorner_\theta }^u
			\|
			B 
			\|_{\HS(U,H)}^2
			\,
			dw
			\bigg]
			\,
			du
			\\
			&
			\leq
			8p^6
			| \sup\nolimits_{ h \in \H } \values_h |^{-4\beta} 
			\|  B  \|_{ \HS(U,H_\beta) }^6
			\int_s^t
 \|
(-A)^{\rho - \beta}
e^{(t- \llcorner u \lrcorner_\theta )A}
\|_{L(H)}^2
			(u - \llcorner u \lrcorner_\theta )^2 
			\,
			du
			\\
			&
			\leq
		8p^6 
	| \sup\nolimits_{ h \in \H } \values_h |^{-4\beta} 
	\|  B  \|_{ \HS(U,H_\beta) }^6  
	\int_s^t
				( t - \llcorner u \lrcorner_\theta )^{2 \beta - 2 \rho} 
	\,
	du 
		[ | \theta |_T ]^2 
			.
		\end{split}
	\end{equation}
	Combining this and~\eqref{eq:some112} 
	assures that~\eqref{eq:Estimate2} holds.
	In the next step we observe that
    for every 
    	$ s \in [0,T] $, $ t \in [s,T] $ 
    it holds that
	\begin{equation}
		\begin{split}
			&
			\Big\|
			\int_s^t
			\chi_{\llcorner u \lrcorner_\theta} 
			e^{(t- \llcorner u \lrcorner_\theta)A}
			\Big( 
			\tfrac{ 4  
				\|  \mathbb{B}
				X_{\llcorner u \lrcorner_\theta, u}  \|_U^2
			X_{\llcorner u \lrcorner_\theta, u}
		}
			{(1 + \| X_{\llcorner u \lrcorner_\theta, u} \|_H^2)^3}
			-
			\tfrac{
				2 B   
				\mathbb{B} 
				X_{\llcorner u \lrcorner_\theta, u}
				+ 
				\| B \|_{\HS(U,H)}^2
				X_{\llcorner u \lrcorner_\theta, u} 
			}{
			( 1 + \| X_{\llcorner u \lrcorner_\theta, u} \|_H^2 )^2
		}
		\Big)  
		\, 
		du
		\Big\|_{\L^p(\P; H_\eta)}
		\\
		&
		\leq
		\int_s^t
		\Big\| 
		e^{(t-\llcorner u \lrcorner_\theta)A}
		\Big( 
		\tfrac{ 4  
			\|  \mathbb{B}
			X_{\llcorner u \lrcorner_\theta, u}  \|_U^2
		X_{\llcorner u \lrcorner_\theta, u}
	}
		{(1 + \| X_{\llcorner u \lrcorner_\theta, u} \|_H^2)^3}
		-
		\tfrac{
			2 B
			\mathbb{B}
			X_{\llcorner u \lrcorner_\theta, u}
			+ 
			\| B \|_{\HS(U,H)}^2 
			X_{\llcorner u \lrcorner_\theta, u}
		}{
		( 1 + \| X_{\llcorner u \lrcorner_\theta, u} \|_H^2 )^2
	}
	\Big) 
	\Big\|_{\L^p(\P; H_\eta)}
	\, 
	du
	\\
	&
	\leq
	\int_s^t
	\| 
	( - A )^{\eta - \beta}
	e^{(t-\llcorner u \lrcorner_\theta)A}
	\|_{L(H)}
	\Big\|  
	\tfrac{ 4  
		\|  \mathbb{B}
		X_{\llcorner u \lrcorner_\theta, u}  \|_U^2
	X_{\llcorner u \lrcorner_\theta, u}
}
	{(1 + \| X_{\llcorner u \lrcorner_\theta, u} \|_H^2)^3}
	-
	\tfrac{
		2 B
		\mathbb{B}
		X_{\llcorner u \lrcorner_\theta, u}
		+ 
		\| B \|_{\HS(U,H)}^2 
		X_{\llcorner u \lrcorner_\theta, u}
	}{
		( 1 + \| X_{\llcorner u \lrcorner_\theta, u} \|_H^2 )^2
	}  
	\Big\|_{\L^p(\P; H_\beta)}
	\, 
	du
	\\
	&
	\leq 
	\int_s^t
	\| 
	( - A )^{\eta - \beta}
	e^{(t-\llcorner u \lrcorner_\theta)A}
	\|_{L(H)}
	\\
	&
	\quad
	\cdot
	\Big\|  
	\Big( 
	\tfrac{ 4  
		\| \mathbb{B} \|_{L(H,U)}^2
		\| X_{\llcorner u \lrcorner_\theta, u}  \|_{H}^2}
	{(1 + \| X_{\llcorner u \lrcorner_\theta, u} \|_H^2)^3}
	+
	\tfrac{
		2
		\| B \mathbb{B} \|_{ L(H_\beta, H_\beta)} 
			+
		 \| B \|_{\HS(U,H)}^2
	}{
	( 1 + \| X_{\llcorner u \lrcorner_\theta, u} \|_H^2 )^2
}
\Big)
\| X_{\llcorner u \lrcorner_\theta, u} \|_{ H_\beta }
\Big\|_{\L^p(\P; \R)}
\, 
du
\\
&
\leq 
\int_s^t
\| 
( - A )^{\eta - \beta}
e^{(t-\llcorner u \lrcorner_\theta)A}
\|_{L(H)}
\\
&
\quad
\cdot 
\big\|  
\big(   
\| B \|_{\HS(U,H)}^2  
+ 
2
\| B \|_{ L(U, H_\beta) }
	\| \mathbb{B} \|_{ L(H_\beta, U)}
	+
	\| B \|_{\HS(U,H)}^2
\big)
\| X_{\llcorner u \lrcorner_\theta, u} \|_{H_\beta}
\big\|_{\L^p(\P; \R)}
\, 
du
.
\end{split}
\end{equation}
This and the fact that
$ \| \mathbb{B} \|_{L(H_\beta, U)}
\leq 
| \sup_{ h \in \H } \values_h |^{-\beta} 
\| \mathbb{B} \|_{L(H, U) } $
demonstrate that
for every 
	$ s \in [0,T] $, $ t \in [s,T] $
it holds that
\begin{equation}
\begin{split}
&
\Big\|
\int_s^t
\chi_{\llcorner u \lrcorner_\theta} 
e^{(t-\llcorner u \lrcorner_\theta)A}
\Big( 
\tfrac{ 4  
	\|  \mathbb{B}
	X_{ \llcorner u \lrcorner_\theta, u}  \|_U^2
	X_{ \llcorner u \lrcorner_\theta, u}
}
{(1 + \| X_{ \llcorner u \lrcorner_\theta, u} \|_H^2)^3}
-
\tfrac{
	2 B   
	\mathbb{B}
	X_{ \llcorner u \lrcorner_\theta, u}
	+ 
	\| B \|_{\HS(U,H)}^2
	X_{ \llcorner u \lrcorner_\theta, u} 
}{
	( 1 + \| X_{\llcorner u \lrcorner_\theta, u} \|_H^2 )^2
}
\Big)  
\, 
du
\Big\|_{\L^p(\P; H_\eta)}
\\
&
\leq 
2
\int_s^t
\| 
( - A )^{\eta - \beta}
e^{(t-\llcorner u \lrcorner_\theta)A}
\|_{L(H)}
\\
&
\quad
\cdot
\big(
\| B \|_{ \HS(U,H) }^2  
+
| \sup\nolimits_{h \in \H} \values_h |^{-\beta}
\| B \|_{ \HS(U,H_\beta) }
\| \mathbb{B} \|_{ L(H, U) }
\big)
\| X_{\llcorner u \lrcorner_\theta, u} \|_{\L^{ p}(\P;H_\beta)} 
\, 
du
\\
&
\leq
2
\int_s^t
\| B \|_{ \HS(U,H) }
\| 
( - A )^{\eta - \beta}
e^{(t-\llcorner u \lrcorner_\theta)A}
\|_{L(H)}
\\
&
\quad
\cdot
\big(
 \| B \|_{ \HS(U,H) }  
+
| \sup\nolimits_{h \in \H} \values_h |^{-\beta}
\| B \|_{ \HS(U,H_\beta) } 
\big)
\| X_{ \llcorner u \lrcorner_\theta, u} \|_{\L^{ p}(\P;H_\beta)} 
\, 
du
.
\end{split}
\end{equation}
The
Burkholder-Davis-Gundy type inequality in Lemma~7.7 in Da Prato \& Zabczyk~\cite{dz92}
and 
the fact that
$ \| B \|_{ \HS(U,H)}
\leq | \sup_{ h \in \H } \values_h |^{-\beta} \| B \|_{\HS(U,H_\beta)} $
hence
establish for every
	$ s \in [0,T] $, $ t \in [s,T] $ 
 that
\begin{equation}
	\begin{split}
	\label{eq:intermediate}
		&
		\Big\|
		\int_s^t
		\chi_{\llcorner u \lrcorner_\theta} 
		e^{(t-\llcorner u \lrcorner_\theta)A}
		\Big( 
		\tfrac{ 4  
			\|  \mathbb{B}
			X_{\llcorner u \lrcorner_\theta, u}  \|_U^2
	X_{\llcorner u \lrcorner_\theta, u}	
	}
		{(1 + \| X_{\llcorner u \lrcorner_\theta, u} \|_H^2)^3}
		-
		\tfrac{
			2 B   
			\mathbb{B} 
			X_{\llcorner u \lrcorner_\theta, u}
			+ 
			\| B \|_{\HS(U,H)}^2 
	X_{\llcorner u \lrcorner_\theta, u}	
	}{
		( 1 + \| X_{\llcorner u \lrcorner_\theta, u} \|_H^2 )^2
	}
	\Big) 
	\, 
	du
	\Big\|_{\L^p(\P; H_\eta )}
	\\
	&
	\leq
	2 
	\| B \|_{ \HS(U,H) }
	\big(
	\| B \|_{ \HS(U,H) } 
	+
	| \sup\nolimits_{h \in \H} \values_h |^{-\beta}
	\| B \|_{ \HS(U,H_\beta) } 
	\big)
	\\
	&
	\quad
	\cdot
	\int_s^t   
\| 
( - A )^{\eta - \beta}
e^{(t-\llcorner u \lrcorner_\theta)A}
\|_{L(H)}
	\Big(
	\tfrac{p (p-1)}{2}
	\smallint_{\llcorner u \lrcorner_\theta}^u
	\|
	B  \|_{ \HS(U,H_\beta) }^2
	\, ds 
	\Big)^{\nicefrac{1}{2}}
	\, 
	du
	\\
	&
	\leq
	2
	\sqrt{2}
	p 
\| B \|_{ \HS(U,H_\beta) }^3 
| \sup\nolimits_{h \in \H} \values_h |^{-2\beta} 
	\int_s^t   
\| 
( - A )^{\eta - \beta}
e^{(t-\llcorner u \lrcorner_\theta)A}
\|_{L(H)}
	( u - \llcorner u \lrcorner_\theta )^{ \nicefrac{1}{2} } 
	\, 
	du
	\\
	&
	\leq
2
\sqrt{2}
p  
\| B \|_{ \HS(U,H_\beta) }^3 
| \sup\nolimits_{h \in \H} \values_h |^{-2\beta} 
[ | \theta |_T ]^{ \nicefrac{1}{2} } 
\int_s^t
	(t-\llcorner u \lrcorner_\theta)^{ \beta - \eta }   
\, 
du	 
	.
	\end{split} 
	\end{equation}
	Moreover, note that for every
		$ s \in [0,T] $, $ t \in [s,T] $ 
	 it holds that
	\begin{equation}
	\begin{split}
	&
	\int_s^t
	(t-\llcorner u \lrcorner_\theta)^{ \beta - \eta }    
	\, 
	du  
	\leq  
	\int_s^t
	(t-u)^{ \beta - \eta } 
	\,
	du  
	\leq  
	\tfrac{
		(t-s)^{ 1 + \beta - \eta }    
	}
	{
		1 + \beta - \eta
	}  
	.
	\end{split}
	\end{equation}
	Combining this and~\eqref{eq:intermediate}
	establishes~\eqref{eq:Estimate3}.
The proof of Lemma~\ref{lemma:some_helping_estimates}
is thus completed.
\end{proof}
\begin{lemma}
	\label{lemma:more_regularity}
	Assume Setting~\ref{setting:Strengthened_strong_a_priori_moment_bounds},
	assume that
	$ \beta \leq \gamma $,
	and
	let
	$ p \in [2,\infty) $,   
	$ t \in [0,T] $.
	Then it holds that
	\begin{equation}
	\begin{split}
	\| \O_t \|_{ \L^p( \P; H_\gamma ) }
	&
	\leq 
	3 p  
	\| B \|_{\HS(U,H_\beta)}
	\tfrac{ [ \max \{ T, 1 \} ]^{ \nicefrac{3}{2}  
	} }
	{  1 + 2 \beta - 2 \gamma }
	\big( 
	1
	+
	4 p^2 
	\| B \|_{ \HS(U, H_\beta) }^2
	| \sup\nolimits_{ h \in \H } \values_h |^{-2\beta} 
	\big)
	.
	\end{split}
	\end{equation}
\end{lemma}
\begin{proof}[Proof of Lemma~\ref{lemma:more_regularity}]
	Throughout this proof
	for every $ s \in [0,T] $ 
	let  
	$ X_{s, (\cdot) }(\cdot) 
	=
	( X_{s,u}(\omega) )_{ (u,\omega)\in[s,T]\times\Omega }
	\colon $ 
	$ [s,T] \times \Omega
	\to H_\beta $
	be
	an
	$ ( \f_u )_{ u \in [s,T] } $-adapted
	stochastic process with continuous sample paths
	which satisfies for every
	$ u \in [s,T] $
	that
	$ [ X_{s,u} ]_{\P, \mathcal{B}(H_\beta) }
	=
	\int_s^u B \, dW_\tau $. 
	Lemma~\ref{lemma:Ito_representation}
	(with
	$ X_{s, u} = X_{s, u} $
	for
	$ u \in [s,T] $, $ s \in [0,T] $ 
	in the notation of 
	Lemma~\ref{lemma:Ito_representation})
	implies that
	\begin{equation}
	\begin{split}
	\label{eq:triangle_inequality}
	& 
	\| \O_t \|_{ \L^p( \P; H_\gamma ) } 
	\leq 
	\Big\|
	\int_0^t
	\chi_{ \llcorner u \lrcorner_\theta } 
	\Big(
	\tfrac{
		e^{(t- \llcorner u \lrcorner_\theta)A} 
		B  }{1+\| X_{{\llcorner u \lrcorner_\theta}, u} \|_H^2}
	-
	\tfrac{2 
		e^{(t- \llcorner u \lrcorner_\theta)A} 
		X_{{ \llcorner u \lrcorner_\theta }, u} \langle X_{{ \llcorner u \lrcorner_\theta }, u}, B  (\cdot) \rangle_H}
	{ ( 1 + \| X_{{ \llcorner u \lrcorner_\theta }, u} \|_H^2 )^2 }
	\Big)
	\,
	dW_u
	\Big\|_{L^p(\P; H_\gamma )}
	\\
	&
	+ 
	\Big\|
	\int_0^t
	\chi_{\llcorner u \lrcorner_\theta} 
	e^{(t-\llcorner u \lrcorner_\theta)A}
	\Big[ 
	\tfrac{ 4  
		\|  \mathbb{B}
		X_{\llcorner u \lrcorner_\theta, u}  \|_U^2
		X_{\llcorner u \lrcorner_\theta, u}
	}
	{(1 + \| X_{\llcorner u \lrcorner_\theta, u} \|_H^2)^3}
	-
	\tfrac{
		2 B 
		\mathbb{B}
		X_{\llcorner u \lrcorner_\theta, u}
		+ 
		\| B \|_{\HS(U,H)}^2
		X_{\llcorner u \lrcorner_\theta, u} 
	}{
		( 1 + \| X_{\llcorner u \lrcorner_\theta, u} \|_H^2 )^2
	}
	\Big]
	\, 
	du
	\Big\|_{\L^p(\P; H_\gamma )}
	.
	\end{split}
	\end{equation}
	Moreover, observe that the
	triangle inequality
	proves that
	\begin{equation}
	\begin{split}
	\label{eq:split_triangle}
	&
	\Big\|
	\int_0^t
	\chi_{ \llcorner u \lrcorner_\theta } 
	\Big(
	\tfrac{
		e^{(t- \llcorner u \lrcorner_\theta)A} 
		B  }{1+\| X_{{ \llcorner u \lrcorner_\theta }, u} \|_H^2}
	-
	\tfrac{2 
		e^{(t- \llcorner u \lrcorner_\theta )A} 
		X_{{ \llcorner u \lrcorner_\theta }, u} \langle X_{{ \llcorner u \lrcorner_\theta }, u}, B (\cdot) \rangle_H}
	{ ( 1 + \| X_{{ \llcorner u \lrcorner_\theta }, u} \|_H^2 )^2 }
	\Big)
	\,
	dW_u
	\Big\|_{L^p(\P; H_\gamma )}
	\\
	&
	\leq
	\Big\|
	\int_0^t 
	\chi_{ \llcorner u \lrcorner_\theta } 
	e^{(t- \llcorner u \lrcorner_\theta )A} 
	\,
	\tfrac{B  }{1+\| X_{{ \llcorner u \lrcorner_\theta }, u} \|_H^2}  
	\, 
	dW_u 
	\Big\|_{L^p(\P; H_\gamma )} 
	\\
	&
	\quad
	+
	\Big\|
	\int_0^t
	\chi_{ \llcorner u \lrcorner_\theta } 
	\tfrac{2 
		e^{(t- \llcorner u \lrcorner_\theta )A} 
		X_{{ \llcorner u \lrcorner_\theta }, u} \langle X_{{ \llcorner u \lrcorner_\theta }, u}, B  (\cdot) \rangle_H}
	{ ( 1 + \| X_{{\llcorner u \lrcorner_\theta}, u} \|_H^2 )^2 }
	\,
	dW_u  
	\Big\|_{L^p(\P; H_\gamma )}
	.
	\end{split}
	\end{equation}
	Next note that
	Lemma~\ref{lemma:some_helping_estimates}
	(with 
	$ p = p $,
	$ \rho = \gamma $,
	$ \eta = \gamma $, 
	$ X_{s, \tau} = X_{s, \tau} $
	for 
	$ \tau \in [s,T] $,
	$ s \in [0,T] $ 
	in the notation of Lemma~\ref{lemma:some_helping_estimates})
	shows that
	\begin{equation}
	\begin{split} 
	&
	\Big\|
	\int_0^t 
	\chi_{ \llcorner u \lrcorner_\theta } 
	e^{(t- \llcorner u \lrcorner_\theta)A} 
	\,
	\tfrac{B }{1+\| X_{{\llcorner u \lrcorner_\theta}, u} \|_H^2}  
	\, 
	dW_u
	\Big\|_{ L^p(\P; H_\gamma )}
	\leq  
	\| B \|_{\HS(U,H_\beta)}
	\tfrac{
		p  t^{ \nicefrac{1}{2} + \beta - \gamma } 
	}
	{ \sqrt{ 2 ( 1 + 2 \beta - 2 \gamma ) } }
	\end{split}
	\end{equation}
	\begin{equation}
	\begin{split}
	\label{eq:helping_lemma}
	&
	\Big\|
	\int_0^t
	\chi_{ \llcorner u \lrcorner_\theta } 
	\tfrac{2 
		e^{(t- \llcorner u \lrcorner_\theta )A} 
		X_{{ \llcorner u \lrcorner_\theta }, u} \langle X_{{ \llcorner u \lrcorner_\theta }, u}, B  (\cdot) \rangle_H}
	{ ( 1 + \| X_{{ \llcorner u \lrcorner_\theta }, u} \|_H^2 )^2 }
	\,
	dW_u 
	\Big\|_{L^p(\P; H_\gamma )}
	\\
	&
	\leq 
	2 \sqrt{2} p^3  
	\| B \|_{ \HS(U, H_\beta) }^3
	| \sup\nolimits_{ h \in \H } \values_h |^{-2\beta}
	\tfrac{  
		t^{\nicefrac{1}{2} + \beta - \gamma} 
	}
	{ \sqrt{ 1 + 2 \beta - 2 \gamma } }  
	|\theta|_T
	,
	\end{split}
	\end{equation}
	and
	\begin{equation}
	\begin{split}
	\label{eq:sec_est}
	&
	\Big\|
	\int_0^t
	\chi_{\llcorner u \lrcorner_\theta} 
	e^{(t-\llcorner u \lrcorner_\theta)A}
	\Big( 
	\tfrac{ 4  
		\|  \mathbb{B}
		X_{\llcorner u \lrcorner_\theta, u}  \|_U^2 X_{\llcorner u \lrcorner_\theta, u}  }
	{(1 + \| X_{\llcorner u \lrcorner_\theta, u} \|_H^2)^3}
	-
	\tfrac{
		2 B   
		\mathbb{B}  X_{\llcorner u \lrcorner_\theta, u} 
		+ 
		\| B \|_{\HS(U,H)}^2  X_{\llcorner u \lrcorner_\theta, u} 
	}{
		( 1 + \| X_{\llcorner u \lrcorner_\theta, u} \|_H^2   )^2
	}
	\Big)
	\, 
	du
	\Big\|_{\L^p(\P; H_\gamma)}
	\\
	&
	\leq 
	2
	\sqrt{2}
	p 
	\| B \|_{ \HS(U,H_\beta) }^3 
	| \sup\nolimits_{h \in \H} \values_h |^{-2\beta} 
	\tfrac{  t^{ 1 + \beta - \gamma } 
	}{  1 + \beta - \gamma } 
	[ | \theta |_T ]^{ \nicefrac{1}{2} }
	.
	\end{split}
	\end{equation}
	Next we combine~\eqref{eq:split_triangle}--\eqref{eq:helping_lemma}
	to obtain that
	\begin{equation}
	\begin{split}
	\label{eq:fir_est}
	&
	\Big\|
	\int_0^t
	\chi_{\llcorner u \lrcorner_\theta} 
	\Big(
	\tfrac{
		e^{(t- \llcorner u \lrcorner_\theta )A} 
		B  }{1+\| X_{{ \llcorner u \lrcorner_\theta }, u} \|_H^2}
	-
	\tfrac{2 
		e^{(t- \llcorner u \lrcorner_\theta)A} 
		X_{{\llcorner u \lrcorner_\theta}, u} \langle X_{{\llcorner u \lrcorner_\theta}, u}, B  (\cdot) \rangle_H}
	{ ( 1 + \| X_{{\llcorner u \lrcorner_\theta}, u} \|_H^2 )^2 }
	\Big)
	\,
	dW_u 
	\Big\|_{L^p(\P; H_\gamma )}
	\\
	&
	\leq 
	\| B \|_{\HS(U,H_\beta)}
	\tfrac{ p 	
		t^{ \nicefrac{1}{2} + \beta - \gamma } 
	}
	{ \sqrt{ 2 ( 1 + 2 \beta - 2 \gamma ) } }
	\big( 
	1
	+
	4 p^2 
	\| B \|_{ \HS(U, H_\beta) }^2
	| \sup\nolimits_{ h \in \H } \values_h |^{-2\beta} 
	|\theta|_T
	\big)
	.  
	\end{split}
	\end{equation}
	This, \eqref{eq:triangle_inequality},
	\eqref{eq:sec_est}, the fact that
	$ \sqrt{2} + \tfrac{1}{\sqrt{2}} \leq 3 $,
	the
	fact that
	$ 1 + 2 \beta - 2 \gamma \leq 2 ( 1 + \beta - \gamma ) $,
	and the fact that
	$ \forall \, x \in (0,1]
	\colon \nicefrac{ 1 }{ \sqrt{ x } }
	\leq \nicefrac{ 1 }{ x } $ 
	demonstrate that
	\begin{equation}
	\begin{split}
	&
	\| \O_t \|_{ \L^p( \P; H_\gamma ) } 
	\leq 
	4
	\sqrt{2}
	p  
	\| B \|_{ \HS(U,H_\beta) }^3 
	| \sup\nolimits_{h \in \H} \values_h |^{-2\beta} 
	\tfrac{ [ \max \{ T, 1 \} ]^{ \nicefrac{3}{2}   } }{ 1 + 2 \beta - 2 \gamma }  
	\\
	&
	\quad
	+
	\tfrac{p }{ \sqrt{2} }  
	\| B \|_{\HS(U,H_\beta)}
	\tfrac{ [ \max \{ T, 1 \} ]^{ \nicefrac{3}{2}   } }
	{  1 + 2 \beta - 2 \gamma }
	\big( 
	1
	+
	4 p^2 
	\| B \|_{ \HS(U, H_\beta) }^2
	| \sup\nolimits_{ h \in \H } \values_h |^{-2\beta}  
	\big) 
	\\
	&
	\leq 
	3 p 
	\| B \|_{\HS(U,H_\beta)}
	\tfrac{ [ \max \{ T, 1 \} ]^{ \nicefrac{3}{2}  } }
	{ 1 + 2 \beta - 2 \gamma }
	\big( 
	1
	+
	4 p^2 
	\| B \|_{ \HS(U, H_\beta) }^2
	| \sup\nolimits_{ h \in \H } \values_h |^{-2\beta} 
	\big)
	.
	\end{split}
	\end{equation}
	The proof of Lemma~\ref{lemma:more_regularity}
	is thus completed.
\end{proof}
\subsection[H\"older continuity
of temporal approximations of stochastic convolutions]{H\"older continuity
	of temporally semi-discrete approximations of stochastic convolutions}
\label{subsection:Holder}
In this subsection we
combine 
Lemma~\ref{lemma:Ito_representation}
and
Lemma~\ref{lemma:some_helping_estimates}
to establish in
Lemma~\ref{lemma:noise_diff} 
a temporal regularity property
for the approximation 
process 
$ \O \colon [0,T] \times \Omega \to H_\gamma $
from Setting~\ref{setting:Strengthened_strong_a_priori_moment_bounds}.
\begin{lemma}
	\label{lemma:noise_diff}
	Assume Setting~\ref{setting:Strengthened_strong_a_priori_moment_bounds},
	assume that
	$ \beta \leq \gamma $,
	and let 
	$ p \in [2,\infty) $,   
	$ \rho \in [ 0, \nicefrac{1}{2} + \beta - \gamma ) $. 
	Then it holds for every 
	$ s \in [0,T] $, $ t \in [s,T] $
	that
	\begin{equation}
	\begin{split}
	\| \O_t - \O_{ s }
	\|_{ \L^p( \P; H_\gamma ) }  
	&\leq
	\tfrac{3 p^3 
		\| B \|_{\HS(U,H_\beta)}
		[\max \{ T, 1 \} ]^{ 2
		}
		\max 
		\{
		| \sup\nolimits_{ h \in \H } \values_h |^{- 2 \beta},
		1 \}
		(
		1   
		+  
		8
		\| B \|_{\HS(U, H_\beta)}^2 
		)
	}{ \sqrt{  1 + 2 ( \beta - \gamma - \rho )  } } 
	(t-s)^\rho 
	.
	\end{split}
	\end{equation}
\end{lemma}
\begin{proof}[Proof of Lemma~\ref{lemma:noise_diff}]
	Throughout this proof
	for every $ s \in [0,T] $
	let
	$ X_{s, (\cdot) }(\cdot) 
	=
	( X_{s,t}(\omega) )_{ (t,\omega)\in[s,T]\times\Omega }
	\colon $
	$ [s,T] \times \Omega
	\to H_\beta $
	be
	an
	$ ( \f_t )_{ t \in [s,T] } $-adapted
	stochastic process with continuous sample paths
	which satisfies for every 
	$ t \in [s,T] $
	that
	$ [ X_{s,t} ]_{\P, \mathcal{B}(H_\beta) }
	=
	\int_s^t B \, dW_u $. 
	Lemma~\ref{lemma:Ito_representation}
	(with
	$ X_{s, t} = X_{s, t} $
	for  
	$ t \in [s,T] $,
	$ s \in [0,T] $ 
	in the notation of  
	Lemma~\ref{lemma:Ito_representation})
	and the triangle inequality
	prove for every  
	$ s \in [0,T] $, $ t \in [s,T] $
	that
	\begin{equation}
	\begin{split}
	\label{eq:first_diff} 
	& 
	\| \O_t - \O_{ s }
	\|_{ \L^p( \P; H_\gamma ) }
	\\
	&
	\leq
	\Big\|
	\int_{ s }^t 
	\chi_{ \llcorner u \lrcorner_\theta } 
	e^{(t- \llcorner u \lrcorner_\theta)A}
	\Big[
	\tfrac{ B }{1+\| X_{{\llcorner u \lrcorner_\theta}, u} \|_H^2}
	-
	\tfrac{2 X_{{\llcorner u \lrcorner_\theta}, u} \langle 
		X_{{\llcorner u \lrcorner_\theta}, u},  B  (\cdot) \rangle_H}
	{ ( 1 + \| X_{{\llcorner u \lrcorner_\theta}, u} \|_H^2 )^2 } 
	\Big]
	\, 
	dW_u 
	\Big\|_{L^p( \P; H_\gamma ) }
	\\
	& 
	+  
	\Big\|
	\int_{ s }^t
	\chi_{ \llcorner u \lrcorner_\theta } 
	e^{(t- \llcorner u \lrcorner_\theta )A}
	\Big[ 
	\tfrac{ 4  
		\|  \mathbb{B}
		X_{ \llcorner u \lrcorner_\theta, u}  \|_U^2
		X_{\llcorner u \lrcorner_\theta, u} }
	{(1 + \| X_{\llcorner u \lrcorner_\theta, u} \|_H^2)^3}
	-
	\tfrac{
		2 B  
		\mathbb{B}
		X_{\llcorner u \lrcorner_\theta, u}
		+ 
		\| B \|_{\HS(U,H)}^2
		X_{\llcorner u \lrcorner_\theta, u} 
	}{
		( 1 + \| X_{\llcorner u \lrcorner_\theta, u} \|_H^2 )^2
	}
	\Big]  
	\, 
	du
	\Big\|_{\L^p( \P; H_\gamma ) }
	\\
	&
	+
	\Big\|
	\int_0^{ s } 
	\chi_{ \llcorner u \lrcorner_\theta } 
	\big[
	e^{(t- \llcorner u \lrcorner_\theta)A}
	-
	e^{(s - \llcorner u \lrcorner_\theta )A}
	\big]
	\Big[
	\tfrac{ B }{1+\| X_{{\llcorner u \lrcorner_\theta}, u} \|_H^2}
	-
	\tfrac{2 X_{{\llcorner u \lrcorner_\theta}, u} \langle 
		X_{{\llcorner u \lrcorner_\theta}, u},  B  (\cdot) \rangle_H}
	{ ( 1 + \| X_{{\llcorner u \lrcorner_\theta}, u} \|_H^2 )^2 } 
	\Big]
	\, 
	dW_u 
	\Big\|_{L^p( \P; H_\gamma ) }
	\\
	& 
	+  
	\int_0^{ s }
	\Big\|
	\big[
	e^{(t- \llcorner u \lrcorner_\theta)A}
	-
	e^{(s - \llcorner u \lrcorner_\theta)A}
	\big]
	\Big[ 
	\tfrac{ 4  
		\|  \mathbb{B}
		X_{\llcorner u \lrcorner_\theta, u}  \|_U^2
		X_{\llcorner u \lrcorner_\theta, u}
	}
	{(1 + \| X_{\llcorner u \lrcorner_\theta, u} \|_H^2)^3}
	-
	\tfrac{
		2 B  
		\mathbb{B}
		X_{\llcorner u \lrcorner_\theta, u}
		+ 
		\| B \|_{\HS(U,H)}^2 
		X_{\llcorner u \lrcorner_\theta, u}
	}{
		( 1 + \| X_{\llcorner u \lrcorner_\theta, u} \|_H^2 )^2
	}
	\Big]
	\Big\|_{\L^p( \P; H_\gamma ) }
	\, 
	du 
	.
	\end{split}
	\end{equation}
	Furthermore, observe that the triangle inequality implies for every 
	$ s \in [0,T] $, $ t \in [s,T] $ 
	that
	\begin{equation} 
	\begin{split} 
	\label{eq:some_triangle}
	&
	\Big\|
	\int_{ s }^t 
	\chi_{ \llcorner u \lrcorner_\theta } 
	e^{(t- \llcorner u \lrcorner_\theta )A}
	\Big[
	\tfrac{ B }{1+\| X_{{\llcorner u \lrcorner_\theta}, u} \|_H^2}
	-
	\tfrac{2 X_{{\llcorner u \lrcorner_\theta}, u} \langle 
		X_{{\llcorner u \lrcorner_\theta}, u},  B  (\cdot) \rangle_H}
	{ ( 1 + \| X_{{\llcorner u \lrcorner_\theta}, u} \|_H^2 )^2 } 
	\Big]
	\, 
	dW_u 
	\Big\|_{L^p( \P; H_\gamma ) }
	\\
	&
	\leq
	\Big\|
	\int_{ s }^t 
	\chi_{ \llcorner u \lrcorner_\theta } 
	e^{(t- \llcorner u \lrcorner_\theta )A} 
	\tfrac{ B }{1+\| X_{{ \llcorner u \lrcorner_\theta }, u} \|_H^2} 
	\, 
	dW_u 
	\Big\|_{L^p( \P; H_\gamma ) }
	\\
	&
	\quad
	+
	\Big\|
	\int_{ s }^t 
	\chi_{ \llcorner u \lrcorner_\theta }  
	e^{(t- \llcorner u \lrcorner_\theta )A} 
	\tfrac{2 X_{{ \llcorner u \lrcorner_\theta }, u} \langle 
		X_{{ \llcorner u \lrcorner_\theta }, u},  B  (\cdot) \rangle_H}
	{ ( 1 + \| X_{{ \llcorner u \lrcorner_\theta }, u} \|_H^2 )^2 }  
	\, 
	dW_u 
	\Big\|_{L^p( \P; H_\gamma ) }
	.
	\end{split}
	\end{equation}
	Next note that
	Lemma~\ref{lemma:some_helping_estimates}
	(with 
	$ p = p $,
	$ \rho = \gamma $,
	$ \eta = \gamma $,
	$ X_{s, t} = X_{s, t} $
	for  
	$ t \in [s,T] $,
	$ s \in [0,T] $ 
	in the notation of Lemma~\ref{lemma:some_helping_estimates})
	shows for every   
	$ s \in [0,T] $, $ t \in [s,T] $ 
	that
	\begin{equation}
	\begin{split} 
	\label{eq:est1}
	&
	\Big\|
	\int_{ s }^t 
	\chi_{ \llcorner u \lrcorner_\theta } 
	e^{(t- \llcorner u \lrcorner_\theta )A} 
	\,
	\tfrac{B }{1+\| X_{{ \llcorner u \lrcorner_\theta }, u} \|_H^2}  
	\, 
	dW_u
	\Big\|_{L^p(\P; H_\gamma)}
	%
	%
	\leq
	\tfrac{ p  
		\| B \|_{\HS(U,H_\beta)} }
	{ \sqrt{ 2 ( 1 + 2 \beta -2 \gamma ) }  }  
	( t - s )^{ \nicefrac{1}{2} + \beta - \gamma } 
	,
	\end{split}
	\end{equation}
	\begin{equation}
	\begin{split}
	\label{eq:est2}
	&
	\Big\|
	\int_{ s }^t 
	\chi_{ \llcorner u \lrcorner_\theta } 
	e^{(t- \llcorner u \lrcorner_\theta )A} 
	\,
	\tfrac{2 X_{{ \llcorner u \lrcorner_\theta }, u} \langle X_{{ \llcorner u \lrcorner_\theta }, u}, B 
		\, 
		dW_u  \rangle_H}
	{ ( 1 + \| X_{{ \llcorner u \lrcorner_\theta }, u} \|_H^2 )^2 } 
	\Big \|_{L^p(\P; H_\gamma)}
	\\
	&	 
	\leq  
	\tfrac{
		2 \sqrt{2} p^3  
		\| B \|_{ \HS( U, H_\beta ) }^3
		| \sup_{ h \in \H } \values_h |^{-2\beta}
	}
	{
		\sqrt{ 1 + 2 \beta - 2 \gamma }
	}  
	| \theta |_T   
	(t-s)^{ \nicefrac{1}{2} + \beta - \gamma } 
	,
	\end{split}
	\end{equation}
	and
	\begin{equation}
	\begin{split}
	\label{eq:est3}
	&
	\Big\|
	\int_s^t
	\chi_{ \llcorner u \lrcorner_\theta } 
	e^{(t- \llcorner u \lrcorner_\theta)A}
	\Big[ 
	\tfrac{ 4  
		\|  \mathbb{B}
		X_{\llcorner u \lrcorner_\theta, u}  \|_U^2
		X_{\llcorner u \lrcorner_\theta, u} 
	}
	{(1 + \| X_{\llcorner u \lrcorner_\theta, u} \|_H^2)^3}
	-
	\tfrac{
		2 B  
		\mathbb{B}
		X_{\llcorner u \lrcorner_\theta, u} 
		+ 
		\| B  \|_{\HS(U,H)}^2 
		X_{\llcorner u \lrcorner_\theta, u} 
	}{
		( 1 + \| X_{\llcorner u \lrcorner_\theta, u} \|_H^2 )^2
	}
	\Big]
	\, 
	du
	\Big\|_{\L^p(\P; H_\gamma)}
	\\
	&
	\leq 
	\tfrac{ 2 \sqrt{2} p  
	}{ 1 + \beta - \gamma  } 
	\| B \|_{\HS(U, H_\beta)}^3
	| \sup\nolimits_{ h \in \H } \values_h |^{-2\beta} 
	[ | \theta |_T ]^{\nicefrac{1}{2} } 
	(t-s)^{1 + \beta -  \gamma} 
	.
	\end{split}
	\end{equation}
	Moreover,
	observe that 
	the  
	Burkholder-Davis-Gundy type inequality in Lemma~7.7 in Da Prato \& Zabczyk~\cite{dz92}
	assures
	for every 
	$ s \in [0,T] $, $ t \in [s,T] $ 
	that
	\begin{equation}
	\begin{split}
	& 
	\Big\|
	\int_0^{ s } 
	\chi_{ \llcorner u \lrcorner_\theta } 
	\big[
	e^{(t-\llcorner u \lrcorner_\theta)A}
	-
	e^{( s - \llcorner u \lrcorner_\theta)A}
	\big]
	\Big[
	\tfrac{ B }{1+\| X_{{\llcorner u \lrcorner_\theta}, u} \|_H^2}
	-
	\tfrac{2 X_{{\llcorner u \lrcorner_\theta}, u} \langle 
		X_{{\llcorner u \lrcorner_\theta}, u},  B  (\cdot) \rangle_H}
	{ ( 1 + \| X_{{\llcorner u \lrcorner_\theta}, u} \|_H^2 )^2 } 
	\Big]
	\, 
	dW_u 
	\Big\|_{L^p( \P; H_\gamma ) }^2
	\\
	&
	\leq 
	\tfrac{ p^2 }{ 2 }
	\int_0^{ s } 
	\Big\|  
	e^{( s - \llcorner u \lrcorner_\theta)A}
	(
	e^{(t-s)A}
	-
	\operatorname{Id}_H
	)
	\Big[
	\tfrac{ B }{1+\| X_{\llcorner u \lrcorner_\theta, u} \|_H^2}
	-
	\tfrac{2 X_{\llcorner u \lrcorner_\theta, u} \langle 
		X_{\llcorner u \lrcorner_\theta, u},  B  (\cdot) \rangle_H}
	{ ( 1 + \| X_{\llcorner u \lrcorner_\theta, u} \|_H^2 )^2 } 
	\Big] 
	\Big\|_{\L^p( \P; \HS(U, H_\gamma ) ) }^2
	\,
	du
	\\
	&
	\leq
	\tfrac{p^2}{2}
	\int_0^{ s }   
	\big\|
	(-A)^{ \gamma + \rho - \beta  } e^{(s - \llcorner u \lrcorner_\theta)A}
	\big\|_{L(H)}^2
	\big\|
	(-A)^{ - \rho }
	( 
	e^{(t-s)A}
	-
	\operatorname{Id}_H
	) 
	\big\|_{L(H)}^2
	\\
	&
	\quad 
	\cdot
	\Big\|
	\tfrac{ B }{1+\| X_{\llcorner u \lrcorner_\theta, u} \|_H^2}
	-
	\tfrac{2 X_{\llcorner u \lrcorner_\theta, u} \langle 
		X_{\llcorner u \lrcorner_\theta, u},  B  (\cdot) \rangle_H}
	{ ( 1 + \| X_{\llcorner u \lrcorner_\theta, u} \|_H^2 )^2 }  
	\Big\|_{\L^p( \P; \HS(U, H_\beta ) ) }^2
	\,
	du
	\\
	&
	\leq 
	\tfrac{p^2  
	}{2}
	\int_0^{ s }   
	\big\|
	(-A)^{ \gamma + \rho - \beta  } e^{(s - \llcorner u \lrcorner_\theta)A}
	\big\|_{L(H)}^2
	(t-s)^{ 2 \rho } 
	\\
	&
	\quad
	\cdot
	\Big\|
	\tfrac{ 
		\| B \|_{\HS(U, H_\beta)} }{1+\| X_{\llcorner u \lrcorner_\theta, u} \|_H^2}
	+
	\tfrac{2 
		\| X_{\llcorner u \lrcorner_\theta, u} \langle 
		X_{\llcorner u \lrcorner_\theta, u},  B  (\cdot) \rangle_H 
		\|_{\HS(U,H_\beta)}
	}
	{ ( 1 + \| X_{\llcorner u \lrcorner_\theta, u} \|_H^2 )^2 }  
	\Big\|_{\L^p( \P; \R ) }^2 
	\,
	du 
	\\
	&
	\leq
	\tfrac{p^2  
	}{2}
	\int_0^{ s }   
	\big\|
	(-A)^{ \gamma + \rho - \beta  } e^{(s - \llcorner u \lrcorner_\theta)A}
	\big\|_{L(H)}^2
	(t-s)^{ 2 \rho }
	\\
	&
	\quad
	\cdot
	\Big\| 
	\| B \|_{\HS(U, H_\beta)} 
	+
	\tfrac{2 
		\| X_{\llcorner u \lrcorner_\theta, u} \|_{H_\beta} 
		\| 
		X_{\llcorner u \lrcorner_\theta, u} 
		\|_H 
		\| B  
		\|_{\HS(U, H )} 
	}
	{ ( 1 + \| X_{\llcorner u \lrcorner_\theta, u} \|_H^2 )^2 }  
	\Big\|_{ \L^p(\P; \R) }^2
	\,
	du
	.
	\end{split} 
	\end{equation} 
	Furthermore, note that the fact that
	$ \gamma + \rho - \beta  
	\in 
	[0, \nicefrac{1}{2} ) $
	and the 
	Burkholder-Davis-Gundy type inequality in Lemma~7.7 in Da Prato \& Zabczyk~\cite{dz92}
	hence imply for every 
	$ s \in [0,T] $, $ t \in [s,T] $
	that
	\begin{equation}
	\begin{split}
	\label{eq:BDGEstimated}
	&
	\Big\|
	\int_0^{ s } 
	\chi_{ \llcorner u \lrcorner_\theta } 
	\big[
	e^{(t- \llcorner u \lrcorner_\theta )A}
	-
	e^{(s - \llcorner u \lrcorner_\theta )A}
	\big]
	\Big[
	\tfrac{ B }{1+\| X_{{ \llcorner u \lrcorner_\theta }, u} \|_H^2}
	-
	\tfrac{2 X_{{ \llcorner u \lrcorner_\theta }, u} \langle 
		X_{{ \llcorner u \lrcorner_\theta }, u},  B  (\cdot) \rangle_H}
	{ ( 1 + \| X_{{ \llcorner u \lrcorner_\theta }, u} \|_H^2 )^2 } 
	\Big]
	\, 
	dW_u 
	\Big\|_{L^p( \P; H_\gamma ) }^2
	\\
	&
	\leq 
	\tfrac{p^2  
	}{2}
	\int_0^{ s }
	\big\|
	(-A)^{ \gamma + \rho - \beta  } e^{(s - \llcorner u \lrcorner_\theta)A}
	\big\|_{L(H)}^2
	(t-s)^{ 2 \rho }
	%
	%
	\big[ 
	\| B \|_{\HS(U, H_\beta)} 
	+ 
	\| X_{\llcorner u \lrcorner_\theta, u} \|_{ \L^{ p}( \P;H_\beta)}  
	\| B  
	\|_{\HS(U, H )} 
	\big]^2
	\,
	du
	\\
	&
	\leq 
	\tfrac{p^2  
	}{2}
	\int_0^{ s }  
	\big\|
	(-A)^{ \gamma + \rho - \beta  } e^{(s - \llcorner u \lrcorner_\theta)A}
	\big\|_{L(H)}^2
	(t-s)^{ 2 \rho }
	\\
	&
	\quad
	\cdot
	\Big[ 
	\| B \|_{\HS(U, H_\beta)} 
	+ 
	\Big[
	\tfrac{ p ( p-1 ) }{2} 
	\int_{ \llcorner u \lrcorner_\theta }^u 
	\| B \|_{ \HS(U, H_\beta) }^2
	\,
	dw
	\Big]^{ \nicefrac{1}{2} } 
	\| B  
	\|_{\HS(U, H )} 
	\Big]^2
	\, du
	\\
	&
	\leq 
	\tfrac{p^2  
	}{ 2 }
	\int_0^{ s }     
	\big\|
	(-A)^{ \gamma + \rho - \beta  } e^{(s - \llcorner u \lrcorner_\theta)A}
	\big\|_{L(H)}^2
	(t-s)^{ 2 \rho }
	\\
	&
	\quad
	\cdot
	\big[
	\| B \|_{\HS(U, H_\beta)} 
	+
	p 
	\| B \|_{ \HS(U, H_\beta) }
	\| B \|_{ \HS(U, H ) } 
	( u - \llcorner u \lrcorner_\theta )^{ \nicefrac{1}{2} }
	\big]^2
	\, du
	\\
	&
	\leq 
	\tfrac{p^2   
	}{ 2 }
	%
	(t-s)^{ 2 \rho }
	\| B \|_{\HS(U, H_\beta)}^2
	\big( 
	1
	+
	p 
	\sqrt{T} 
	\| B \|_{ \HS(U, H ) } 
	\big)^2
	\int_0^{ s }   
	\tfrac{ 
		1}
	{   
		( s - \llcorner u \lrcorner_\theta)^{ 2 ( \gamma + \rho - \beta )  }
	} 
	\, du 
	.
	\end{split} 
	\end{equation}
	In addition, observe that for every 
	$ s \in [0,T] $, $ t \in [s,T] $ 
	it holds that
	\begin{equation}
	\begin{split}
	&
	\int_0^{ s }   
	\tfrac{ 
		du }
	{   
		(s - \llcorner u \lrcorner_\theta)^{ 2 ( \gamma + \rho - \beta )  }
	}
	\leq 
	\int_0^{ s }   
	\tfrac{ 
		du }
	{   
		(s -  u)^{ 2 ( \gamma + \rho - \beta )  }
	} 
	%
	= 
	\tfrac{ s^{ 1 + 2 ( \beta - \gamma - \rho ) } }
	{ 1 + 2 ( \beta - \gamma - \rho ) } 
	\leq
	\tfrac{ T^{ 1 + 2 (\beta - \gamma - \rho ) } }{ 1 + 2 ( \beta - \gamma - \rho ) } 
	.
	\end{split}
	\end{equation}
	Combining this and~\eqref{eq:BDGEstimated}
	demonstrates that for every 
	$ s \in [0,T] $, $ t \in [s,T] $ 
	it holds that
	\begin{equation}
	\begin{split}
	\label{eq:Diff1}
	&
	\Big\|
	\int_0^{ s } 
	\chi_{ \llcorner u \lrcorner_\theta } 
	\big[
	e^{(t- \llcorner u \lrcorner_\theta )A}
	-
	e^{(s - \llcorner u \lrcorner_\theta )A}
	\big]
	\Big[
	\tfrac{ B }{1+\| X_{{ \llcorner u \lrcorner_\theta }, u} \|_H^2}
	-
	\tfrac{2 X_{{ \llcorner u \lrcorner_\theta}, u} \langle 
		X_{{\llcorner u \lrcorner_\theta}, u},  B  (\cdot) \rangle_H}
	{ ( 1 + \| X_{{\llcorner u \lrcorner_\theta}, u} \|_H^2 )^2 } 
	\Big]
	\, 
	dW_u 
	\Big\|_{L^p( \P; H_\gamma ) }
	\\
	&
	\leq
	\tfrac{ p   
	}{ \sqrt{ 2 } }
	(t-s)^{ \rho }
	\| B \|_{\HS(U, H_\beta)} 
	\big( 
	1
	+
	p 
	\sqrt{T} 
	\| B \|_{ \HS(U, H ) }  
	\big) 
	\tfrac{  
		T^{ \nicefrac{1}{2} + \beta - \gamma - \rho } 
	}{
		\sqrt{ 1 + 2 ( \beta - \gamma - \rho ) } } 
	.
	\end{split} 
	\end{equation}
	In the next step we observe that for every 
	$ s \in [0,T] $, $ t \in [s,T] $ 
	it holds that
	\begin{equation}
	\begin{split}
	& 
	\int_0^{ s }
	\Big\|
	\big[
	e^{(t-\llcorner u \lrcorner_\theta)A}
	-
	e^{(s - \llcorner u \lrcorner_\theta)A}
	\big]
	\Big[
	\tfrac{ 4  
		\|  \mathbb{B}
		X_{ \llcorner u \lrcorner_\theta, u}  \|_U^2
		X_{ \llcorner u \lrcorner_\theta, u}
	}
	{(1 + \| X_{ \llcorner u \lrcorner_\theta, u} \|_H^2)^3}
	-
	\tfrac{
		2 B  
		\mathbb{B}
		X_{ \llcorner u \lrcorner_\theta, u}
		+ 
		\| B \|_{\HS(U,H)}^2 
		X_{ \llcorner u \lrcorner_\theta, u}
	}{
		( 1 + \| X_{ \llcorner u \lrcorner_\theta, u} \|_H^2 )^2
	}
	\Big]  
	\Big\|_{\L^p( \P; H_\gamma ) }
	\, 
	du
	\\
	&
	\leq
	\int_0^{ s }
	\Big\|
	\|
	(-A)^{ \gamma + \rho - \beta  } e^{(s - \llcorner u \lrcorner_\theta)A}
	\|_{L(H)}
	\|
	(-A)^{ - \rho }
	( 
	e^{(t-s)A}
	-
	\operatorname{Id}_H
	) 
	\|_{L(H)}
	\\
	&
	\quad
	\cdot
	\Big\|
	\tfrac{ 4  
		\| \mathbb{B}
		X_{\llcorner u \lrcorner_\theta, u}  \|_U^2
		X_{\llcorner u \lrcorner_\theta, u}
	}
	{(1 + \| X_{\llcorner u \lrcorner_\theta, u} \|_H^2)^3}
	-
	\tfrac{
		2 B  
		\mathbb{B}
		X_{\llcorner u \lrcorner_\theta, u}
		+ 
		\| B \|_{\HS(U,H)}^2 
		X_{\llcorner u \lrcorner_\theta, u}
	}{
		( 1 + \| X_{\llcorner u \lrcorner_\theta, u} \|_H^2 )^2
	}
	\Big\|_{H_\beta}   
	\Big\|_{\L^p( \P; \R ) }
	\,
	du
	\\
	&
	\leq  
	\int_0^{ s }
	\|
	(-A)^{ \gamma + \rho - \beta  } e^{(s - \llcorner u \lrcorner_\theta)A}
	\|_{L(H)}
	(t-s)^\rho 
	\\
	&
	\quad
	\cdot
	\Big\| 
	\Big[
	\tfrac{ 4  
		\|  \mathbb{B}
		X_{\llcorner u \lrcorner_\theta, u}  \|_U^2 
	}
	{(1 + \| X_{\llcorner u \lrcorner_\theta, u} \|_H^2)^3}
	+
	\tfrac{
		2 
		\| B  
		\mathbb{B} \|_{L(H_\beta, H_\beta)} 
		+ 
		\| B \|_{\HS(U,H)}^2  
	}{
		( 1 + \| X_{\llcorner u \lrcorner_\theta, u} \|_H^2 )^2
	}    
	\Big]
	\| X_{ \llcorner u \lrcorner_\theta, u} \|_{ H_\beta }
	\Big\|_{\L^p( \P; \R ) }
	\,
	du
	\\
	&
	\leq   
	\int_0^{ s } 
	\|
	(-A)^{ \gamma + \rho - \beta  } e^{(s - \llcorner u \lrcorner_\theta )A}
	\|_{L(H)}
	%
	(t-s)^{ \rho }
	\\
	&
	\quad
	\cdot
	\Big\| 
	\Big[
	\tfrac{ 4  
		\|  \mathbb{B} \|_{L(H, U)}^2
		\| X_{\llcorner u \lrcorner_\theta, u}  \|_H^2 
	}
	{(1 + \| X_{\llcorner u \lrcorner_\theta, u} \|_H^2)^3}
	+
	\tfrac{
		2 
		\| B \|_{L(U,H_\beta)}  
		\| \mathbb{B} \|_{L(H_\beta, U)} 
		+ 
		\| B \|_{\HS(U,H)}^2  
	}{
		( 1 + \| X_{\llcorner u \lrcorner_\theta, u} \|_H^2 )^2
	}    
	\Big]
	\| X_{ \llcorner u \lrcorner_\theta, u} \|_{ H_\beta }
	\Big\|_{\L^p( \P; \R ) }
	\,
	du
	.
	\end{split} 
	\end{equation} 
	\sloppy
	This, the fact that
	$ \| \mathbb{B} \|_{L(H_\beta, U)}
	\leq 
	| \sup_{ h \in \H } \values_h |^{-\beta} 
	\| \mathbb{B} \|_{L(H, U) } $, 
	and the fact that
	$ \| B \|_{ \HS(U,H)}
	\leq | \sup_{ h \in \H } \values_h |^{-\beta} \| B \|_{\HS(U,H_\beta)} $
	show for every  
	$ s \in [0,T] $, $ t \in [s,T] $ 
	that
	\begin{equation}
	\begin{split}
	& 
	\int_0^{ s } 
	\Big\|
	\big[
	e^{(t-\llcorner u \lrcorner_\theta)A}
	-
	e^{(s - \llcorner u \lrcorner_\theta)A}
	\big]
	\Big[
	\tfrac{ 4  
		\|  \mathbb{B}
		X_{\llcorner u \lrcorner_\theta, u}  \|_U^2
		X_{\llcorner u \lrcorner_\theta, u}
	}
	{(1 + \| X_{\llcorner u \lrcorner_\theta, u} \|_H^2)^3}
	-
	\tfrac{
		2 B  
		\mathbb{B}
		X_{\llcorner u \lrcorner_\theta, u}
		+ 
		\| B \|_{\HS(U,H)}^2 
		X_{\llcorner u \lrcorner_\theta, u}
	}{
		( 1 + \| X_{\llcorner u \lrcorner_\theta, u} \|_H^2 )^2
	}
	\Big] 
	\Big\|_{\L^p( \P; H_\gamma ) }
	\, 
	du  
	\\
	&
	\leq   
	%
	(t-s)^{ \rho }
	\int_0^{ s }
	\|
	(-A)^{ \gamma + \rho - \beta  } e^{(s - \llcorner u \lrcorner_\theta)A}
	\|_{L(H)}  
	\\
	&
	\quad
	\cdot
	\big\| 
	\big( 
	\|  \mathbb{B} \|_{L(H, U)}^2 
	+ 
	2 
	\| B \|_{L(U,H_\beta)}  
	\| \mathbb{B} \|_{L(H_\beta, U)} 
	+ 
	\| B \|_{\HS(U,H)}^2  
	\big)
	\| X_{ \llcorner u \lrcorner_\theta, u} \|_{ H_\beta }
	\big\|_{\L^p( \P; \R ) }
	\,
	du
	\\
	&
	\leq   
	%
	(t-s)^{ \rho }
	\int_0^{ s }
	\|
	(-A)^{ \gamma + \rho - \beta  } e^{(s - \llcorner u \lrcorner_\theta)A}
	\|_{L(H)} 
	\\
	&
	\quad
	\cdot
	\big( 
	\|  \mathbb{B} \|_{L(H, U)}^2   
	+ 
	2 
	| \sup\nolimits_{ h \in \H } \values_h |^{- \beta}
	\| B \|_{\HS(U,H_\beta )}  
	\| \mathbb{B} \|_{L(H, U)}  
	+ 
	\| B \|_{\HS(U,H)}^2 
	\big)
	\| X_{ \llcorner u \lrcorner_\theta, u} 
	\|_{\L^p( \P; H_\beta ) } 
	\,
	du
	\\
	&
	\leq 
	4
	%
	(t-s)^{ \rho }
	\|  B \|_{\HS(U, H_\beta )}^2
	| \sup\nolimits_{ h \in \H } \values_h |^{- 2 \beta}  
	\\
	&
	\quad
	\cdot
	\int_0^{ s }
	\| X_{ \llcorner u \lrcorner_\theta, u} 
	\|_{\L^p( \P; H_\beta ) } 
	\|
	(-A)^{ \gamma + \rho - \beta  } e^{( s - \llcorner u \lrcorner_\theta)A}
	\|_{L(H)}
	\, du
	.
	\end{split} 
	\end{equation}
	The
	Burkholder-Davis-Gundy type inequality in Lemma~7.7 in Da Prato \& Zabczyk~\cite{dz92} 
	therefore proves 
	for every 
	$ s \in [0,T] $, $ t \in [s,T] $
	that
	\begin{equation}
	\begin{split}
	\label{eq:Needed2}
	& 
	\int_0^{ s }
	\Big\| 
	\big[
	e^{(t-\llcorner u \lrcorner_\theta)A}
	-
	e^{( s - \llcorner u \lrcorner_\theta)A}
	\big]
	\Big[
	\tfrac{ 4  
		\|  \mathbb{B}
		X_{\llcorner u \lrcorner_\theta, u}  \|_U^2
		X_{\llcorner u \lrcorner_\theta, u}
	}
	{(1 + \| X_{\llcorner u \lrcorner_\theta, u} \|_H^2)^3}
	-
	\tfrac{
		2 B  
		\mathbb{B}
		X_{\llcorner u \lrcorner_\theta, u}
		+ 
		\| B \|_{\HS(U,H)}^2 
		X_{\llcorner u \lrcorner_\theta, u}
	}{
		( 1 + \| X_{\llcorner u \lrcorner_\theta, u} \|_H^2 )^2
	}
	\Big]
	\Big\|_{\L^p( \P; H_\gamma ) }
	\, 
	du  
	\\
	&
	\leq
	4 
	%
	(t-s)^{ \rho }
	\|  B \|_{\HS(U, H_\beta )}^2 
	| \sup\nolimits_{ h \in \H } \values_h |^{- 2 \beta} 
	\\
	&
	\quad
	\cdot
	\int_0^{ s }
	\Big[
	\tfrac{ p ( p - 1 ) }{ 2 }
	\int_{ \llcorner u \lrcorner_\theta }^u 
	\| B \|_{\HS(U, H_\beta)}^2 \, dw
	\Big]^{ \nicefrac{1}{2} }
	\|
	(-A)^{ \gamma + \rho - \beta  } e^{( s - \llcorner u \lrcorner_\theta)A}
	\|_{L(H)}
	\, du  
	\\
	&
	\leq
	2
	\sqrt{2}
	p  
	(t-s)^{ \rho }  
	\| B \|_{\HS(U, H_\beta)}^3
	|
	\sup\nolimits_{ h \in \H } 
	\values_h 
	|^{ - 2 \beta }
	\\
	&
	\quad
	\cdot
	\int_0^{ s }
	( u - \llcorner u \lrcorner_\theta  )^{ \nicefrac{1}{2} } 
	\|
	(-A)^{ \gamma + \rho - \beta  } e^{(s - \llcorner u \lrcorner_\theta)A}
	\|_{L(H)}
	\, du
	\\
	&
	\leq 
	2
	\sqrt{2}
	p  
	[ | \theta |_T ]^{ \nicefrac{1}{2} } 
	(t-s)^{ \rho }
	\| B \|_{\HS(U, H_\beta)}^3
	| \sup\nolimits_{ h \in \H } \values_h |^{- 2 \beta}  
	\int_0^{ s }
	\tfrac{ du }
	{
		( s - \llcorner u \lrcorner_\theta)^{ \gamma + \rho - \beta } 
	}
	.
	\end{split}
	\end{equation}
	Moreover, note that for every 
	$ s \in [0,T] $ 
	it holds that
	\begin{equation}
	\begin{split}
	&
	\int_0^{ s }   
	(s - \llcorner u \lrcorner_\theta)^{  \beta -   \gamma - \rho }
	\, du
	\leq 
	\int_0^{ s }   
	( s - u )^{  \beta -  \gamma - \rho  }
	\, du 
	= 
	\tfrac{ s^{ 1 + \beta - \gamma - \rho } }
	{ 1 + \beta - \gamma - \rho }  
	\leq
	\tfrac{ T^{ 1 + \beta - \gamma - \rho} }{  1 + \beta - \gamma - \rho } 
	.
	\end{split}
	\end{equation}
	Estimate~\eqref{eq:Needed2}
	hence establishes
	for every 
	$ s \in [0,T] $, $ t \in [s,T] $ 
	that
	\begin{equation}
	\begin{split}
	\label{eq:Diff2}
	& 
	\int_0^{ s }
	\Big\|
	\big[
	e^{(t- \llcorner u \lrcorner_\theta)A}
	-
	e^{(s - \llcorner u \lrcorner_\theta)A}
	\big]
	\Big[ 
	\tfrac{ 4  
		\|  \mathbb{B}
		X_{\llcorner u \lrcorner_\theta, u}  \|_U^2
		X_{\llcorner u \lrcorner_\theta, u}
	}
	{(1 + \| X_{\llcorner u \lrcorner_\theta, u} \|_H^2)^3}
	-
	\tfrac{
		2 B  
		\mathbb{B} 
		X_{\llcorner u \lrcorner_\theta, u}
		+ 
		\| B \|_{\HS(U,H)}^2 
		X_{\llcorner u \lrcorner_\theta, u}
	}{
		( 1 + \| X_{\llcorner u \lrcorner_\theta, u} \|_H^2 )^2
	}
	\Big] 
	\Big\|_{\L^p( \P; H_\gamma ) } 
	\,
	du  
	\\
	&
	\leq
	2
	\sqrt{2}
	p   
	[ | \theta |_T ]^{ \nicefrac{1}{2} } 
	(t-s)^{ \rho }
	\| B \|_{\HS(U, H_\beta)}^3 
	| \sup\nolimits_{ h \in \H } \values_h |^{- 2 \beta}  
	\tfrac{  T^{ 1 + \beta - \gamma - \rho} }{ 1 + \beta - \gamma - \rho } 
	.
	\end{split}
	\end{equation}
	Combining~\eqref{eq:first_diff}--\eqref{eq:est3},
	the fact that
	$ \| B \|_{ \HS(U,H)}
	\leq | \sup_{ h \in \H } \values_h |^{-\beta} \| B \|_{\HS(U,H_\beta)} $,
	\eqref{eq:Diff1}, 
	\eqref{eq:Diff2},
	the fact that
	$ \sqrt{ 1 + 2 \beta - 2 \gamma } 
	\leq 1 + \beta - \gamma $,
	the fact that
	$ \sqrt{ 1 + 2 ( \beta - \gamma - \rho ) } 
	\leq 1 + \beta - \gamma - \rho $,
	and the fact that
	$ \frac{ 1 }{ \sqrt{2} } + \frac{ p }{ \sqrt{ 2 } } \| B \|_{\HS(U, H_\beta)}
	+
	2 \sqrt{2} \| B \|_{ \HS(U,H_\beta)}^2 
	\leq
	2 ( 1 + \| B \|_{ \HS(U,H_\beta)}^2  ) p $
	implies that for every 
	$ s \in [0,T] $, $ t \in [s,T] $
	it holds that
	\begin{equation}
	\begin{split}
	&
	( \E[\| \O_t - \O_{ s }
	\|_{ H_\gamma }^p ] )^{ \nicefrac{1}{p}}
	\leq
	\tfrac{p   
		\| B \|_{\HS(U,H_\beta)} 
		\max \{ 
		| \sup\nolimits_{ h \in \H } \values_h |^{-2\beta},
		1 \}
	}{ \sqrt{ 2( 1 + 2 \beta - 2 \gamma ) } }
	\\
	&
	\quad 
	\cdot
	\big( 
	(t-s)^{ \nicefrac{1}{2} + \beta - \gamma } 
	+ 
	4 p^2
	\| B \|_{ \HS(U, H_\beta) }^2 
	| \theta |_T
	(t-s)^{ \nicefrac{1}{2} + \beta - \gamma}
	%
	+
	4 
	\| B \|_{\HS(U,H_\beta)}^2 
	[ |\theta|_T ]^{\nicefrac{1}{2}}
	(t-s)^{ 1 + \beta - \gamma}
	\big)
	\\
	&
	\quad
	+
	\tfrac{ p }{ \sqrt{ 2 } } 
	(t-s)^{   \rho }
	\| B \|_{\HS(U, H_\beta)} 
	\big( 
	1
	+
	p 
	\sqrt{T} 
	| \sup\nolimits_{ h \in \H } \values_h |^{-\beta}
	\| B \|_{ \HS(U, H_\beta ) }  
	\big) 
	\tfrac{  
		T^{ \nicefrac{1}{2} + \beta - \gamma - \rho } 
	}{
		\sqrt{  1 + 2 ( \beta - \gamma - \rho )  } } 
	\\
	&
	\quad
	+
	2
	\sqrt{2} 
	p    
	[ | \theta |_T ]^{ \nicefrac{1}{2} } 
	(t-s)^{ \rho }
	\| B \|_{\HS(U, H_\beta)}^3 
	| \sup\nolimits_{ h \in \H } \values_h |^{- 2 \beta} 
	\tfrac{  
		T^{ 1 + \beta - \gamma - \rho} 
	}{ 1 + \beta - \gamma - \rho  }
	\\
	&
	\leq
	\tfrac{ p  
		\| B \|_{\HS(U,H_\beta)} 
		\max \{ | \sup\nolimits_{ h \in \H } \values_h |^{-2\beta}, 1 \}
		\max \{ T, 1 \}
	}{ \sqrt{ 2 ( 1 + 2 \beta - 2 \gamma ) } }
	\big(
	1   
	+ 
	4 p^2
	\| B \|_{ \HS(U,H_\beta) }^2  
	+
	4 
	\| B \|_{\HS(U, H_\beta)}^2  
	\big)
	\\
	&
	\quad
	\cdot 
	(t-s)^{ \nicefrac{1}{2} + \beta - \gamma }
	%
	%
	+
	p 
	(t-s)^\rho
	\| B \|_{\HS(U, H_\beta)} 
	\max \big \{ 
	| \sup\nolimits_{ h \in \H } \values_h |^{- 2\beta},
	1
	\big\}  
	\\
	&
	\quad 
	\cdot
	\big(
	\tfrac{ 1 }{ \sqrt{ 2 } }
	+
	\tfrac{ p }{ \sqrt{ 2 } }
	\| B \|_{ \HS(U, H_\beta)}
	+
	2
	\sqrt{2}
	\| B \|_{ \HS(U, H_\beta)}^2
	\big)
	\tfrac{ [\max \{ T, 1 \}]^{ \nicefrac{3}{2} + \beta } }{  \sqrt{ 1 + 2(\beta - \gamma - \rho ) } }
	\\
	&
	\leq
	\tfrac{p^3  
		\| B \|_{\HS(U,H_\beta)}
		\max \{  | \sup\nolimits_{ h \in \H } \values_h |^{-2\beta}, 1 \}
		\max \{ T, 1 \} }{ \sqrt{ 2
			(  1 + 2 ( \beta - \gamma - \rho ) ) } }
	\big(
	1   
	+  
	8
	\| B \|_{\HS(U, H_\beta)}^2 
	\big)
	(t-s)^{ \nicefrac{1}{2} + \beta - \gamma } 
	%
	%
	\\
	&
	\quad
	+
	2
	p^2  
	(t-s)^\rho 
	\| B \|_{\HS(U, H_\beta)} 
	\max 
	\big\{
	| \sup\nolimits_{ h \in \H } \values_h |^{- 2\beta},
	1 
	\big\}
	\big(
	1
	+ 
	\| B \|_{ \HS(U,H_\beta)}^2
	\big)
	\tfrac{ [\max \{ T, 1 \}]^{ \nicefrac{3}{2} + \beta } }{  \sqrt{ 1 + 2(\beta - \gamma - \rho ) } }
	\\
	&
	\leq
	\tfrac{3 p^3 
		\| B \|_{\HS(U,H_\beta)}
		[\max \{ T, 1 \} ]^{
			\nicefrac{3}{2} + \beta	
		}
		\max 
		\{
		| \sup\nolimits_{ h \in \H } \values_h |^{- 2 \beta},
		1 \}
		(
		1   
		+  
		8
		\| B \|_{\HS(U, H_\beta)}^2 
		)
	}{ \sqrt{  1 + 2 ( \beta - \gamma - \rho )  } }
	(t-s)^\rho 
	.
	\end{split}
	\end{equation}
	The proof of Lemma~\ref{lemma:noise_diff}
	is thus completed.
\end{proof}
\subsection[Error estimates for temporal approximations of stochastic convolutions]{Error estimates for temporally semi-discrete approximations of stochastic convolutions}
\label{subsection:Error}
In this subsection we
combine  
Lemma~\ref{lemma:Ito_representation}
and
Lemma~\ref{lemma:some_helping_estimates}
above to establish in Lemma~\ref{lemma:error_exact_counterpart} below for every $ p \in [2,\infty) $ an upper bound for the strong $ L^p $-distance between the approximation process 
$ \O \colon [0,T] \times \Omega \to H_\gamma $
from
Setting~\ref{setting:Strengthened_strong_a_priori_moment_bounds} 
 and a suitable stochastic convolution process related to 
 $ \O \colon [0,T] \times \Omega \to H_\gamma $.
\begin{lemma}
	\label{lemma:error_exact_counterpart}
	Assume Setting~\ref{setting:Strengthened_strong_a_priori_moment_bounds}, 
	let  
	$ C \in [1, \infty) $,
	$ \tilde B \in \HS(U, H_\beta) $,
	$ p \in [2, \infty) $,   
	$ \eta \in [0, \nicefrac{1}{2} + \beta - \gamma ) $, 
    $ \rho \in [0, \nicefrac{1}{2} + \beta - \gamma) \cap [0, \nicefrac{1}{2} ) $,
	assume for every 
	$ s \in [0,T] $ 
	that
	$ 	\| \chi_{ \llcorner s \lrcorner_\theta }  
	-
	1 
	\|_{\L^{p}(\P; \R)} \leq C [ | \theta |_T ]^{ \rho } $,
	and
	let
	$ O \colon [0,T] \times  \Omega \to H_\gamma $
	be a stochastic process which satisfies for every
	$ t \in [0,T] $ that
	$ [O_t]_{\P, \mathcal{B}(H_\gamma)} = 
	\int_0^t 
	e^{(t-s)A} 
	\tilde B 
	\, 
	dW_s $.
	Then  
	\begin{equation}
	\begin{split}
	\sup\nolimits_{t\in [0,T]} 
	\|  \O_t -  O_t  \|_{ \L^p( \P; H_\gamma ) }
 	&
	\leq 
	\tfrac{ p  
		[ \max \{ T, 1 \} ]^{\nicefrac{1}{2} + \beta } } { \sqrt{2 (1 - 2 \max \{ 0, \gamma + \eta - \beta \} )} } 
	\| 
	(-A)^{ \min \{ 0, \gamma + \eta - \beta \} }  
	\|_{L(H)} 
	\|  \tilde B - B \|_{\HS(U, H_{\beta-\eta})} 
	\\
	&
	\quad 
	+
	\tfrac{ 8 p^3  C  
		[\max\{T, 1\}]^{ \nicefrac{3}{2} + \beta  }   
	}{\sqrt{  1 - 2 \rho - 2 \max \{ 0, \gamma - \beta \}  } 
	}
	\| B \|_{\HS(U, H_{\beta})}
	\| ( -A )^{ \min \{ 0, \gamma - \beta \} } \|_{L(H)}
	\\
	&
	\quad
	\cdot
	\big(
	1
	+ 
	| \sup\nolimits_{h \in \H} \values_h |^{-2\beta}
	\| B \|_{\HS(U,H_{\beta})}^2
	\big)
	[ |\theta|_T ]^{\rho }
	.
	\end{split}
	\end{equation}
\end{lemma}
\begin{proof}[Proof
	of Lemma~\ref{lemma:error_exact_counterpart}]
	Throughout this proof 
	for every  $ s \in [0,T] $
	let 
	$ X_{s, (\cdot) }(\cdot) 
	=
	( X_{s,t}(\omega) )_{ (t,\omega)\in[s,T]\times\Omega }
	\colon $ 
	$ [s,T] \times \Omega
	\to H_\beta $
	be an
	$ ( \f_t )_{ t \in [s,T] } $-adapted
	stochastic process with continuous sample paths
	which satisfies for every 
	$ t \in [s,T] $
	that
	$ [ X_{s,t} ]_{\P, \mathcal{B}(H_\beta) }
	=
	\int_s^t B \, dW_u $.  
	Observe that
	Lemma~\ref{lemma:Ito_representation} 
	(with 
	$ X_{s, t} = X_{s, t} $
	for
	$ t \in [s,T] $,
	$ s \in [0,T] $ 
	in the notation of
	Lemma~\ref{lemma:Ito_representation}) 
	and the triangle inequality prove 
	for every $ t \in [0,T] $ that
	\begin{equation}
	\begin{split}
	\label{eq:basic_triangle}
	& 
	\|  \O_t -  O_t  \|_{ \L^p( \P; H_\gamma ) }
	\leq
	\Big\|
	\int_0^t
	e^{ (t-u) A }
	( B - \tilde B )
	\, dW_u
	\Big\|_{L^p(\P; H_\gamma)}
	\\
	&
	+
	\Big\|
	\int_0^t 
	\Big(
	\chi_{\llcorner u \lrcorner_\theta} 
	e^{(t- \llcorner u \lrcorner_\theta)A}
	\Big(
	\tfrac{B  }{1+\| X_{{\llcorner u \lrcorner_\theta}, u} \|_H^2}
	-
	\tfrac{2 X_{{\llcorner u \lrcorner_\theta}, u} \langle X_{{\llcorner u \lrcorner_\theta}, u}, B  (\cdot) \rangle_H}
	{ ( 1 + \| X_{{\llcorner u \lrcorner_\theta}, u} \|_H^2 )^2 } 
	\Big)
	- 
	e^{(t-u)A}  
	B  
	\Big)
	\, 
	dW_u
	\Big\|_{L^p(\P; H_\gamma)}
	\\
	& 
	+
	\Big\|
	\int_0^t
	\chi_{\llcorner u \lrcorner_\theta} 
	e^{(t-\llcorner u \lrcorner_\theta)A}
	\Big( 
	\tfrac{ 4  
		\|  \mathbb{B}
		X_{\llcorner u \lrcorner_\theta, u}  \|_U^2
	X_{\llcorner u \lrcorner_\theta, u} 
}
	{(1 + \| X_{\llcorner u \lrcorner_\theta, u} \|_H^2)^3}
	-
	\tfrac{
		2 B   
	\mathbb{B}
		X_{\llcorner u \lrcorner_\theta, u} 
		+ 
		\| B  \|_{\HS(U,H)}^2 
		X_{\llcorner u \lrcorner_\theta, u} 
	}{
		( 1 + \| X_{\llcorner u \lrcorner_\theta, u} \|_H^2 )^2
	}
	\Big)
	\, 
	du
	\Big\|_{\L^p(\P; H_\gamma)}
	.
	\end{split}
	\end{equation}
	In the next step we
	note that
	the fact that
	$ \eta < \nicefrac{1}{2} + \beta - \gamma $
	ensures that
	$ \max \{ \gamma + \eta - \beta, 0 \} < \nicefrac{1}{2} $.
	The
	Burkholder-Davis-Gundy type inequality in Lemma~7.7 in Da Prato \& Zabczyk~\cite{dz92}
	hence shows that 
	for every $ t \in [0,T] $ it holds that
	\begin{equation}
	\begin{split}
	\label{eq:B_defect}
	&
	\Big\|
	\int_0^t
	e^{ (t-u) A }
	( B - \tilde B )
	\, dW_u
	\Big\|_{L^p(\P; H_\gamma)}
	\leq
	\tfrac{ \sqrt{p(p-1)} } { \sqrt{2} } 
	\Big(
	\int_0^t
	\| e^{(t-u)A} ( \tilde B - B ) \|_{\HS(U, H_\gamma)}^2
	\, du
	\Big)^{ \nicefrac{1}{2} }
	\\
	&
	\leq
	\tfrac{ \sqrt{p(p-1)} } { \sqrt{2} } 
\| \tilde B - B \|_{\HS(U, H_{ \beta - \eta })} 
	\Big(
	\int_0^t
	\| (-A)^{ \gamma + \eta - \beta } e^{(t-u)A} \|_{L(H)}^2 
	\, du
	\Big)^{ \nicefrac{1}{2} }
	\\
	&
	\leq
	\tfrac{ \sqrt{p(p-1)} } { \sqrt{2} } 
	\| \tilde B - B \|_{\HS(U, H_{ \beta - \eta })} 
	\Big(
	\int_0^t
	\| 
	(-A)^{ \min \{ 0, \gamma + \eta - \beta \} }  
	\|_{L(H)}^2 
	\| 
	(-A)^{ \max \{ 0, \gamma + \eta - \beta \} } 
	e^{(t-u)A} 
	\|_{L(H)}^2 
	\, du
	\Big)^{ \nicefrac{1}{2} }
	\\
	&
	\leq
	\tfrac{ \sqrt{p(p-1)} } { \sqrt{2} } 
	\|  \tilde B - B  \|_{\HS(U, H_{ \beta - \eta })} 
	\| 
	(-A)^{ \min \{ 0, \gamma + \eta - \beta \} }  
	\|_{L(H)} 
	\Big( 
	\int_0^t
	(t-u)^{ - 2 \max \{ 0, \gamma + \eta - \beta \} } 
	\, du 
	\Big)^{ \nicefrac{1}{2} }
	\\
	&
	=
	\tfrac{ \sqrt{p(p-1)} } { \sqrt{2} }
	%
	\|  \tilde B - B \|_{\HS(U, H_{\beta - \eta} )}
	\| 
	(-A)^{ \min \{ 0, \gamma + \eta - \beta \} }  
	\|_{L(H)}  
	\tfrac{ t^{ \nicefrac{1}{2} - \max \{ 0, \gamma + \eta - \beta \} } } { \sqrt{ 1 - 2 \max \{ 0, \gamma + \eta - \beta \} } } 
	.
	\end{split}
	\end{equation}
	Furthermore, observe that the triangle inequality
	implies for every $ t \in [0,T] $ that
	\begin{equation}
	\begin{split}
	\label{eq:basic_triangle2}
	&
	\Big\|
	\int_0^t 
	\Big(
	\chi_{\llcorner u \lrcorner_\theta}
	e^{(t - \llcorner u \lrcorner_\theta)A}
	\Big[
	\tfrac{B  }{1+\| X_{{\llcorner u \lrcorner_\theta}, u} \|_H^2}
	-
	\tfrac{2 X_{{\llcorner u \lrcorner_\theta}, u} \langle X_{{\llcorner u \lrcorner_\theta}, u}, B  (\cdot) \rangle_H}
	{ ( 1 + \| X_{{\llcorner u \lrcorner_\theta}, u} \|_H^2 )^2 } 
	\Big]
	- 
	e^{(t - u)A}  
	B  
	\Big)
	\, 
	dW_u
	\Big\|_{L^p(\P; H_\gamma)}
	\\
	&
	\leq
	\Big\|
	\int_0^t 
	\Big(
	\chi_{\llcorner u \lrcorner_\theta} 
	e^{(t-\llcorner u \lrcorner_\theta)A} 
	\,
	\tfrac{B  }{1+\| X_{{\llcorner u \lrcorner_\theta}, u} \|_H^2}
	- 
	e^{(t - u)A}  
	B  
	\Big)
	\, 
	dW_u 
	\Big\|_{L^p(\P; H_\gamma)}
	\\
	&
	\quad
	+
	\Big\|
	\int_0^t 
	\chi_{\llcorner u \lrcorner_\theta} 
	e^{(t-\llcorner u \lrcorner_\theta)A} 
	\,
	\tfrac{2 X_{{\llcorner u \lrcorner_\theta}, u} \langle X_{{\llcorner u \lrcorner_\theta}, u}, B
		\, 
		dW_u  
		\rangle_H}
	{ ( 1 + \| X_{{\llcorner u \lrcorner_\theta}, u} \|_H^2 )^2 } 
	\Big\|_{L^p(\P; H_\gamma)}
	.
	\end{split}
	\end{equation}
	In addition, note that the triangle inequality 
	assures for every 
	$ t \in [0,T] $ 
	that
	\begin{equation}
	\begin{split}
	\label{eq:triangle+burkholder_estimate_error_stochastic}
	&
	\Big\|
	\int_0^t 
	\Big(
	\chi_{\llcorner u \lrcorner_\theta} 
	e^{(t-\llcorner u \lrcorner_\theta)A} 
	\,
	\tfrac{B  }{1+\| X_{{\llcorner u \lrcorner_\theta}, u} \|_H^2}
	- 
	e^{(t-u)A} 
	B 
	\Big)
	\, 
	dW_u 
	\Big\|_{L^p(\P; H_\gamma)}
	\\
	&
	\leq
	\Big\|
	\int_0^t 
	\chi_{\llcorner u \lrcorner_\theta} 
	e^{(t-\llcorner u \lrcorner_\theta)A}
	\Big[
	\tfrac{B  }{1+\| X_{{\llcorner u \lrcorner_\theta}, u} \|_H^2}
	-
	B 
	\Big]
	\, 
	dW_u 
	\Big\|_{L^p(\P; H_\gamma)}
	\\
	&
	\quad
	+
	\Big\|
	\int_0^t
	\chi_{ \llcorner u \lrcorner_\theta } 
	\big(
	e^{(t- \llcorner u \lrcorner_\theta)A} 
	- e^{(t-u)A} 
	\big)
	B 
	\, 
	dW_u
	\Big\|_{L^p(\P; H_\gamma)}
	\\
	&
	\quad
	+
	\Big\|
	\int_0^t
	 (
	\chi_{ \llcorner u \lrcorner_\theta } 
	-1
	 )
	e^{(t-u)A}   
	B 
	\,
	dW_u
	\Big\|_{L^p(\P; H_\gamma)}
	.
	\end{split}
	\end{equation}
	Moreover, observe that
	the  
	Burkholder-Davis-Gundy type inequality in Lemma~7.7 in Da Prato \& Zabczyk~\cite{dz92} 
	demonstrates that for every 
	$ t \in [0,T] $
	it holds that
	\begin{equation}
	\begin{split}
	&
	\Big\|
	\int_0^t 
	\chi_{ \llcorner u \lrcorner_\theta }
	e^{(t- \llcorner u \lrcorner_\theta )A}
	\Big[
	\tfrac{B  }{1+\| X_{{ \llcorner u \lrcorner_\theta }, u} \|_H^2}
	-
	B 
	\Big]
	\, 
	dW_u 
	\Big\|_{L^p(\P; H_\gamma)} 
	\\
	&
	\leq 
	\tfrac{\sqrt{p(p-1)}}{ \sqrt{2} }
	\Big(
	\int_0^t  
	\Big\|  
	e^{(t-\llcorner u \lrcorner_\theta)A}
	\tfrac{ 
		\| X_{\llcorner u \lrcorner_\theta, u} \|_H^2
		B
	}{1+\| X_{\llcorner u \lrcorner_\theta, u} \|_H^2}  
	\Big\|_{\L^p(\P; \HS(U,H_\gamma))}^2
	\, 
	du
	\Big)^{\nicefrac{1}{2}}
	\\
	&
	\leq 
	\tfrac{\sqrt{p(p-1)}}{ \sqrt{2} }
	\Big(
	\int_0^t  
	\|  
	(-A)^{ \gamma - \beta }
	e^{(t-\llcorner u \lrcorner_\theta)A}
	\|_{L(H)}^2
	\Big\|
	\tfrac{ 
		\| X_{\llcorner u \lrcorner_\theta, u} \|_H^2
		B
	}{1+\| X_{\llcorner u \lrcorner_\theta, u} \|_H^2}  
	\Big\|_{\L^p(\P; \HS(U,H_\beta))}^2
	\, 
	du
	\Big)^{\nicefrac{1}{2}}
	\\
	&
	\leq
	\tfrac{\sqrt{p(p-1)}}{ \sqrt{2} }
	\Big(
	\int_0^t 
	\|  
	(-A)^{ \gamma - \beta }
	e^{(t-\llcorner u \lrcorner_\theta)A}
	\|_{L(H)}^2
	\| X_{\llcorner u \lrcorner_\theta, u} \|_{\L^{2p}(\P; H)}^4
	\|  
	B 
	\|_{ \HS(U, H_\beta) }^2
	\, 
	du
	\Big)^{\nicefrac{1}{2}}
	.
	\end{split}
	\end{equation}
	The
	Burkholder-Davis-Gundy type inequality in Lemma~7.7 in Da Prato \& Zabczyk~\cite{dz92}
	and the fact that
	$ \| B \|_{\HS(U,H)} 
	\leq
	| \sup_{ h \in \H } \values_h |^{-\beta} 
	\| B \|_{\HS(U,H_\beta)} $ 
	hence ensure that for every $ t \in [0,T] $ 
	it holds that
	\begin{equation}
	\begin{split} 
	\label{eq:estimate_first0}
	&
	\Big\|
	\int_0^t 
	\chi_{\llcorner u \lrcorner_\theta}
	e^{(t-\llcorner u \lrcorner_\theta)A}
	\Big[
	\tfrac{B  }{1+\| X_{{\llcorner u \lrcorner_\theta}, u} \|_H^2}
	-
	B 
	\Big]
	\, 
	dW_u 
	\Big\|_{L^p(\P; H_\gamma)}
	%
	%
	\\
	&
	\leq
	\tfrac{\sqrt{p(p-1)}}{ \sqrt{2} }
	\Big(
	\int_0^t 
\|  
(-A)^{ \gamma - \beta }
e^{(t-\llcorner u \lrcorner_\theta)A}
\|_{L(H)}^2
	\|  
	B 
	\|_{  \HS(U, H_{\beta}) }^2
	\Big(
	p ( 2p - 1 )
	\smallint_{\llcorner u \lrcorner_\theta}^u
	\| B  \|_{  \HS(U,H)  }^2
	\, ds
	\Big)^2
	\, 
	du
	\Big)^{\nicefrac{1}{2}}
	\\
	&
	\leq
	\sqrt{2} p^3
	| \sup\nolimits_{ h \in \H } \values_h |^{-2\beta} 
	\|  
	B 
	\|_{ \HS(U, H_{\beta }) }^3
	\Big(
	\int_0^t 
\|  
(-A)^{ \gamma - \beta }
e^{(t-\llcorner u \lrcorner_\theta)A}
\|_{L(H)}^2
	( u - \llcorner u \lrcorner_\theta )^2 
	\, 
	du
	\Big)^{\nicefrac{1}{2}}
	\\
	&
	\leq 
	\sqrt{2} p^3 
| \sup\nolimits_{ h \in \H } \values_h |^{-2\beta} 
\|  
B 
\|_{ \HS(U, H_{\beta }) }^3
\|  
(-A)^{ \min \{0, \gamma - \beta \} } 
\|_{L(H)}
\\
&
\quad 
\cdot 
\Big(
\int_0^t 
\|  
(-A)^{ \max \{0, \gamma - \beta \} }
e^{(t-\llcorner u \lrcorner_\theta)A}
\|_{L(H)}^2
( u - \llcorner u \lrcorner_\theta )^2 
\, 
du
\Big)^{\nicefrac{1}{2}}	
	\\
	&
	\leq
		\sqrt{2} p^3 
	| \sup\nolimits_{ h \in \H } \values_h |^{-2\beta} 
	\|  
	B 
	\|_{ \HS(U, H_{\beta }) }^3
	\|  
	(-A)^{ \min \{0, \gamma - \beta \} } 
	\|_{L(H)}
	\\
	&\quad 
	\cdot 
	| \theta |_T 
	\Big ( 
	\int_0^t 
	(t - \llcorner u \lrcorner_\theta )^{ 
- 2 \max \{ 0, \gamma - \beta \}	
 } 
	\, 
	du 
	\Big )^{\nicefrac{1}{2}}
	.
	\end{split} 
	\end{equation} 
	Furthermore, observe that
	the fact that
	$ 1 - 2 \max \{ 0, \gamma - \beta \} = 
	\max \{ 1, 1 - 2 \gamma + 2 \beta \} > 0 $
	ensures that
	for every
	$ t \in [0,T] $ it holds that 
	\begin{equation}
	\begin{split}  
	&
	\int_0^t 
	(t - \llcorner u \lrcorner_\theta )^{ 
		- 2 \max \{ 0, \gamma - \beta \}	
	}
	\, 
	du
	\leq 
	\int_0^t 
	(t - u)^{  
			- 2 \max \{ 0, \gamma - \beta \}	
	}  
	\, 
	du 
= 
	\tfrac{  
		t^{ 1  - 2 \max \{ 0, \gamma - \beta \}	
		}  
	}
	{ 1 - 2 \max \{ 0, \gamma - \beta \}	 }  
	.
	\end{split}
	\end{equation} 
	Combining this and~\eqref{eq:estimate_first0} 
	implies that for every
	$ t \in [0,T] $ it holds that
	\begin{equation}
	\begin{split} 
	\label{eq:estimate_first}
	&
	\Big\|
	\int_0^t 
	\chi_{\llcorner u \lrcorner_\theta}
	e^{(t-\llcorner u \lrcorner_\theta)A}
	\Big[
	\tfrac{B  }{1+\| X_{{\llcorner u \lrcorner_\theta}, u} \|_H^2}
	-
	B 
	\Big]
	\, 
	dW_u 
	\Big\|_{L^p(\P; H_\gamma)}
	\\
	&
	\leq
	\sqrt{2} p^3  
	| \sup\nolimits_{ h \in \H } \values_h |^{-2\beta} 
	\|  
	B 
	\|_{ \HS(U, H_{\beta }) }^3
	\|  
	(-A)^{ \min \{0, \gamma - \beta \} } 
	\|_{L(H)}
	\tfrac{  
		t^{ \nicefrac{1}{2} - \max \{ 0, \gamma - \beta \}	
		}  
	}
	{ \sqrt{ 1 - 2 \max \{ 0, \gamma - \beta \}	} }
	|\theta|_T
	.
	\end{split}
	\end{equation}
	In the next step we note that
	the fact that
	$ 1 - 2 \rho - 2 \max \{ 0, \gamma - \beta \} 
	= 1 - 2 \rho 
	+ 2 \min \{ 0, \beta - \gamma \} > 0 $
	and
	the  
	Burkholder-Davis-Gundy type inequality in Lemma~7.7 in Da Prato \& Zabczyk~\cite{dz92}
	show  that for every 
	$ t \in [0,T] $ it holds that
	\begin{equation}
	\begin{split}
	\label{eq:estimate_second}
	&
	\Big\|
	\int_0^t
	\chi_{ \llcorner u \lrcorner_\theta }
	(
	e^{(t- \llcorner u \lrcorner_\theta )A} 
	- e^{(t-u)A} 
	)
	B  
	\, 
	dW_u
	\Big\|_{L^p(\P; H_\gamma)}
	\\
	&
	\leq
	\tfrac{\sqrt{p(p-1)}}{\sqrt{2}}
	\Big(
	\int_0^t 
	\| 
	(
	e^{(t- \llcorner u \lrcorner_\theta)A} 
	- e^{(t-u)A} 
	)
	B  
	\|_{ \HS(U, H_\gamma) }^2
	\,
	du
	\Big)^{\nicefrac{1}{2}}
	\\
	&
	\leq
	\tfrac{p}{\sqrt{2}}
	\Big(
	\int_0^t 
	\| (-A)^{\gamma + \rho - \beta } e^{(t-u)A} \|_{L(H)}^2
	\|
	(-A)^{-\rho}
	(
	e^{(u - \llcorner u \lrcorner_\theta)A} 
	- 
	\operatorname{Id}_H 
	)
	\|_{L(H)}^2
	\| B  
	\|_{ \HS(U, H_{\beta } )}^2
	\,
	du
	\Big)^{\nicefrac{1}{2}}
	\\
	&
	\leq
	\tfrac{ p }{\sqrt{2}}   
	\| B  
	\|_{ \HS(U, H_{\beta} ) }
	\Big(
	\int_0^t 
	\| (-A)^{\gamma + \rho - \beta } e^{(t-u)A} \|_{L(H)}^2  
	( u - \llcorner u \lrcorner_\theta )^{2 \rho}
	\,
	du
	\Big)^{\nicefrac{1}{2}}
	\\
	&
	\leq
	\tfrac{ p }{\sqrt{2}}   
	\| B  
	\|_{ \HS(U, H_{\beta} ) }
	\| 
	(-A)^{ \min \{ 0, \gamma - \beta \} } \|_{L(H)}
	\\
	&\quad \cdot 
	\Big(
	\int_0^t 
	\| (-A)^{ \rho + \max \{ 0, \gamma - \beta \} } e^{(t-u)A} \|_{L(H)}^2  
	( u - \llcorner u \lrcorner_\theta )^{2 \rho}
	\,
	du
	\Big)^{\nicefrac{1}{2}}
	\\
	&
	\leq
	\tfrac{ p }{\sqrt{2}}  
	\| B  
	\|_{ \HS(U, H_{\beta} ) }
	\| 
	(-A)^{ \min \{ 0, \gamma  - \beta \} } \|_{L(H)}
	\Big( 
	\int_0^t  
	\tfrac{ 
	| \theta |_T^{2 \rho}
	}
	{
	( t - u )^{ 2 \rho + 
	2 \max \{0, \gamma - \beta \} }
	}
	\,
	du 
	\Big)^{\nicefrac{1}{2}}
	\\
	&
	\leq
	\tfrac{ p }{\sqrt{2}} 
	\| B  
	\|_{ \HS(U, H_{\beta } ) }
	\| 
	(-A)^{ \min \{ 0, \gamma  - \beta \} } \|_{L(H)}
	\tfrac{
	t^{ \nicefrac{1}{2} - \rho - \max \{0, \gamma - \beta \} } 
	}
	{
	\sqrt{ 1 - 2 \rho - 2 \max \{0, \gamma - \beta \} } 
	} 
	[ |\theta|_T ]^\rho 
	.
	\end{split}
	\end{equation}
	Moreover, observe that the assumption that
	$ \forall \, u \in [0,T], 
	\theta \in \varpi_T 
	\colon 
	\| \chi_{ \llcorner u \lrcorner_\theta }  
	-
	1 
	\|_{\L^{p}(\P; \R)} \leq C [ | \theta |_T ]^{ \rho } $
	and
	the  
	Burkholder-Davis-Gundy type inequality in Lemma~7.7 in Da Prato \& Zabczyk~\cite{dz92}
	 establish for every  
	$ t \in [0,T] $
	that
	\begin{equation}
	\begin{split}
	\label{eq:final_noise_estimate}
	&
	\Big\|
	\int_0^t
	(
	\chi_{ \llcorner u \lrcorner_\theta } 
	-
	1
	)
	e^{(t - u)A}  B 
	\,
	dW_u
	\Big\|_{L^p(\P; H_\gamma)}
	\\
	&
	\leq
	\tfrac{\sqrt{p(p-1)}}{\sqrt{2}}
	\Big(
	\int_0^t
	\big\|
	(
	\chi_{ \llcorner u \lrcorner_\theta } 
	-
	1
	)
	e^{(t - u)A}  B 
	\big\|_{ \L^p( \P; \HS(U, H_\gamma ) )}^2
	\, 
	du
	\Big)^{\nicefrac{1}{2}}
	\\
	&
	\leq
	\tfrac{ p }{\sqrt{2}}
	\Big(
	\int_0^t
	\| 
	\chi_{ \llcorner u \lrcorner_\theta } 
	-1
	\|_{\L^{p}(\P; \R)}^2
	\| (-A)^{\gamma - \beta} 
	e^{(t-u)A} 
	\|_{L(H)}^2
	\| 
	B  
	\|_{ \HS(U, H_{\beta } )}^2
	\, 
	du
	\Big)^{\nicefrac{1}{2}}
	\\
	&
	\leq
	\tfrac{ p C }{\sqrt{2}}  
	\| 
	B 
	\|_{ \HS(U, H_{\beta} )} 
	[ | \theta |_T ]^\rho
	\Big(
	\int_0^t  
\| (-A)^{\gamma - \beta} 
e^{(t-u)A} 
\|_{L(H)}^2
	\, 
	du
	\Big)^{\nicefrac{1}{2}}
	\\
	&
	\leq 
		\tfrac{ p C }{\sqrt{2}}  
	\| 
	B 
	\|_{ \HS(U, H_{\beta} )} 
	[ | \theta |_T ]^\rho
	\| (-A)^{ \min \{0, \gamma - \beta \} }   
	\|_{L(H)}
	\Big(
	\int_0^t  
	\| (-A)^{ \max \{0, \gamma - \beta \} } 
	e^{(t-u)A} 
	\|_{L(H)}^2
	\, 
	du
	\Big)^{\nicefrac{1}{2}}
	\\
	&
	\leq 
	\tfrac{ p C }{ \sqrt{2} } 
	\| 
	B 
	\|_{ \HS(U, H_{\beta} )} 
	[ |\theta|_T ]^{ \rho }
		\| (-A)^{ \min \{0, \gamma - \beta \} }   
	\|_{L(H)}
	\Big (  
	\int_0^t 
	( t - u )^{ -2 \max \{ 0, \gamma - \beta \} } \, du 
	\Big )^{ \nicefrac{1}{2} }
	\\
	&
	=  
	\tfrac{ p C }{ \sqrt{2} }
	\| 
	B 
	\|_{ \HS(U, H_{\beta} ) } 
	[ |\theta|_T ]^{ \rho }
	\| (-A)^{ \min \{0, \gamma - \beta \} }   
	\|_{L(H)}	
	\tfrac{
		t^{\nicefrac{1}{2} -  \max \{ 0, \gamma - \beta \} } 
	}
	{
		\sqrt{1 -2 \max \{ 0, \gamma - \beta \} }
	} 
	.
	\end{split}
	\end{equation}
	Combining this, 
	\eqref{eq:triangle+burkholder_estimate_error_stochastic},
	\eqref{eq:estimate_first},
	\eqref{eq:estimate_second}, 
	and H\"older's inequality 
	demonstrates that for every 
	$ t \in [0,T] $ 
	it holds that
	\begin{equation}
	\begin{split}
	\label{eq:first_part_estimate}
	&
	\Big\|
	\int_0^t 
	\Big(
	\chi_{ \llcorner u \lrcorner_\theta } 
	e^{(t- \llcorner u \lrcorner_\theta )A} 
	\,
	\tfrac{B  }{1+\| X_{{\llcorner u \lrcorner_\theta}, u} \|_H^2}
	- 
	e^{(t-u)A}  
	B  
	\Big)
	\, 
	dW_u 
	\Big\|_{L^p(\P; H_\gamma)}
	\\
	&
	\leq
\sqrt{2} p^3  
| \sup\nolimits_{ h \in \H } \values_h |^{-2\beta} 
\|  
B 
\|_{ \HS(U, H_{\beta }) }^3
\|  
(-A)^{ \min \{0, \gamma - \beta \} } 
\|_{L(H)}
\tfrac{  
	t^{ \nicefrac{1}{2} - \max \{ 0, \gamma - \beta \}	
	}  
}
{ \sqrt{ 1 - 2 \max \{ 0, \gamma - \beta \}	} }
|\theta|_T
	\\
	&
	\quad 
	+
\tfrac{ p }{\sqrt{2}} 
\| B  
\|_{ \HS(U, H_{\beta } ) }
\| 
(-A)^{ \min \{ 0, \gamma  - \beta \} } \|_{L(H)}
\tfrac{
	t^{ \nicefrac{1}{2} - \rho - \max \{0, \gamma - \beta \} } 
}
{
	\sqrt{ 1 - 2 \rho - 2 \max \{0, \gamma - \beta \} } 
} 
[ |\theta|_T ]^\rho
	\\
	&
	\quad
	+ 
	\tfrac{ p C }{ \sqrt{2} }
\| 
B 
\|_{ \HS(U, H_{\beta} ) } 
[ |\theta|_T ]^{ \rho }
\| (-A)^{ \min \{0, \gamma - \beta \} }   
\|_{L(H)}	
\tfrac{
	t^{\nicefrac{1}{2} -  \max \{ 0, \gamma - \beta \} } 
}
{
	\sqrt{1 -2 \max \{ 0, \gamma - \beta \} }
} 
	\\
	&
	\leq 
	\tfrac{
		p  
	C
	}{
	\sqrt{2 ( 1 - 2 \rho - 2 \max \{ 0, \gamma - \beta \} ) }}
\|  
(-A)^{ \min \{0, \gamma - \beta \} } 
\|_{L(H)}
	\| 
	B 
	\|_{ \HS(U, H_{\beta} ) } 
	\big(
	1
	+  
	| \sup\nolimits_{ h \in \H } \values_h |^{ - 2 \beta }
	\|  
	B 
	\|_{ \HS(U, H_\beta) }^2
	\big)
	\\
	&
	\quad
	\cdot
	\max \big\{  
	t^{ \nicefrac{1}{2} 
		+ 
		\min \{ 0, \beta - \gamma \} },
	t^{\nicefrac{1}{2} - \rho + 
		\min \{ 0, \beta - \gamma \} } 
	\big\}
	\Big[
	2 p^2
	|\theta|_T 
	+      
	[ |\theta|_T ]^\rho 
	+ 
	[ |\theta|_T ]^\rho   
	\Big]
	\\
	&
	\leq
	\tfrac{p
		C 
		[\max\{T, 1\} ]^{ \nicefrac{1}{2} + \beta }}{\sqrt{ 2 ( 1 - 2 \rho - 2 \max \{ 0, \gamma - \beta \} ) }}
	\| 
	B 
	\|_{ \HS(U, H_{\beta} ) } 
	\|  
	(-A)^{ \min \{0, \gamma - \beta \} } 
	\|_{L(H)}
	\\
	&
	\quad 
	\cdot 
	\big(
	1
	+ 
	| \sup\nolimits_{ h \in \H } \values_h |^{-2\beta} 
	\|  
	B 
	\|_{ \HS(U, H_\beta) }^2
	\big)
	\big[
	2 p^2
	|\theta|_T 
	+ 
	2   
	[ |\theta|_T ]^\rho 
	\big]	
	\\
	&
	\leq
	\tfrac{ 3 p^3  
     C 
		[\max\{T, 1\}]^{ \nicefrac{3}{2} + \beta }  }{\sqrt{2 
			(  1 - 2 \rho - 2 \max \{0, \gamma - \beta \} ) }}
	\| 
	B 
	\|_{ \HS(U, H_{\beta} ) } 
	\|  
	(-A)^{ \min \{0, \gamma - \beta \} } 
	\|_{L(H)}
	\\
	&
	\quad 
	\cdot 
	\big(
	1
	+ 
	| \sup\nolimits_{ h \in \H } \values_h |^{-2\beta} 
	\|  
	B 
	\|_{  \HS(U, H_\beta) }^2
	\big) 
	[ |\theta|_T ]^{  \rho  }
	.
	\end{split}
	\end{equation}
	Furthermore, note that
	the fact that
	$ \gamma = \max \{ \beta, \gamma \}
	+
	\min \{ 0, \gamma - \beta \} $
	and
	Lemma~\ref{lemma:some_helping_estimates}
	(with 
	$ p = p $,
	$ \rho = \max \{ \beta, \gamma \} $,
	$ \eta = \max \{ \beta, \gamma \} $, 
	$ X_{s, t} = X_{s, t} $
	for 
	$ t \in [s,T] $,
	$ s \in [0,T] $, 
	$ h \in (0,T] $
	in the notation of
	Lemma~\ref{lemma:some_helping_estimates})
	prove that for every $ t \in [0,T] $ it holds that
	\begin{equation}
	\begin{split}
	\label{eq:apply_useful_estimate1}
	&
	\Big\|
	\int_0^t 
	\chi_{\llcorner u \lrcorner_\theta} 
	e^{(t-\llcorner u \lrcorner_\theta)A} 
	\,
	\tfrac{2 X_{{\llcorner u \lrcorner_\theta}, u} 
		\langle X_{{\llcorner u \lrcorner_\theta}, u}, B \, 
		dW_u \rangle_H}
	{ ( 1 + \| X_{{\llcorner u \lrcorner_\theta}, u} \|_H^2 )^2 }   
	\Big \|_{L^p(\P; H_{ \gamma })} 
	\\
	&
	\leq 
	\| ( -A )^{ \min \{ 0, \gamma - \beta \} } \|_{L(H)}
	\Big\|
	\int_0^t 
	\chi_{\llcorner u \lrcorner_\theta} 
	e^{(t-\llcorner u \lrcorner_\theta)A} 
	\,
	\tfrac{2 X_{{\llcorner u \lrcorner_\theta}, u} 
		\langle X_{{\llcorner u \lrcorner_\theta}, u}, B \, 
		dW_u \rangle_H}
	{ ( 1 + \| X_{{\llcorner u \lrcorner_\theta}, u} \|_H^2 )^2 }   
	\Big \|_{L^p(\P; H_{ \max \{ \beta, \gamma \} })} 
	\\
	&
	\leq
	2 \sqrt{2} p^3  
	\| B \|_{ \HS(U,H_{\beta}) }^3
	| \sup\nolimits_{ h \in \H } \values_h |^{-2\beta} 
	\| ( -A )^{ \min \{ 0, \gamma - \beta \} } \|_{L(H)}
	\tfrac{ 
		t^{ \nicefrac{1}{2} + \beta - \max \{ \beta, \gamma \} } 
	}
	{
		\sqrt{  1 + 2 \beta - 2 \max \{ \beta, \gamma \} }
	}  
	|\theta|_T
	\end{split}
	\end{equation}
	and
	\begin{equation}
	\begin{split}
	\label{eq:apply_useful_estimate2}
	&
	\Big\|
	\int_0^t
	\chi_{ \llcorner u \lrcorner_\theta }
	e^{(t- \llcorner u \lrcorner_\theta)A}
	\Big( 
	\tfrac{ 4  
		\|  \mathbb{B}
		X_{\llcorner u \lrcorner_\theta, u}  \|_{H}^2
		X_{\llcorner u \lrcorner_\theta, u} 
	}
	{(1 + \| X_{\llcorner u \lrcorner_\theta, u} \|_H^2)^3}
	-
	\tfrac{
		2 B   
		\mathbb{B}
		X_{\llcorner u \lrcorner_\theta, u} 
		+ 
		\| B  \|_{\HS(U,H)}^2
		X_{\llcorner u \lrcorner_\theta, u}  
	}{
		( 1 + \| X_{\llcorner u \lrcorner_\theta, u} \|_H^2 )^2
	}
	\Big)
	\, 
	du
	\Big\|_{\L^p(\P; H_\gamma )}
	\\
	&
	\leq 
	\| ( -A )^{ \min \{ 0, \gamma - \beta \} } \|_{L(H)}
	\\
	&
	\quad \cdot 
	\Big\|
	\int_0^t
	\chi_{ \llcorner u \lrcorner_\theta }
	e^{(t- \llcorner u \lrcorner_\theta)A}
	\Big( 
	\tfrac{ 4  
		\| \mathbb{B}
		X_{\llcorner u \lrcorner_\theta, u}  \|_{H}^2
	X_{\llcorner u \lrcorner_\theta, u} 
	}
	{(1 + \| X_{\llcorner u \lrcorner_\theta, u} \|_H^2)^3}
	-
	\tfrac{
		2 B   
	\mathbb{B}
		X_{\llcorner u \lrcorner_\theta, u} 
		+ 
		\| B  \|_{\HS(U,H)}^2
		X_{\llcorner u \lrcorner_\theta, u}  
	}{
		( 1 + \| X_{\llcorner u \lrcorner_\theta, u} \|_H^2 )^2
	}
	\Big)
	\, 
	du
	\Big\|_{\L^p(\P; H_{   \max \{\beta, \gamma \}  } )}
	\\
	&
	\leq
	2
	\sqrt{2}
	p 
	\|  B \|_{ \HS( U, H_{\beta} ) }^3
	| \sup\nolimits_{ h \in \H } \values_h |^{-2\beta} 
		\| ( -A )^{ \min \{ 0, \gamma - \beta \} } \|_{L(H)}
	\tfrac{  t^{ 1 + \beta - \max \{ \beta, \gamma \} } 
}{  1 + \beta - \max \{ \beta, \gamma \} } 
	[ |\theta|_T ]^{\nicefrac{1}{2}}
	.
	\end{split}
	\end{equation}
	Moreover, observe that
	\begin{equation}
	\begin{split}
	\label{eq:trivial estimate}
	\frac{1}{ \sqrt{ 1 + 2 \beta - 2 \max \{ \beta, \gamma \} } }
	=
	\frac{1}{ \sqrt{ 1 - 2 \max \{ 0, \gamma - \beta \} } }
	\leq
	\frac{1}{ \sqrt{ 1 - 2 \rho - 2 \max \{ 0, \gamma - \beta \} } }
	.
	\end{split}
	\end{equation}
	Combining~\eqref{eq:basic_triangle},
	\eqref{eq:B_defect}, 
	\eqref{eq:basic_triangle2}, 
	\eqref{eq:first_part_estimate},
	\eqref{eq:apply_useful_estimate1}, 
	\eqref{eq:apply_useful_estimate2},
	the fact that
	$ \sqrt{ 1 + 2 \beta - 2 \max \{ \beta, \gamma \} } \leq 1 + \beta - \max \{ \beta, \gamma \} $,
	and the fact that
	$ \frac{ 3 }{\sqrt{2}} 
	+ 2 \sqrt{2} 
	+ 2 \sqrt{2} \leq 8 $
	hence ensures that
	for every $ t \in [0,T] $ it holds that
	\begin{equation}
	\begin{split}
	& 
	\|  \O_t -  O_t  \|_{ \L^p( \P; H_\gamma ) }
	\leq    
	\|  \tilde B - B \|_{\HS(U, H_{\beta - \eta} )}
	\| 
	(-A)^{ \min \{ 0, \gamma + \eta - \beta \} }  
	\|_{L(H)}  
	\tfrac{ p [ \max\{ T, 1 \} ]^{ \nicefrac{1}{2} + \beta } } { \sqrt{ 2(1 - 2 \max \{ 0, \gamma + \eta - \beta \} ) } } 
\\
&
\quad 
+
\tfrac{ 3 p^3  
	C 
	[\max\{T, 1\}]^{ \nicefrac{3}{2} + \beta }  }{\sqrt{2 
		(  1 - 2 \rho - 2 \max \{0, \gamma - \beta \} ) }}
\| 
B 
\|_{ \HS(U, H_{\beta} ) } 
\|  
(-A)^{ \min \{0, \gamma - \beta \} } 
\|_{L(H)}
\\
&
\quad 
\cdot 
\big(
1
+ 
| \sup\nolimits_{ h \in \H } \values_h |^{-2\beta} 
\|  
B 
\|_{  \HS(U, H_\beta) }^2
\big) 
[ |\theta|_T ]^{  \rho  }
	\\
	&
	\quad
	+
	2 \sqrt{2} p^3  
	\| B \|_{ \HS(U,H_{\beta}) }^3
	| \sup\nolimits_{ h \in \H } \values_h |^{-2\beta} 
		\| ( -A )^{ \min \{ 0, \gamma - \beta \} } \|_{L(H)}
	\tfrac{
		[ \max\{ T, 1 \} ]^{ \nicefrac{1}{2} + \beta }
	}
	{
		\sqrt{ 1 + 2 \beta - 2 \max \{ \beta, \gamma \} }
	}  
	|\theta|_T
	\\
	&
	\quad
	+
	2
	\sqrt{2}
	p 
	\|  B \|_{ \HS( U, H_{\beta} ) }^3
	| \sup\nolimits_{ h \in \H } \values_h |^{-2\beta} 
		\| ( -A )^{ \min \{ 0, \gamma - \beta \} } \|_{L(H)}
	\tfrac{ [ \max \{ T, 1 \} ]^{ 1 + \beta }}{ 
		\sqrt{ 1 + 2 \beta - 2 \max \{ \beta, \gamma \} } } 
	[ |\theta|_T ]^{\nicefrac{1}{2}}
	\\
	&
	\leq 
	\tfrac{ p  
		[ \max \{ T, 1 \} ]^{\nicefrac{1}{2} + \beta } } { \sqrt{2 (1 - 2 \max \{ 0, \gamma + \eta - \beta \} )} } 
	\| 
	(-A)^{ \min \{ 0, \gamma + \eta - \beta \} }  
	\|_{L(H)} 
	\|  \tilde B - B \|_{\HS(U, H_{\beta-\eta})} 
	\\
	&
	\quad 
	+
	\tfrac{ 8 p^3  C  
		[\max\{T, 1\}]^{ \nicefrac{3}{2} + \beta  }   
	}{\sqrt{  1 - 2 \rho - 2 \max \{ 0, \gamma - \beta \}  } 
	}
	\| B \|_{\HS(U, H_{\beta})}
	\| ( -A )^{ \min \{ 0, \gamma - \beta \} } \|_{L(H)}
	\\
	&
	\quad
	\cdot
	\big(
	1
	+ 
	| \sup\nolimits_{h \in \H} \values_h |^{-2\beta}
	\| B \|_{\HS(U,H_{\beta})}^2
	\big)
	[ |\theta|_T ]^{\rho }
	.
	\end{split}
	\end{equation}
	The
	proof of Lemma~\ref{lemma:error_exact_counterpart} is thus completed.
\end{proof}
\subsection[Exponential moments of temporal approximations of stochastic convolutions]{Exponential moments of temporally semi-discrete approximations of stochastic convolutions}
\label{subsection:ExpoMoments}
In this subsection
we
first derive two auxiliary lemmas    
(see 
Lemma~\ref{lemma:Some_helping_estimates} 
and
Lemma~\ref{lemma:some_helping_estimates2}
below)
 which we then 
 combine to establish
 in Lemma~\ref{lemma:Exp_regularity} below 
appropriate
exponential moment bounds
for the
approximation
process
$ \O \colon [0,T] \times \Omega \to H_\gamma $
from
Setting~\ref{setting:Strengthened_strong_a_priori_moment_bounds}.  
\begin{lemma}
	\label{lemma:Some_helping_estimates}
	\sloppy
	Assume Setting~\ref{setting:Strengthened_strong_a_priori_moment_bounds},
	let 
	$ p \in [1,\infty) $,   
	and
	for every $ s \in [0,T] $
	let 
		$ X_{s, (\cdot) }(\cdot) 
	=
	( X_{s,t}(\omega) )_{ (t,\omega)\in[s,T]\times\Omega }
	\colon 
	[s,T] \times \Omega
	\to H_\beta $ 
	be an $ ( \f_t )_{ t \in [s,T]} $-adapted 
	stochastic process
	with continuous sample paths 
	which satisfies for every
	$ t \in [s, T] $
	that
	$
	[
	X_{s, t}
	]_{\P, \mathcal{B}(H_\beta)}
	= \int_s^t B 
	\, 
	dW_u
	$.
	Then 
	it holds 
	for every
	$ t \in [0,T] $ 
	that
	\begin{equation}
	\begin{split}
	\label{eq:estimate3}
	&
	\Big\|
	\int_0^t
	\chi_{\llcorner u \lrcorner_\theta}
	e^{(t-\llcorner u \lrcorner_\theta)A}
	\Big( 
	\tfrac{ 4  
		\| \mathbb{B}
		X_{\llcorner u \lrcorner_\theta, u}  \|_U^2
		X_{\llcorner u \lrcorner_\theta, u}	
	}
	{(1 + \| X_{\llcorner u \lrcorner_\theta, u} \|_H^2)^3}
	-
	\tfrac{
		2 B  
		\mathbb{B}
		X_{\llcorner u \lrcorner_\theta, u}
		+ 
		\| B  \|_{\HS(U,H)}^2 
		X_{\llcorner u \lrcorner_\theta, u}
	}{
		( 1 + \| X_{\llcorner u \lrcorner_\theta, u} \|_H^2 )^2
	}
	\Big) 
	\, 
	du
	\Big\|_{\L^p(\P; H )}
	\\
	&
	\leq
	2
	\| B \|_{ \HS(U,H ) }^2
	t
	.
	\end{split}
	\end{equation}
\end{lemma}
\begin{proof}[Proof of Lemma~\ref{lemma:Some_helping_estimates}]
	Note that
	for every $ t \in [0,T] $
	it holds that
	\begin{equation}
	\begin{split}
	&
	\Big\|
	\int_0^t
	\chi_{\llcorner u \lrcorner_\theta}
	e^{(t-\llcorner u \lrcorner_\theta)A}
	\Big( 
	\tfrac{ 4  
		\|  \mathbb{B}
		X_{\llcorner u \lrcorner_\theta, u}  \|_U^2
		X_{\llcorner u \lrcorner_\theta, u}
	}
	{(1 + \| X_{\llcorner u \lrcorner_\theta, u} \|_H^2)^3}
	-
	\tfrac{
		2 B   
	\mathbb{B}
		X_{\llcorner u \lrcorner_\theta, u}
		+ 
		\| B \|_{\HS(U,H)}^2
		X_{\llcorner u \lrcorner_\theta, u} 
	}{
		( 1 + \| X_{\llcorner u \lrcorner_\theta, u} \|_H^2 )^2
	}
	\Big)  
	\, 
	du
	\Big\|_{\L^p(\P; H )}
	\\
	&
	\leq
	\int_0^t
	\Big\| 
	e^{(t-\llcorner u \lrcorner_\theta)A}
	\Big( 
	\tfrac{ 4  
		\|  \mathbb{B}
		X_{\llcorner u \lrcorner_\theta, u}  \|_U^2
		X_{\llcorner u \lrcorner_\theta, u}
	}
	{(1 + \| X_{\llcorner u \lrcorner_\theta, u} \|_H^2)^3}
	-
	\tfrac{
		2 B
		\mathbb{B}
		X_{\llcorner u \lrcorner_\theta, u}
		+ 
		\| B \|_{\HS(U,H)}^2 
		X_{\llcorner u \lrcorner_\theta, u}
	}{
		( 1 + \| X_{\llcorner u \lrcorner_\theta, u} \|_H^2 )^2
	}
	\Big) 
	\Big\|_{\L^p(\P; H )}
	\, 
	du
	\\
	&
	\leq
	\int_0^t 
	\Big\|  
	\tfrac{ 4  
		\|  \mathbb{B}
		X_{\llcorner u \lrcorner_\theta, u}  \|_U^2
		X_{\llcorner u \lrcorner_\theta, u}
	}
	{(1 + \| X_{\llcorner u \lrcorner_\theta, u} \|_H^2)^3}
	-
	\tfrac{
		2 B
		\mathbb{B}
		X_{\llcorner u \lrcorner_\theta, u}
		+ 
		\| B \|_{\HS(U,H)}^2 
		X_{\llcorner u \lrcorner_\theta, u}
	}{
		( 1 + \| X_{\llcorner u \lrcorner_\theta, u} \|_H^2 )^2
	}  
	\Big\|_{\L^p(\P; H )}
	\, 
	du
	\\
	&
	\leq 
	\int_0^t 
	\Big\|  
	\tfrac{ 4  
		\| \mathbb{B} \|_{L(H,U)}^2
		\| X_{\llcorner u \lrcorner_\theta, u}  \|_{H}^3 }
	{(1 + \| X_{\llcorner u \lrcorner_\theta, u} \|_H^2)^3}
	+
	\tfrac{
		(
		2
		\| B \mathbb{B} \|_{ L(H, H)} 
		+
		\| B \|_{\HS(U,H)}^2
		)
		\| X_{\llcorner u \lrcorner_\theta, u} \|_H
	}{
		( 1 + \| X_{\llcorner u \lrcorner_\theta, u} \|_H^2 )^2
	} 
	\Big\|_{\L^p(\P; \R)}
	\, 
	du
	\\
	&
	\leq 
	\int_0^t  
	\big(  
	\tfrac{  
		\| B \|_{\HS(U,H)}^2  
	}{2}
	+  
	\| B \|_{ L(U, H ) }
	\| \mathbb{B} \|_{ L(H, U)}
	+
	\tfrac{ 
		\| B \|_{\HS(U,H)}^2
	}{2}
	\big)  
	\, 
	du
	.
	\end{split}
	\end{equation}
	This completes the proof of Lemma~\ref{lemma:Some_helping_estimates}.
\end{proof}
\begin{lemma}
	\label{lemma:some_helping_estimates2}
	Assume Setting~\ref{setting:Strengthened_strong_a_priori_moment_bounds}
	and
		for every $ s \in [0,T] $
	let 
	$ X_{s, (\cdot) }(\cdot) 
	=
	( X_{s,t}(\omega) )_{ (t,\omega)\in[s,T]\times\Omega }
	\colon 
	[s,T] \times \Omega
	\to H_\beta $  
	be an $ ( \f_t )_{ t \in [s,T]} $-adapted 
	stochastic process
	with continuous sample paths 
	which satisfies for every 
	$ t \in [s, T] $
	that
	$
	[
	X_{s, t}
	]_{\P, \mathcal{B}( H_\beta )}
	= \int_s^t B 
	\, 
	dW_u
	$.
	Then 
	it holds 
	for every 
	$ n \in \N $,
	$ t \in [0,T] $
	that
	\begin{equation}
	\begin{split}  
	& 
	\Big\|
	\int_0^t 
	\chi_{ \llcorner u \lrcorner_\theta }
	e^{(t- \llcorner u \lrcorner_\theta )A} 
	\,
	\Big[ 
	\tfrac{B }{1+\| X_{{\llcorner u \lrcorner_\theta}, u} \|_H^2}  
	-
	\tfrac{2 X_{{\llcorner u \lrcorner_\theta}, u} \langle X_{{\llcorner u \lrcorner_\theta}, u}, B ( \cdot ) \rangle_H}
	{ ( 1 + \| X_{{\llcorner u \lrcorner_\theta}, u} \|_H^2 )^2 }
	\Big] 
	\, 
	dW_u
	\Big\|_{L^{2n}(\P; H )}^{2n}
	\\
	&  
	\leq 
	\tfrac{ 9^n ( 2n) !  }{ 8^n n! }
	\| B \|_{ \HS(U,H)}^{2n}
	t^n 
	.
	\end{split}
	\end{equation}
\end{lemma}
\begin{proof}[Proof of Lemma~\ref{lemma:some_helping_estimates2}]
	Throughout this proof  
	let $ \mathbb{U} \subseteq U $ be an orthonormal basis of $ U $,
	let
	$ Z \colon [0,T] \times \Omega \to
	H $
	be an 
	$ ( \f_t )_{ t \in [0,T] } $-adapted
	stochastic process 
	which satisfies for every
	$ t \in [0,T] $ that
	\begin{equation}
	\begin{split}  
	[ Z_t ]_{ \P, \B(H ) }
	= 
	\int_0^t 
	\chi_{\llcorner u \lrcorner_\theta}
	e^{(t-\llcorner u \lrcorner_\theta)A} 
	\,
	\Big[ 
	\tfrac{B }{1+\| X_{{\llcorner u \lrcorner_\theta}, u} \|_H^2}  
	-
	\tfrac{2 X_{{\llcorner u \lrcorner_\theta}, u} \langle X_{{\llcorner u \lrcorner_\theta}, u}, B 
		( \cdot )  \rangle_H}
	{ ( 1 + \| X_{{\llcorner u \lrcorner_\theta}, u} \|_H^2 )^2 }
	\Big] 
	\, 
	dW_u
	,  
	\end{split}
	\end{equation}
	and let
	$ \mathcal{Z} \colon [0,T] \times \Omega \to
	H $
	be an
	$ ( \f_t )_{ t \in [0,T] } $-adapted 
	stochastic process
	with left-continuous sample paths
	and finite right limits
	which satisfies for every
	$ t \in [0,T] $ that
	\begin{equation}
	[ \mathcal{Z}_t ]_{ \P, \B( H ) }
	= 
	[ Z_{ \llcorner t \lrcorner_\theta } ]_{ \P, \B( H ) }
	+ 
	\int_{ \llcorner t \lrcorner_\theta }^t 
	\chi_{ \llcorner t \lrcorner_\theta }
	\Big[ 
	\tfrac{B }{1+\| X_{ \llcorner t \lrcorner_\theta, u} \|_H^2}  
	-
	\tfrac{2 X_{ \llcorner t \lrcorner_\theta, u} \langle X_{ \llcorner t \lrcorner_\theta, u}, B 
		( \cdot ) \rangle_H}
	{ ( 1 + \| X_{ \llcorner t \lrcorner_\theta, u} \|_H^2 )^2 }
	\Big] 
	\, 
	dW_u 
	.   
	\end{equation}
	Note that
	It\^o's formula
	proves for every
	$ p \in [2, \infty) $,
	$ t \in [0,T] $ that
	\begin{equation}
	\begin{split}
	\label{eq:ApplyIto}
	&
 [	\| \mathcal{Z}_t \|_H^p ]_{ \P, \B(\R) }
	=
[	\| Z_{ \llcorner t \lrcorner_{ \theta } } \|_H^p ]_{ \P, \B(\R) }
	\\
	&
	+
	\int_{ \llcorner t \lrcorner_{ \theta } }^t
	\chi_{ \llcorner u \lrcorner_{ \theta } }
	p
	\| \mathcal{Z}_u \|_H^{ p - 2 }
	\left < 
	\mathcal{Z}_u,
	\tfrac{B }{1+\| X_{ \llcorner t \lrcorner_\theta, u} \|_H^2}  
	-
	\tfrac{2 X_{ \llcorner t \lrcorner_\theta, u} \langle X_{ \llcorner t \lrcorner_\theta, u}, B 
		( \cdot ) \rangle_H}
	{ ( 1 + \| X_{ \llcorner t \lrcorner_\theta, u} \|_H^2 )^2 }
	\right >_H
	\, dW_u
	\\
	&
	+
	\bigg[ 
	\tfrac{1}{2}
	\int_{ \llcorner t \lrcorner_{ \theta } }^t
	\bigg[
	\chi_{ \llcorner u \lrcorner_{ \theta } }
	\sum_{ { \bf u } \in \mathbb{U} }
	\bigg\{
	p
	\| \mathcal{Z}_u \|_H^{p-2}
	\Big\| 
	\tfrac{ B { \bf u } }{1+\| X_{ \llcorner t \lrcorner_\theta, u} \|_H^2}  
	-
	\tfrac{2 X_{ \llcorner t \lrcorner_\theta, u} \langle X_{ \llcorner t \lrcorner_\theta, u}, B 
		{ \bf u } \rangle_H}
	{ ( 1 + \| X_{ \llcorner t \lrcorner_\theta, u} \|_H^2 )^2 }
	\Big\|_H^2
	\\
	&
	+
	p ( p - 2 )
	\chi_{ \llcorner t \lrcorner_\theta }
	\1_{ \{ \mathcal{Z}_u \neq 0 \} }
	\,
	\| \mathcal{Z}_u \|_H^{p-4}
	\Big|
	\left <
	\mathcal{Z}_u,
	\tfrac{B { \bf u } }{1+\| X_{ \llcorner t \lrcorner_\theta, u} \|_H^2}  
	-
	\tfrac{2 X_{ \llcorner t \lrcorner_\theta, u} \langle X_{ \llcorner t \lrcorner_\theta, u}, B 
		{ \bf u } \rangle_H}
	{ ( 1 + \| X_{ \llcorner t \lrcorner_\theta, u} \|_H^2 )^2 }
	\right >_H \Big|^2
	\bigg\}
	\bigg] 
	\, du
	\bigg]_{ \P, \B(\R) }
	.
	\end{split}
	\end{equation}
	In addition, observe that the
	Burkholder-Davis-Gundy type inequality in Lemma~7.7 in Da Prato \& Zabczyk~\cite{dz92}
	ensures
	 for every
	$ p \in [2, \infty) $,
	$ t \in [0,T] $ that
	\begin{equation} 
	\begin{split}
	\label{eq:UnsuresMartingale}
	&
	\| \mathcal{Z}_t \|_{ \L^p(\P; H) }
	\leq 
	\| Z_{ \llcorner t \lrcorner_\theta } 
	\|_{ \L^p( \P; H) }
	+
	\Big \| 
	\int_{ \llcorner t \lrcorner_\theta }^t 
	\chi_{ \llcorner t \lrcorner_\theta }
	\Big[ 
	\tfrac{B }{1+\| X_{ \llcorner t \lrcorner_\theta, u} \|_H^2}  
	-
	\tfrac{2 X_{ \llcorner t \lrcorner_\theta, u} \langle X_{ \llcorner t \lrcorner_\theta, u}, B 
		( \cdot ) \rangle_H}
	{ ( 1 + \| X_{ \llcorner t \lrcorner_\theta, u} \|_H^2 )^2 }
	\Big] 
	\, 
	dW_u  
	\Big \|_{ L^p( \P; H) }
	\\
	& 
	\leq 
	\Big[ 
	\tfrac{ p^2 }{ 2 }
	\int_0^{ \llcorner t \lrcorner_\theta } 
	\Big\| 
	\tfrac{B }{1+\| X_{ \llcorner u \lrcorner_\theta, u} \|_H^2}  
	-
	\tfrac{2 X_{ \llcorner u \lrcorner_\theta, u} 
		\langle X_{ \llcorner u \lrcorner_\theta, u}, B 
		( \cdot ) \rangle_H}
	{ ( 1 + \| X_{ \llcorner u \lrcorner_\theta, u} \|_H^2 )^2 }
	\Big\|_{ \L^p( \P; \HS(U,H) ) }^2
	\, 
	d u
	\Big]^{ \nicefrac{1}{2} }
	\\
	&
	\quad 
	+
	\Big[ 
	\tfrac{ p^2 }{ 2 }
	\int_{ \llcorner t \lrcorner_\theta }^t 
	\Big\| 
	\tfrac{B }{1+\| X_{ \llcorner t \lrcorner_\theta, u} \|_H^2}  
	-
	\tfrac{2 X_{ \llcorner t \lrcorner_\theta, u} \langle X_{ \llcorner t \lrcorner_\theta, u}, B 
		( \cdot ) \rangle_H}
	{ ( 1 + \| X_{ \llcorner t \lrcorner_\theta, u} \|_H^2 )^2 }
	\Big\|_{ \L^p( \P; \HS(U,H) ) }^2
	\, 
	d u
	\Big]^{ \nicefrac{1}{2} }
	\\
	&
	\leq
	2
	\Big[ 
	\tfrac{ p^2 }{ 2 }
	\int_0^t  
	\Big\| 
	\tfrac{ \| B \|_{\HS(U,H)} }{1+\| X_{ \llcorner u \lrcorner_\theta, u} \|_H^2}  
	+
	\tfrac{2 
		\| X_{ \llcorner u \lrcorner_\theta, u} \|_H^2
		\| B \|_{ \HS(U,H) } }
	{ ( 1 + \| X_{ \llcorner u \lrcorner_\theta, u} \|_H^2 )^2 }
	\Big\|_{ \L^p( \P; \HS(U,H) ) }^2
	\, 
	d u
	\Big]^{ \nicefrac{1}{2} }
	\\
	&
	\leq 
	2
	\Big[ 
	\tfrac{ p^2 }{ 2 }
	\int_0^t  
	4 \| B \|_{ \HS(U,H) }^2
	\, 
	d u
	\Big]^{ \nicefrac{1}{2} }
	\leq
	2 \sqrt{2} p
	\max \{ T, 1 \} 
	\| B \|_{ \HS(U,H) } 
	.
	\end{split}
	\end{equation}
	This implies for every
	$ p \in [2,\infty) $, 
	$ t \in [0,T] $ 
	that
	\begin{equation} 
	\begin{split} 
	\label{eq:EstablishMartingale1}
	&
	\int_{ \llcorner t \lrcorner_{ \theta } }^t
	\Big\| 
	\chi_{ \llcorner u \lrcorner_{ \theta } }
	\| \mathcal{Z}_u \|_H^{ p - 2 }
	\left < 
	\mathcal{Z}_u,
	\tfrac{B }{1+\| X_{ \llcorner t \lrcorner_\theta, u} \|_H^2}  
	-
	\tfrac{2 X_{ \llcorner t \lrcorner_\theta, u} \langle X_{ \llcorner t \lrcorner_\theta, u}, B 
		( \cdot ) \rangle_H}
	{ ( 1 + \| X_{ \llcorner t \lrcorner_\theta, u} \|_H^2 )^2 }
	\right >_H
	\Big\|_{ \L^2(\P; \HS(U,H) ) }^2
	\, du
	\\
	&
	\leq
	\int_{ \llcorner t \lrcorner_{ \theta } }^t
	\Big\|  
	\| \mathcal{Z}_u \|_H^{ p - 2 }
	\left < 
	\mathcal{Z}_u,
	\tfrac{B }{1+\| X_{ \llcorner t \lrcorner_\theta, u} \|_H^2}  
	-
	\tfrac{2 X_{ \llcorner t \lrcorner_\theta, u} \langle X_{ \llcorner t \lrcorner_\theta, u}, B 
		( \cdot ) \rangle_H}
	{ ( 1 + \| X_{ \llcorner t \lrcorner_\theta, u} \|_H^2 )^2 }
	\right >_H
	\Big\|_{ \L^2(\P; \HS(U,H) ) }^2
	\, du
	\\
	&
	\leq
	\int_{ \llcorner t \lrcorner_{ \theta } }^t
	\Big\|  
	\| \mathcal{Z}_u \|_H^{ p - 1 }
	\big\|
	\tfrac{B }{1+\| X_{ \llcorner t \lrcorner_\theta, u} \|_H^2}  
	-
	\tfrac{2 X_{ \llcorner t \lrcorner_\theta, u} \langle X_{ \llcorner t \lrcorner_\theta, u}, B 
		( \cdot ) \rangle_H}
	{ ( 1 + \| X_{ \llcorner t \lrcorner_\theta, u} \|_H^2 )^2 }
	\big\|_{ \HS(U,H) }
	\Big\|_{ \L^2(\P; \R ) }^2
	\, du
	\\
	&
	\leq
	\int_{ \llcorner t \lrcorner_{ \theta } }^t
	\Big\|  
	\| \mathcal{Z}_u \|_H^{ p - 1 }
	\big(
	\tfrac{ \| B \|_{\HS(U,H)} }{1+\| X_{ \llcorner t \lrcorner_\theta, u} \|_H^2}  
	+
	\tfrac{2 
		\| X_{ \llcorner t \lrcorner_\theta, u} \|_H^2
		\| B \|_{ \HS(U,H) } }
	{ ( 1 + \| X_{ \llcorner t \lrcorner_\theta, u} \|_H^2 )^2 }
	\big)
	\Big\|_{ \L^2(\P; \R ) }^2
	\, du
	\\
	&
	\leq 
	4
	\| B \|_{ \HS(U,H) }^2
	\int_{ \llcorner t \lrcorner_{ \theta } }^t
	\|  
	\mathcal{Z}_u
	\|_{ \L^{ 2(p-1) }(\P; H ) }^{2(p-1)}
	\, du
	<
	\infty 
	.
	\end{split} 
	\end{equation} 
		Moreover, note that for every
	$ t \in [0,T] $ it holds that
	\begin{equation}
	\begin{split}  
	[ Z_t ]_{ \P, \B(H ) }
	& = 
	e^{(t-\llcorner t \lrcorner_\theta)A} 
	\bigg[ 
	[ Z_{ \llcorner t \lrcorner_\theta } ]_{ \P, \B(H) }
	+
	\int_{ \llcorner t \lrcorner_\theta }^t 
	\chi_{ \llcorner t \lrcorner_\theta }
	\Big[ 
	\tfrac{B }{1+\| X_{ \llcorner t \lrcorner_\theta, u} \|_H^2}  
	-
	\tfrac{2 X_{ \llcorner t \lrcorner_\theta, u} \langle X_{ \llcorner t \lrcorner_\theta, u}, B 
		( \cdot ) \rangle_H}
	{ ( 1 + \| X_{ \llcorner t \lrcorner_\theta, u} \|_H^2 )^2 }
	\Big] 
	\, 
	dW_u   
	\bigg] 
	.
	\end{split}
	\end{equation}
	Combining this, 
	\eqref{eq:ApplyIto}, 
	\eqref{eq:EstablishMartingale1},
	and Tonelli's theorem
	establishes that for every
	$ p \in [2, \infty) $,
	$ t \in [0,T] $ it holds that
	\begin{equation}
	\begin{split}
	& 
	\E[ 
	\| Z_t \|_H^p
	]
	=
	\E [ 
	\| 
	e^{ ( t- \llcorner t \lrcorner_{ \theta } )A } 
	\mathcal{Z}_t 
	\|_H^p	
	]
	\leq 
	\E[ 
	\| \mathcal{Z}_t \|_H^p
	]
	=
	\E[ 
	\| Z_{ \llcorner t \lrcorner_{ \theta } } \|_H^p
	]
	\\
	&
	+
	\tfrac{1}{2}
	\int_{ \llcorner t \lrcorner_{ \theta } }^t
	\E
	\bigg[
	\chi_{ \llcorner u \lrcorner_{ \theta } }
	\sum_{ { \bf u } \in \mathbb{U} }
	\bigg\{
	p
	\| \mathcal{Z}_u \|_H^{p-2}
	\Big\| 
	\tfrac{ B { \bf u } }{1+\| X_{ \llcorner t \lrcorner_\theta, u} \|_H^2}  
	-
	\tfrac{2 X_{ \llcorner t \lrcorner_\theta, u} \langle X_{ \llcorner t \lrcorner_\theta, u}, B 
		{ \bf u } \rangle_H}
	{ ( 1 + \| X_{ \llcorner t \lrcorner_\theta, u} \|_H^2 )^2 }
	\Big\|_H^2
	\\
	&
	+
	p ( p - 2 )
	\chi_{ \llcorner t \lrcorner_\theta }
	\1_{ \{ \mathcal{Z}_u \neq 0 \} }
	\,
	\| \mathcal{Z}_u \|_H^{p-4}
	\Big|
	\left <
	\mathcal{Z}_u,
	\tfrac{B { \bf u } }{1+\| X_{ \llcorner t \lrcorner_\theta, u} \|_H^2}  
	-
	\tfrac{2 X_{ \llcorner t \lrcorner_\theta, u} \langle X_{ \llcorner t \lrcorner_\theta, u}, B 
		{ \bf u } \rangle_H}
	{ ( 1 + \| X_{ \llcorner t \lrcorner_\theta, u} \|_H^2 )^2 }
	\right >_H \Big|^2
	\bigg\}
	\bigg] 
	\, du
	.
	\end{split} 
	\end{equation} 
	The Cauchy-Schwarz inequality hence assures for every
	$ p \in [2, \infty) $,
	$ t \in [0,T] $ 
	that
	\begin{equation}
	\begin{split}
	\label{eq:Demonstrate}
	& 
	\E[ 
	\| Z_t \|_H^p
	] 
	\leq 
	\E[ 
	\| \mathcal{Z}_t \|_H^p
	]
	\leq
	\E[ 
	\| Z_{ \llcorner t \lrcorner_{ \theta } } \|_H^p
	]
	\\
	&
	\quad 
	+
	\tfrac{1}{2}
	\int_{ \llcorner t \lrcorner_{ \theta } }^t
	\E
	\bigg[ 
	\sum_{ { \bf u } \in \mathbb{U} }
	\bigg\{
	p
	\| \mathcal{Z}_u \|_H^{p-2}
	\Big( 
	\tfrac{ \| B { \bf u } \|_H  }{ 1+\| X_{ \llcorner t \lrcorner_\theta, u} \|_H^2 }  
	+
	\tfrac{2  
		\| X_{ \llcorner t \lrcorner_\theta, u} \|_H  
		\| X_{ \llcorner t \lrcorner_\theta, u} \|_H   \|B 
		{ \bf u } \|_H  }
	{ ( 1 + \| X_{ \llcorner t \lrcorner_\theta, u} \|_H^2 )^2 }
	\Big)^2
	\\
	&
	\quad
	+ 
	p ( p - 2 )
	\| \mathcal{Z}_u \|_H^{p-2}
	\Big(
	\tfrac{ \| B { \bf u } \|_H  }{  1+\| X_{ \llcorner t \lrcorner_\theta, u} \|_H^2  }  
	+
	\tfrac{ 2 \| X_{ \llcorner t  \lrcorner_\theta, u} 
		\|_H 
		\| X_{ \llcorner t \lrcorner_\theta, u} \|_H  
		\| B 
		{ \bf u } \|_H  }
	{ ( 1 + \| X_{ \llcorner t \lrcorner_\theta, u} \|_H^2 )^2 }
	\Big)^2
	\bigg\}
	\bigg] 
	\, du
	\\
	&
	\leq 
	\E[ 
	\| Z_{ \llcorner t \lrcorner_{ \theta } } \|_H^p
	]
	+
	p ( p - 1) 
	\big[ 
	\tfrac{1}{2}
	\big]
	\| B \|_{ \HS(U,H) }^2
	\int_{ \llcorner t \lrcorner_{ \theta } }^t
	\E
	\bigg[  
	\| \mathcal{Z}_u \|_H^{p-2}
	\Big( 
	\tfrac{ 1 }{  1+\| X_{ \llcorner t \lrcorner_\theta, u} \|_H^2  }  
	+
	\tfrac{ 2 \| X_{ \llcorner t \lrcorner_\theta, u} \|_H^2
	}
	{ ( 1 + \| X_{ \llcorner t \lrcorner_\theta, u} \|_H^2 )^2 }
	\Big)^2 
	\bigg] 
	\, du
	\\
	&
	\leq  
	\E[ 
	\| Z_{ \llcorner t \lrcorner_{ \theta } } \|_H^p
	]
	+
	p ( p- 1)
	\big[
	\tfrac{9}{8} 
	\big]
	\| B \|_{ \HS(U,H)}^2
	\int_{ \llcorner t \lrcorner_{ \theta } }^t
	\E
	[  
	\| \mathcal{Z}_u \|_H^{p-2}
	]
	\, du 
	.
	\end{split}
	\end{equation}
	This
	demonstrates that for every
	$ n \in \N $,
	$ t \in [0,T] $ it holds that
	\begin{equation}
	\begin{split} 
	\E[ 
	\| Z_t \|_H^{2n}
	]
	\leq 
	\E[
	\| \mathcal{Z}_t \|_H^{2n}
	]	
	&\leq 
	\E[ 
	\| Z_{ \llcorner t \lrcorner_{ \theta } } \|_H^{2n}
	]
	+
	2n ( 2n - 1)
	\big[ \tfrac{9}{8} \big]
	\| B \|_{ \HS(U,H)}^2
	\int_{ \llcorner t  \lrcorner_{ \theta } }^{t }
	\E
	\big [  
	\| \mathcal{Z}_u \|_H^{2n-2}
	\big ]
	\, du 
	\\
	&
	\leq
	\ldots
	\leq
	\E[ 
	\| Z_0 \|_H^{2n}
	]
	+
	2n ( 2n - 1)
	\big[ \tfrac{9}{8} \big]
	\| B \|_{ \HS(U,H)}^2
	\int_0^{t }
	\E
	\big [  
	\| \mathcal{Z}_u \|_H^{2n-2}
	\big ]
	\, du 
	\\
	&
	=
    2n ( 2n - 1)
	\big[ \tfrac{9}{8} \big]
	\| B \|_{ \HS(U,H)}^2
	\int_0^{t }
	\E
	\big [  
	\| \mathcal{Z}_u \|_H^{2n-2}
	\big ]
	\, du 
	.
	\end{split}
	\end{equation} 
	Therefore, we obtain that for every
	$ n \in \N $, $ t_0 \in [0,T] $ it holds that
	\begin{equation} 
	\begin{split} 
	&
	\E[ 
	\| Z_{t_0} \|_H^{2n}
	]
	\leq
	\E[
	\| \mathcal{Z}_{t_0} \|_H^{2n}
	]
	\leq  
	2n ( 2n - 1 ) \big[ \tfrac{9}{8} \big]
	\| B \|_{ \HS(U,H)}^2
	\int_0^{ t_0 }
	\E
	\big [  
	\| \mathcal{Z}_u \|_H^{2n-2}
	\big ]
	\, du 
	\\
	&
	\leq \ldots
	\leq  
	(2n)!
	\big[ \tfrac{9}{8} \big]^n
	\| B \|_{ \HS(U,H)}^{2n}
	\int_0^{ t_0 }
	\int_0^{t_1}
	\cdots
	\int_0^{t_{n-1}}
	\, dt_n \cdots dt_2 \, dt_1
	=
	\tfrac{ 9^n ( 2n) !}{ 8^n n! }
	\| B \|_{ \HS(U,H)}^{2n}
	( t_0 )^n
	.
	\end{split} 
	\end{equation}
	The proof of Lemma~\ref{lemma:some_helping_estimates2}
	is thus completed.
\end{proof}
\begin{lemma}
	\label{lemma:Exp_regularity}
	Assume Setting~\ref{setting:Strengthened_strong_a_priori_moment_bounds}
	and let
	$ \varepsilon \in [0, \nicefrac{ 1 }{ (8 [
		\max \{ \| B \|_{ \HS(U,H) }, 1 \} ]^2
		\max \{ T, 1 \} )^2 } ) $.
	Then  it holds for every
	$ t \in [0,T] $ that
	\begin{equation}
	\begin{split}
	\E[ e^{ \varepsilon \| \O_t \|_H^2 } ] 
	&
	\leq
	\tfrac{ 2 }{ 1 - \varepsilon^2 [ 8 \| B \|_{ \HS(U,H) }
		\max \{ \| B \|_{ \HS(U,H) }, 1 \}
		\max \{ T, 1 \} ]^4 }  
	.
	\end{split}
	\end{equation}
\end{lemma}
\begin{proof}[Proof of Lemma~\ref{lemma:Exp_regularity}]
	Throughout this proof
	for every $ s \in [0,T] $ 
	let
	$ X_{s, (\cdot) }(\cdot) 
	=
	( X_{s,t}(\omega) )_{ (t,\omega)\in[s,T]\times\Omega }
	\colon $
	$ [s,T] \times \Omega
	\to H_\beta $
	be
	an
	$ ( \f_u )_{ u \in [s,T] } $-adapted
	stochastic process with continuous sample paths
	which satisfies for every
	$ u \in [s,T] $
	that
	$ [ X_{s,u} ]_{\P, \mathcal{B}(H_\beta) }
	=
	\int_s^u B \, dW_\tau $. 
Lemma~\ref{lemma:Ito_representation}
	(with
	$ X_{s, t} = X_{s, t} $
	for
	$ t \in [s,T] $,
	$ s \in [0,T] $
	in the notation of 
	Lemma~\ref{lemma:Ito_representation}), 
	Lemma~\ref{lemma:Some_helping_estimates}
	(with
	$ p = 2 n $,
	$ X_{s, t} = X_{s, t} $
	for
	$ t \in [s,T] $, 
	$ s \in [0,T] $, 
	$ n \in \N $
	in the notation of
	Lemma~\ref{lemma:Some_helping_estimates}), 
	Lemma~\ref{lemma:some_helping_estimates2} 
	(with 
	$ X_{s, t} = X_{s, t} $
	for
	$ t \in [s,T] $,
	$ s \in [0,T] $,  
	$ n \in \N $
	in the notation of 
	Lemma~\ref{lemma:some_helping_estimates2}),
	and the triangle inequality
	ensure that for every
	$ n \in \N $,
	$ t \in [0,T] $
	it holds that 
	\begin{equation}
	\begin{split}
	\label{eq:same_exp11}
	&
	\| \O_t \|_{ \L^{2n}( \P; H) }
	%
	\leq 
	\Big\|
	\int_0^t
	\chi_{\llcorner u \lrcorner_\theta}
	\Big(
	\tfrac{
		e^{(t- \llcorner u \lrcorner_\theta)A} 
		B  }{1+\| X_{{\llcorner u \lrcorner_\theta}, u} \|_H^2}
	-
	\tfrac{2 
		e^{(t- \llcorner u \lrcorner_\theta)A} 
		X_{{\llcorner u \lrcorner_\theta}, u} \langle X_{{\llcorner u \lrcorner_\theta}, u}, B  (\cdot) \rangle_H}
	{ ( 1 + \| X_{{\llcorner u \lrcorner_\theta}, u} \|_H^2 )^2 }
	\Big)
	\,
	dW_u
	\Big\|_{L^{2n}(\P; H )}
	\\
	&
	+ 
	\Big\|
	\int_0^t
	\chi_{ \llcorner u \lrcorner_\theta }
	e^{(t-\llcorner u \lrcorner_\theta)A}
	\Big[ 
	\tfrac{ 4  
		\|  \mathbb{B}
		X_{\llcorner u \lrcorner_\theta, u}  \|_U^2
		X_{\llcorner u \lrcorner_\theta, u}
	}
	{(1 + \| X_{\llcorner u \lrcorner_\theta, u} \|_H^2)^3}
	-
	\tfrac{
		2 B 
		\mathbb{B} 
		X_{\llcorner u \lrcorner_\theta, u}
		+ 
		\| B \|_{\HS(U,H)}^2
		X_{\llcorner u \lrcorner_\theta, u} 
	}{
		( 1 + \| X_{\llcorner u \lrcorner_\theta, u} \|_H^2 )^2
	}
	\Big]
	\, 
	du
	\Big\|_{\L^{2n}(\P; H )}
	\\
	&
	\leq 
	\tfrac{3}{2 \sqrt{2}} \big( \tfrac{ ( 2n )! }{ n! } \big)^{\nicefrac{1}{2n}}
	\| B \|_{ \HS(U,H) } t^{ \nicefrac{1}{2} }
	+
	2 \| B \|_{ \HS(U,H) }^2 t
	\\
	&
	\leq
	4 \| B \|_{ \HS(U,H) }
	\max \{ \| B \|_{ \HS(U,H) }, 1 \}
	\max \{ T, 1 \}
	\big( \tfrac{ ( 2n )! }{ n! } \big)^{\nicefrac{1}{2n}}
	.
	\end{split}
	\end{equation}
	This, the fact that for every
	$ x \in [0, \infty) $ it holds that
	$ 
	e^x \leq 2  [ \smallsum_{m=0}^\infty
	\tfrac{x^{2m}}{(2m)!} ] 
	$
	(see, e.g., Lemma~2.4 in Hutzenthaler et al.\ \cite{HutzenthalerJentzenWang2017Published}),
	the dominated convergence theorem, 
	and the fact that
	$ \forall
	\, m \in \N\colon (4m)! \leq 2^{4m} [(2m)!]^2 $
	imply that for every
	$ t \in [0,T] $ it holds that
	\begin{equation}
	\begin{split}
	&
	\E[ e^{ \varepsilon \| \O_t \|_H^2 } ] 
	\leq
	2 \,
	\E \bigg[
	\sum_{ n=0 }^\infty 
	\tfrac{ \varepsilon^{2n} \| \O_t \|_H^{4n} }{ (2n)! }
	\bigg] 
	=
	2
	\left[
	\sum_{ n=0 }^\infty 
	\E \Big[
	\tfrac{ \varepsilon^{2n} \| \O_t \|_H^{4n} }{ (2n)! }
	\Big] 
	\right] 
	\\
	&
	\leq
	2
	\left[
	\sum_{n=0}^\infty 
	\tfrac{ (4n)! }   { [ (2n) ! ]^2 }
	\varepsilon^{2n}
	[ 4   \| B \|_{ \HS(U,H) }
	\max \{ \| B \|_{ \HS(U,H) }, 1 \}
	\max \{ T, 1 \} ]^{4n}
	\right]
	\\
	&
	\leq
	2
	\left[
	\sum_{n=0}^\infty 
	2^{4n}
	\varepsilon^{2n}
	[ 4 \| B \|_{ \HS(U,H) }
	\max \{ \| B \|_{ \HS(U,H) }, 1 \}
	\max \{ T, 1 \} ]^{4n}
	\right] 
	\\
	&
	=
	\tfrac{ 2 }{ 1 - \varepsilon^2 [ 8  \| B \|_{ \HS(U,H) }
		\max \{ \| B \|_{ \HS(U,H) }, 1 \}
		\max \{ T, 1 \} ]^4 }
	.
	\end{split}
	\end{equation}
	The proof of Lemma~\ref{lemma:Exp_regularity}
	is thus completed.
\end{proof}
\section[Regularity properties of space-time 
approximations of stochastic convolutions]{Regularity properties of
	tamed-truncated space-time 
	approximations of stochastic convolutions}
\label{section:Conslusion}
\begin{setting}
\label{setting:Full_discrete}
Assume Setting~\ref{setting:main},
let
$ \beta \in [0, \infty) $,
$ T \in (0,\infty) $, 
let
$ ( \Omega, \F, \P, ( \f_t )_{t \in [0,T]} ) $ 
be a filtered probability space which fulfills the usual conditions,
let $ (W_t)_{t\in [0,T]} $
be an $ \operatorname{Id}_U $-cylindrical   
$ ( \f_t )_{t\in [0,T]} $-Wiener process,  
let
$ B \in  \HS(U, H_\beta) $,
let
$ \mathbb{U} \subseteq U $
be an orthonormal basis of $ U $,
let 
$ P_I \colon H \to H $,
$ I \in \mathcal{P}(\H) $, 
and
$ \hat P_J \colon U \to U $,
$ J \in \mathcal{P}( \mathbb{U} ) $,
be the linear operators 
which satisfy for every 
$ x \in H $,
$ y \in U $,
$ I \in \mathcal{P}(\H) $,
$ J \in \mathcal{P}(\mathbb{U} ) $ 
that
$ P_I(x)
= \sum_{h \in I} \langle h, x \rangle_H h $
and
$ \hat P_J(y)
= \sum_{u \in J} \langle u, y \rangle_U u $,
let
$ \chi^{\theta, I, J} \colon [0,T] \times \Omega \to [0,1] $,
$ \theta \in \varpi_T $,
$ I \in \mathcal{P}_0(\H) $,
$ J \in \mathcal{P}(\mathbb{U} ) $, 
be 
$ ( \f_t )_{t\in [0,T]} $-adapted
stochastic 
processes,
and
let
$ \O^{\theta, I, J } \colon [0,T] \times \Omega \to P_I( H ) $,
$ \theta \in \varpi_T $,
$ I \in \mathcal{P}_0(\H) $,
$ J \in \mathcal{P}( \mathbb{U} ) $, 
be  
stochastic 
processes  
which satisfy for every 
$ \theta \in \varpi_T $,
$ I \in \mathcal{P}_0(\H) $, 
$ J \in \mathcal{P}( \mathbb{U} ) $,
$ t \in [0,T] $
that 
$ \O_0^{\theta,I, J} = 0 $
and
\begin{equation} 
\begin{split}  
\label{eq:Scheme}  
[
\O_t^{\theta, I, J}  
]_{\P, \mathcal{B}( P_I(H) ) }
=   
[
e^{(t-\llcorner t \lrcorner_\theta )A}
\O^{\theta, I, J}_{
	\llcorner t \lrcorner_\theta 
}  
]_{\P, \mathcal{B}( P_I(H) ) }
+
\frac{
	\int_{ \llcorner t \lrcorner_\theta  }^t
	\chi_{\llcorner t \lrcorner_\theta }^{
		\theta, I, J}
	e^{(t- \llcorner t \lrcorner_\theta  )A}
	P_I B \hat P_J
	\, 
	dW_s
}{
1 + 
\| 
\int_{ \llcorner t \lrcorner_\theta }^t
P_I B \hat P_J
\, 
dW_s 
\|_H^2
} 
.
\end{split}
\end{equation}
\end{setting}
\subsection{Main results}
Here we apply the results from Section~\ref{section:StrengthenedAprioriBound} in order to obtain our main results concerning 
tamed-truncated 
space-time 
approximations (see~\eqref{eq:Scheme} above) of stochastic convolutions. 
A uniform boundedness of moments in fractional order smoothness spaces is presented in Corollary~\ref{Corollary:FiniteMoments}, a uniform H\"older continuity in time is shown in Corollary~\ref{corollary:App_noise_regularity}, strong convergence rates are established in Corollary~\ref{corollary:noise_approximation_converges}, and Corollary~\ref{corollary:ExpMomentsNoise} concerns a uniform boundedness of exponential moments.
\begin{corollary}
	\label{Corollary:FiniteMoments}
	Assume Setting~\ref{setting:Full_discrete}
	and
	let
	$ p \in [1, \infty) $, 
	$ \gamma \in [0, \nicefrac{1}{2} + \beta ) $. 
	Then it holds that
	\begin{equation} 
	\label{eq:Apriori1}
	\sup\nolimits_{\theta \in \varpi_T }
	\sup\nolimits_{ J \in \mathcal{P}( \mathbb{U} ) }
	\sup\nolimits_{ I \in \mathcal{P}_0(\H) }
	\sup\nolimits_{ t \in [0,T] }
	\| \O_t^{ \theta, I, J } \|_{ \L^{p}( \P; H_{ \gamma } )} 
	< \infty 
	.
	\end{equation}
\end{corollary}
\begin{proof}[Proof of Corollary~\ref{Corollary:FiniteMoments}]
	Observe that
Lemma~\ref{lemma:more_regularity}
(with 
$ \beta = \beta $,
$ \gamma = \max \{ \gamma, \beta \} $,
$ T = T $,
$ \theta = \theta $,
$ ( \Omega, \F, \P ) =
( \Omega, \F, \P ) $,
$ ( \f_u )_{ u \in [0,T]} 
=
( \f_u )_{ u \in [0,T]} $,
$ ( W_u )_{ u \in [0,T] } 
=
( W_u )_{ u \in [0,T] } $,
$ B = ( U \ni u \mapsto P_I B \hat P_J (u) \in H_\beta ) $,
$ \chi = \chi^{\theta, I, J} $,
$ \O = ( [0,T] \times \Omega \ni (t,\omega) \mapsto 
\O_t^{ \theta, I, J }( \omega ) \in H_{ \max \{ \gamma, \beta \} } ) $,
$ p = \max \{ p, 2 \} $  
for  
$ \theta \in \varpi_T $,
$ I \in \mathcal{P}_0(\H) $,
$ J \in \mathcal{P}( \mathbb{U} ) $  
in the notation of Lemma~\ref{lemma:more_regularity})
and H\"older's inequality
show that~\eqref{eq:Apriori1} holds.
The proof of Corollary~\ref{Corollary:FiniteMoments}
is thus completed.
\end{proof}
\begin{corollary}
	\label{corollary:App_noise_regularity}
	Assume Setting~\ref{setting:Full_discrete}
	and 
	let 
	$ p \in [1, \infty) $, 
	$ \gamma \in [0, \nicefrac{1}{2} + \beta ) $, 
	$ \rho \in [0, \nicefrac{1}{2} + \beta - \gamma) \cap [0, \nicefrac{1}{2} ) $.
	Then it holds for every
		$ s \in [0,T] $, $ t \in [s,T] $ 
	that
	\begin{equation}
	\begin{split}
	\label{eq:DownToGridEstimate}
	&
	\sup\nolimits_{\theta \in \varpi_T}
	\sup\nolimits_{J \in \mathcal{P}( \mathbb{U} )}
	\sup\nolimits_{ I \in \mathcal{P}_0(\H) } 
	\| \O_t^{\theta, I, J} 
	- 
	\O_s^{\theta, I, J}
	\|_{\L^p( \P; H_\gamma)}
	\\
	&
	\leq
	\tfrac{3 \max [ \{ p, 2 \} ]^3
		\| B \|_{\HS(U,H_\beta)}
		[\max \{ T, 1 \} ]^{
			2	
		}
		\max \{
		| \sup\nolimits_{ h \in \H } \values_h |^{- 2 \beta },
		1
		\} 
		(
		1   
		+  
		8
		\| B \|_{\HS(U, H_\beta)}^2 
		)
		\| (-A)^{  	
			\min \{0, \gamma - \beta \} } \|_{L(H)}
	}{ \sqrt{  1 + 2 ( \beta - \max \{ \gamma, \beta \} - \rho ) } } 
	(t-s)^\rho
	.
	\end{split}
	\end{equation}
	
\end{corollary}
\begin{proof}[Proof of Corollary~\ref{corollary:App_noise_regularity}]  
	Note that
	Lemma~\ref{lemma:noise_diff}
	(with
	$ H = H $,
	$ \beta = \beta $, 
	$ \gamma = \max \{ \gamma, \beta \} $, 
	$ T = T $,
	$ \theta = \theta $,
	$ ( \Omega, \F, \P ) = ( \Omega, \F, \P ) $,
	$ ( \f_t )_{ t \in [0,T] } = ( \f_t )_{ t \in [0,T] } $, 
	$ ( W_t )_{ t \in [0,T] } = ( W_t )_{ t \in [0,T] } $, 
	$ B = ( U \ni u \mapsto P_I B \hat P_J (u)  \in H_\beta ) $,
	$ \chi = \chi^{ \theta, I, J } $,
	$ \O = ( [0,T] \times \Omega \ni (t,\omega) 
	\mapsto  \O_t^{\theta, I, J}(\omega) \in H_{ \max \{ \gamma, \beta \} } ) $,
	$ p = \max \{ p, 2 \} $,  
	$ \rho = \rho $
	for
	$ \theta \in \varpi_T $,
	$ I \in \mathcal{P}_0(\H) $,
	$ J \in \mathcal{P}( \mathbb{U} ) $ 
	in the notation of Lemma~\ref{lemma:noise_diff}) 
	establishes~\eqref{eq:DownToGridEstimate}.
	The proof of 
	Corollary~\ref{corollary:App_noise_regularity}
	is thus completed.
\end{proof} 
\begin{corollary}
	\label{corollary:noise_approximation_converges}
	Assume Setting~\ref{setting:Full_discrete},
	let   
	$ p, C \in [1,\infty) $,  
	$ \gamma \in [0, \nicefrac{1}{2} + \beta ) $,
	$ \eta \in [0, \nicefrac{1}{2} + \beta - \gamma ) $,   
    $ \rho \in [0, \nicefrac{1}{2} + \beta - \gamma) \cap [0, \nicefrac{1}{2} ) $, 
	assume for every 
$ s \in [0,T] $,
$ \theta \in \varpi_T $,
$ I \in \mathcal{P}_0(\H) $, 
$ J \in \mathcal{P}( \mathbb{U} ) $
that
$ 	\| \chi_{ \llcorner s \lrcorner_\theta }^{ \theta, I, J }  
-
1 
\|_{\L^{ \max \{p, 2\} }(\P; \R)} \leq C [ | \theta |_T ]^{ \rho } $,
	and
	let $ O \colon [0,T] \times \Omega \to H_\gamma $
	be a stochastic process 
	which satisfies for every $ t \in [0,T] $ that
	$ [ 
	O_t
	]_{\P, \B(H_\gamma)} = \int_0^t e^{ (t-s) A } B \, dW_s $.
	Then it holds 
	for every  
	$ I \in \mathcal{P}_0(\H) $,
	$ J \in \mathcal{P}( \mathbb{U} ) $,
	$ K \in \mathcal{P}(\H) $,
	$ \theta \in \varpi_T $
	with $ I \subseteq K $
	that
	\begin{equation}
	\begin{split}
	& 
	\sup\nolimits_{t\in [0,T]} 
	\| \O_t^{\theta, I, J} -  P_K O_t 
	\|_{\L^p(\P; H_\gamma)}
	\\
	&
	\leq
	\tfrac{  \max \{p, 2\}  
		[ \max \{ T, 1 \} ]^{\nicefrac{1}{2} + \beta } } { \sqrt{2 (1 - 2 \max \{ 0, \gamma + \eta - \beta \} )} } 
	\| 
	(-A)^{ \min \{ 0, \gamma + \eta - \beta \} }  
	\|_{L(H)} 
	\|  B - P_I B \hat P_J \|_{\HS(U, H_{\beta - \eta } )} 
	\\
	&
 	\quad 
	+
	\tfrac{ 8 [\max \{p, 2\}]^3  C  
		[\max\{T, 1\}]^{ \nicefrac{3}{2} + \beta  }   
	}{\sqrt{  1 - 2 \rho - 2 \max \{ 0, \gamma - \beta \}  } 
	}
\| B  \|_{\HS(U, H_{\beta})} 	\| ( -A )^{ \min \{ 0, \gamma - \beta \} } \|_{L(H)}
	\\
	&
	\qquad
	\cdot
	\big(
	1
	+ 
	| \sup\nolimits_{h \in \H} \values_h |^{-2\beta}
	\| B \|_{\HS(U,H_{\beta})}^2
	\big)
	[ |\theta|_T ]^{\rho }  
	.
	\end{split}
	\end{equation}
\end{corollary}
\begin{proof}[Proof of Corollary~\ref{corollary:noise_approximation_converges}]
Note that
Lemma~\ref{lemma:error_exact_counterpart}
(with
$ H = H $,
$ \beta = \beta $, 
$ \gamma = \gamma $,
$ T = T $, 
$ ( \Omega, \F, \P ) 
=
( \Omega, \F, \P ) $,
$ ( \f_t )_{ t \in [0,T] } = ( \f_t )_{ t \in [0,T] } $,
$ ( W_t )_{ t \in [0,T] } = ( W_t )_{ t \in [0,T] } $, 
$ B = ( U \ni u \mapsto P_I B \hat P_J (u) \in H_\beta ) $, 
$ \chi = \chi^{ \theta, I, J } $,
$ \O = 
( [0,T] \times \Omega \ni (t, \omega) 
\mapsto 
\O_t^{\theta, I, J}(\omega )
\in H_{ \gamma } ) $, 
$ C = C $,
$ \tilde B = ( U \ni u \mapsto P_K B(u) \in H_\beta ) $, 
$ p = \max \{ p, 2 \} $, 
$ \eta = \eta $, 
$ \rho = \rho $, 
$ O = ( [0, T] \times \Omega \ni ( t, \omega ) \mapsto  
P_K  O_t (\omega) \in H_{ \gamma  } ) $
for  
$ \theta \in \varpi_T $,
$ I \in \mathcal{P}_0( \H ) $,
$ J \in \mathcal{P}( \mathbb{U} ) $,
$ K \in \mathcal{P}(\H) $ 
with 
$ I \subseteq K $
in the notation of Lemma~\ref{lemma:error_exact_counterpart}) 
ensures 
that for every
$ \theta \in \varpi_T $,
$ I \in \mathcal{P}_0(\H) $,
$ J \in \mathcal{P}( \mathbb{U} ) $,
$ K \in \mathcal{P}(\H) $
with
$ I \subseteq K $
it holds that
	\begin{equation}
\begin{split}
&
	\sup\nolimits_{t\in [0,T]} 
\| \O_t^{\theta, I, J} -  P_K O_t 
\|_{\L^p(\P; H_\gamma)} 
\\
&
\leq
\tfrac{  \max \{p, 2\}  
	[ \max \{ T, 1 \} ]^{\nicefrac{1}{2} + \beta } } { \sqrt{2 (1 - 2 \max \{ 0, \gamma + \eta - \beta \} )} } 
\| 
(-A)^{ \min \{ 0, \gamma + \eta - \beta \} }  
\|_{L(H)} 
	\|  P_K B - P_I B \hat P_J \|_{\HS(U, H_{\beta - \eta } )} 
\\
&
\quad 
+
\tfrac{ 8 [\max \{p, 2\}]^3  C  
	[\max\{T, 1\}]^{ \nicefrac{3}{2} + \beta  }   
}{\sqrt{  1 - 2 \rho - 2 \max \{ 0, \gamma - \beta \}  } 
}
\| P_I B \hat P_J \|_{\HS(U, H_{\beta})}
\| ( -A )^{ \min \{ 0, \gamma - \beta \} } \|_{L(H)}
\\
&
\qquad
\cdot
\big(
1
+ 
| \sup\nolimits_{h \in \H} \values_h |^{-2\beta}
\| P_I B \hat P_J \|_{\HS(U,H_{\beta})}^2
\big)
[ |\theta|_T ]^{\rho }
.
\end{split}
\end{equation}
This completes the proof of Corollary~\ref{corollary:noise_approximation_converges}.
\end{proof}
\begin{corollary}
	\label{corollary:ExpMomentsNoise}
	Assume Setting~\ref{setting:Full_discrete}
	and let $ \varepsilon \in [0, \nicefrac{ 1 }{ (8 [
		\max \{ \| B \|_{ \HS(U,H) }, 1 \} ]^2
		\max \{ T, 1 \} )^2 } ) $.
	Then 
	\begin{equation}
	\begin{split}
	\label{eq:Exp2}
	&
	\sup\nolimits_{\theta \in \varpi_T}
	\sup\nolimits_{J \in \mathcal{P}( \mathbb{U} ) }
	\sup\nolimits_{I \in \mathcal{P}_0(\H) }
	\sup\nolimits_{s \in [0,T]} 
	\E [
	\exp
	(
	\varepsilon
	\|  
	\O_s^{\theta, I, J} 
	\|_H^2   
	)
	] 
	< 
	\infty 
	.
	\end{split}
	\end{equation}
\end{corollary}
\begin{proof}[Proof of Corollary~\ref{corollary:ExpMomentsNoise}]
Note that
	Lemma~\ref{lemma:Exp_regularity}
	(with 
	$ H = H $,
	$ \beta = \beta $,
	$ \gamma = 0 $,
	$ T = T $,
	$ \theta = \theta $,
	$ ( \Omega, \F, \P) = ( \Omega, \F, \P ) $,
	$ ( \f_t )_{ t \in [0,T] } = ( \f_t )_{ t \in [0,T] } $,
	$ B = ( U \ni x \mapsto P_I B \hat P_J(x) \in H_\beta ) $, 
	$ \chi = \chi^{ \theta, I, J } $,
	$ \O = ( [0,T] \times \Omega \ni (t,\omega) \mapsto 
	\O_t^{ \theta, I, J }( \omega ) \in H ) $,
	$ \varepsilon = \varepsilon $
	for  
	$ I \in \mathcal{P}_0(\H) $,
	$ J \in \mathcal{P}( \mathbb{U} ) $,
	$ \theta \in \varpi_T $
	in the notation of
	Lemma~\ref{lemma:Exp_regularity})
	assures that~\eqref{eq:Exp2} holds.
	The proof of Corollary~\ref{corollary:ExpMomentsNoise}
	is thus completed.
\end{proof}
\subsubsection*{Acknowledgements}
This project has been partially supported through the SNSF-Research project $ 200021\_156603 $ ''Numerical 
approximations of nonlinear stochastic ordinary and partial differential equations''.
\newpage
\bibliographystyle{acm}
\bibliography{../Bib/bibfile}

\end{document}